\documentclass[a4paper]{article}

\usepackage[english]{babel}

\usepackage[utf8]{inputenc}
\usepackage[T1]{fontenc}
\usepackage{lmodern}

\usepackage{amssymb,amsmath,amsthm}

\usepackage{todonotes}
\usepackage{enumerate}
\usepackage{mathtools,bbm}
\usepackage[normalem]{ulem}
\usepackage{cancel}
\usepackage{verbatim}

\usepackage{a4}
\usepackage{multirow}
\usepackage[footnotesize,bf,centerlast]{caption}
\usepackage[small,euler-digits,icomma,OT1,T1]{eulervm}

\usepackage{xcolor,colortbl}
\definecolor{Gray}{gray}{0.80}
\definecolor{LightGray}{gray}{0.90}

\newcommand{\cA}{\mathcal{A}}

\newcommand{\cC}{\mathcal{C}}
\newcommand{\cD}{\mathcal{D}}

\newcommand{\cG}{\mathcal{G}}

\newcommand{\cL}{\mathcal{L}}

\newcommand{\cP}{\mathcal{P}}
\newcommand{\cQ}{\mathcal{Q}}

\newcommand{\bG}{\mathbb{G}}

\newcommand{\bN}{\mathbb{N}}

\newcommand{\bR}{\mathbb{R}}

\newcommand{\bfH}{\mathbf{H}}

\newcommand{\PR}{\mathbb{P}}
\newcommand{\bONE}{\mathbbm{1}}

\newcommand{\dd}{ \mathrm{d}}

\DeclareMathOperator{\arccosh}{arccosh}

\DeclareMathOperator*{\LIM}{LIM} 
\DeclareMathOperator*{\subLIM}{subLIM}
\DeclareMathOperator*{\superLIM}{superLIM}

\renewcommand{\epsilon}{\varepsilon}

\newcommand{\vn}[1]{\left| \! \left| #1\right| \! \right|}

\newcommand{\ip}[2]{\langle #1,#2\rangle}

\numberwithin{equation}{section}

\newtheorem{theorem}{Theorem}[section]
\newtheorem{lemma}[theorem]{Lemma}
\newtheorem{proposition}[theorem]{Proposition}
\newtheorem{corollary}[theorem]{Corollary}

\theoremstyle{definition}
\newtheorem{definition}[theorem]{Definition}

\newtheorem{remark}[theorem]{Remark}
\newtheorem*{remark*}{Remark}
\newtheorem{conjecture}[theorem]{Conjecture}
\newtheorem{assumption}[theorem]{Assumption}

\newtheorem{condition}[theorem]{Condition}

\addto\captionsenglish{}

\title{Path-space moderate deviation principles for the random field Curie-Weiss model}

\author{Francesca Collet\thanks{Delft Institute of Applied Mathematics, Delft University of Technology, Mourik van Broekmanweg 6, 2628 XE Delft (The Netherlands). \emph{E-mail address}: f.collet-1@tudelft.nl} \and Richard C. Kraaij\thanks{Fakultät für Mathematik, Ruhr-University of Bochum, Postfach 102148, 44721 Bochum (Germany). \emph{E-mail address}: Richard.Kraaij@rub.de}}

\date{}

\begin{document}

\maketitle

\begin{abstract}
\noindent We analyze the dynamics of moderate fluctuations for macroscopic observables of the random field Curie-Weiss model (i.e., standard Curie-Weiss model embedded in a site-dependent, i.i.d. random environment). We obtain path-space moderate deviation principles via a general analytic approach based on convergence of non-linear generators and uniqueness of viscosity solutions for associated Hamilton–Jacobi equations. The moderate asymptotics depend crucially on the phase we consider and moreover, the space-time scale range for which fluctuations can be proven is restricted by the addition of the disorder. \\

\noindent \emph{Keywords:} moderate deviations $\cdot$  interacting particle systems $\cdot$ mean-field interaction $\cdot$ quenched random environment $\cdot$ Hamilton–Jacobi equation $\cdot$ perturbation theory for Markov processes
\end{abstract}

\addtocontents{toc}{\protect\setcounter{tocdepth}{0}}
\section{Introduction}
\addtocontents{toc}{\protect\setcounter{tocdepth}{3}}

The study of the normalized sum of random variables and its asymptotic behavior plays a central role in probability and statistical mechanics. Whenever the variables are independent and have finite variance, the central limit theorem ensures that the sum with square-root normalization converges to a Gaussian distribution. The generalization of this result to dependent variables is particularly interesting in statistical mechanics where the random variables are correlated through an interaction Hamiltonian. For explicitly solvable models many properties are well understood. In this category fall the so-called Curie-Weiss models for which one can explicitly explain important phenomena such as multiple phases, metastable states and, particularly, how macroscopic observables fluctuate around their mean values when close to or at critical temperatures. Ellis and Newman characterized the distribution of the normalized sum of spins (\emph{empirical magnetization}) for a wide class of mean-field Hamiltonian of Curie-Weiss type \cite{ElNe78a,ElNe78b,ElNeRo80}. They found conditions, in terms of thermodynamic properties, that lead in the infinite volume limit to a Gaussian behavior and those which lead to a higher order exponential probability distribution. Equilibrium large deviation principles have been established in \cite{Ell85}, wheras path-space counterparts have been derived in \cite{Com87}. Static and dynamical moderate deviations have been obtained in \cite{EiLo04,CoKr17} respectively.

We are interested in the fluctuations of the magnetization for the \emph{random field Curie–Weiss model}, which is derived from the standard Curie–Weiss by replacing the constant external magnetic field by local and random fields which interact with each spin of the system. 

The random field Curie-Weiss model has the advantage that, while still being analytically tractable, it has a very rich phase-structure. The phase diagram exhibits interesting critical points: a critical curve where the transition from paramagnetism to ferromagnetism is second-order, a first-order boundary line and moreover, depending on the distribution of the randomness, a tri-critical point may exist \cite{SaWr85}. As a consequence, the model has been used as a playground to test new ideas.

We refer to \cite{AdMPaZa92} for the characterization of infinite volume Gibbs states; \cite{KuLN07} for Gibbs/non-Gibbs transitions; \cite{Kue97,IaKu10,FoKuRe12} for the study of metastates; \cite{MaPi98,FoMaPi00,BiBoIo09} for the metastability analysis; and references therein. From a static viewpoint, the behavior of the fluctuations for this system is clear.  In \cite{AdMPe91}, a central limit theorem is proved and some remarkable new features as compared to the usual non-random model are shown. In particular, depending on temperature, fluctuations may have Gaussian or non-Gaussian limit; in both cases, however, such a limit depends on the realization of the local random external fields, implying that fluctuations are non-self-averaging. Large and moderate deviations with respect to the corresponding (disorder dependent) Gibbs measure have been studied as well. An almost sure large deviation principle can be obtained from \cite{Co89} if the external fields are bounded and from \cite{LoMeTo13} if they are unbounded or dependent. Almost sure moderate deviations are characterized in \cite{LoMe12} under mild assumptions on the randomness.

As already mentioned, all the results recalled so far have been derived at equilibrium; on the contrary, we are interested in describing the time evolution of fluctuations, obtaining non-equilibrium properties. Fluctuations for the random field Curie-Weiss model were studied on the level of a path-space large deviation principle in \cite{DaPdHo96} and on the level of a path-space (standard and non-standard) central limit theorem in \cite{CoDaP12}. The purpose of the present paper is to study dynamical moderate deviations 
of a suitable macroscopic observable. In the random field Curie-Weiss model we are considering, the disorder comes from a site-dependent magnetic field which is $\eta_i = \pm 1$. The single spin-flip dynamics induces a Markovian evolution on a bi-dimensional magnetization. The first component is the usual empirical average of the spin values: $m_n = n^{-1} \sum_{i=1}^n \sigma_i$. The second component is $q_n = n^{-1} \sum_{i=1}^n \sigma_i \eta_i$ and measures the relative alignment between the spins and their local  random fields. The observable we are interested in is therefore the pair $(m_n,q_n)$ and we aim at analyzing its path-space moderate fluctuations.\\
 A moderate deviation principle is technically a large deviation principle and consists in a refinement of a (standard or non-standard) central limit theorem, in the sense that it characterizes the exponential decay of deviations from the average on a smaller scale. We apply the generator convergence approach to  large deviations by Feng-Kurtz \cite{FK06} to characterize the most likely behavior for the trajectories of fluctuations around the stationary solution(s) in the various regimes. Our findings highlight the following distinctive aspects:

\begin{itemize}
\item The moderate asymptotics depend crucially on the phase we are considering. The physical phase transition is reflected at this level via a sudden change in the speed and rate function of the moderate deviation principle. In particular, our findings indicate that fluctuations are Gaussian-like in the sub- and supercritical regimes, while they are not at criticalities. \\
Moreover, if the inverse temperature and the magnetic field intensity are size-dependent and approach a critical threshold, the rate function retains the features of the phases traversed by the sequence of parameters and is a mixture of the rate functions corresponding to the visited regimes. 
\item In the sub- and supercritical regimes, the processes $m_n$ and $q_n$ evolve on the same time-scale and we characterize deviations from the average of the pair $(m_n,q_n)$. For the proof we will refer to the large deviation principle in \cite[Appendix~A]{CoKr17}. On the contrary, at criticality, we have a natural time-scale separation for the evolutions of our processes: $q_n$ is fast and converges exponentially quickly to zero, whereas $m_n$ is slow and its limiting behavior can be determined after suitably ``averaging out'' the dynamics of $q_n$. Corresponding to this observation, we need to prove a path-space large deviation principle for a projected process, in other words for the component $m_n$ only. The projection on a one-dimensional subspace relies on the synergy between the convergence of the Hamiltonians \cite{FK06} and the perturbation theory for Markov processes \cite{PaStVa77}. The method exploits a technique known for (linear) infinitesimal generators in the context of non-linear generators and, to the best of our knowledge, is original. 
Moreover, due to the fact that the perturbed functions we are considering do not allow for a uniform bound for the sequence of Hamiltonians, in the present case we need a more sophisticated notion of convergence of Hamiltonians than the one used in \cite{CoKr17}. To circumvent this unboundedness problem, we relax our definition of limiting operator. More precisely, we follow \cite{FK06} and introduce two Hamiltonians $H_{\dagger}$ and $H_\ddagger$, that are limiting upper and lower bounds for the sequence of Hamiltonians $H_n$, respectively. We then characterize $H$ by matching the upper and lower bound. \\
The same techniques have been recently applied in \cite{CGK} to tackle path-space moderate deviations for a system of interacting particles with {\em unbounded state space}.
\item The fluctuations  are considerably affected by the addition of quenched disorder: the range of space-time scalings for which moderate deviation principles can be proven is restricted by the necessity of controlling the fluctuations of the field. 
\item In \cite{CoDaP12}, at second or higher order criticalities, the contribution to fluctuations coming from the random field is enhanced so as to completely offset the contribution coming from thermal fluctuations. The moderate scaling allows to go beyond this picture and to characterize the thermal fluctuations at the critical line and at the tri-critical point.
\end{itemize}

It is worth to mention that our statements are in agreement with the static results found in \cite{LoMe12}. The paper is organized as follows.

\tableofcontents

\vspace{0.5cm}

Appendix~\ref{appendix:large_deviations_for_projected_processes} is devoted to the derivation of a large deviation principle via solution of Hamilton-Jacobi equation and it is included to make the paper as much self-contained as possible.

\section{Model and main results}\label{sct:CWfi}

\subsection{Notation and definitions}


Before entering the contents of the paper, we introduce some notation. We start with the definition of good rate-function and of large deviation principle for a sequence of random variables. 
	
\begin{definition}
Let $\{X_n\}_{n \geq 1}$ be a sequence of random variables on a Polish space $\mathcal{X}$. Furthermore, consider a function $I : \mathcal{X} \rightarrow [0,\infty]$ and a sequence $\{r_n\}_{n \geq 1}$ of positive numbers such that $r_n \rightarrow \infty$. We say that
\begin{itemize}
\item  
the function $I$ is a \textit{good rate-function} if the set $\{x \, | \, I(x) \leq c\}$ is compact for every $c \geq 0$.
\item 
the sequence $\{X_n\}_{n\geq 1}$ is \textit{exponentially tight} at speed $r_n$ if, for every $a \geq 0$, there exists a compact set $K_a \subseteq \mathcal{X}$ such that $\limsup_n r_n^{-1} \log \, \PR[X_n \notin K_a] \leq - a$.
\item 
the sequence $\{X_n\}_{n\geq 1}$ satisfies the \textit{large deviation principle} with speed $r_n$ and good rate-function $I$, denoted by 
\begin{equation*}
\PR[X_n \approx a] \asymp e^{-r_n I(a)},
\end{equation*}
if, for every closed set $A \subseteq \mathcal{X}$, we have 
\begin{equation*}
\limsup_{n \rightarrow \infty} \, r_n^{-1} \log \PR[X_n \in A] \leq - \inf_{x \in A} I(x),
\end{equation*}
and, for every open set $U \subseteq \mathcal{X}$, 
\begin{equation*}
\liminf_{n \rightarrow \infty} \, r_n^{-1} \log \PR[X_n \in U] \geq - \inf_{x \in U} I(x).
\end{equation*}
\end{itemize}
\end{definition}
	
Throughout the whole paper $\cA\cC$ will denote the set of absolutely continuous curves in $\bR^d$. For the sake of completeness, we recall the definition of absolute continuity.

\begin{definition} 
A curve $\gamma: [0,T] \to \mathbb{R}^d$ is absolutely continuous if there exists a function $g \in L^1([0,T],\bR^d)$ such that for $t \in [0,T]$ we have $\gamma(t) = \gamma(0) + \int_0^t g(s) \dd s$. We write $g = \dot{\gamma}$.\\
A curve $\gamma: \bR^+ \to \mathbb{R}^d$ is absolutely continuous if the restriction to $[0,T]$ is absolutely continuous for every $T \geq 0$. 	
\end{definition}

An important and non-standard definition that we will often use is the notion of $o(1)$ for a sequence of functions.

\begin{definition}\label{def:o(1)}
Let $\{g_n\}_{n \geq 1}$ be a sequence of real functions. We say that 
\[
g_n(x) = g(x) + o(1)
\]
if $\sup_{n \geq 1} \sup_x \vert g_n(x) \vert < \infty$ and $\lim_{n \to \infty} \sup_{x \in K} \vert g_n(x) - g(x) \vert = 0$, for all compact sets $K$.
\end{definition}

To conclude we fix notation for a collection of function-spaces. 
\begin{definition}
Let $k \geq 1$ and $E$ a closed subset of $\mathbb{R}^d$. We will denote by 
\begin{itemize}
\item
$C_l^k(E)$ (resp. $C_u^k(E)$) the set of functions that are bounded from below (resp. above) in $E$ and are $k$ times differentiable on a neighborhood of $E$ in $\mathbb{R}^d$.
\item
$C_c^k(E)$ the set of functions that are constant outside some compact set in $E$ and are $k$ times continuously differentiable on a neighborhood of $E$ in $\mathbb{R}^d$. Finally, we set \mbox{$C_c^\infty(E) := \bigcap_k C_c^k(E)$.}
\end{itemize}
\end{definition}

\subsection{Microscopic and macroscopic description of the model}

Let $\sigma = \left( \sigma_i \right)_{i=1}^n \in \{-1,+1\}^n$ be a configuration of $n$ spins. Moreover, let $\eta = (\eta_i)_{i=1}^n \in \{-1,+1\}^n$ be a sequence of i.i.d. random variables distributed according to $\mu  = \frac{1}{2} \left( \delta_{-1}+\delta_{1} \right) $. \\
For a given realization of $\eta$,~ $\left\{ \sigma(t) \right\}_{t \geq 0}$ evolves as a Markov process on $\{-1,+1\}^n$, with infinitesimal generator
\begin{equation}\label{CWfi:micro:gen}
\mathcal{G}_n f(\varsigma) = \sum_{i=1}^{n} e^{-\beta \varsigma_i ( m_n + B \eta_i)} \left[ f(\varsigma^i) - f(\varsigma) \right],
\end{equation}
where $\varsigma^i$ is the configuration obtained from $\varsigma$ by flipping the $i$-th spin; $\beta$ and $B$ are positive parameters representing the inverse temperature and the coupling strength of the external magnetic field, and $m_n = \frac{1}{n} \sum_{i=1}^n \varsigma_i$.\\
The two terms in the rates of \eqref{CWfi:micro:gen} have different effects: the first one tends to align the spins, while the second one tends to point each of them in the direction of its local field. 

In addition to the usual empirical magnetization, we define also the empirical averages
\[
 q_n (t) := \frac{1}{n} \sum_{i=1}^n \sigma_i (t) \eta_i \quad \mbox{ and } \quad \overline{\eta}_n := \frac{1}{n} \sum_{i=1}^n \eta_i.
\]
Let $E_n$ be the image of $\{-1,1\}^n \times \{-1,1\}^n$ under the map $(\sigma,\eta) \mapsto (m_n,q_n)$.  The Glauber dynamics on the configurations, corresponding to the generator \eqref{CWfi:micro:gen}, induce Markovian dynamics on $E_n$ for the process $\left\{ \left( m_n(t), q_n(t) \right) \right\}_{t \geq 0}$, that in turn evolves with generator
\begin{align}\label{CWfi:micro:gen:m}
 \mathcal{A}_n f (x, y) &= \frac{n(1 + \overline{\eta}_n + x + y)}{4} \, e^{-\beta(x+B)} \left[ f \left( x - \frac{2}{n}, y - \frac{2}{n}\right) - f(x, y)\right] \nonumber \\
&+ \frac{n(1 - \overline{\eta}_n + x - y)}{4} \, e^{-\beta(x-B)} \left[ f \left( x - \frac{2}{n}, y + \frac{2}{n} \right) - f(x, y)\right] \nonumber \\
&+ \frac{n(1 + \overline{\eta}_n - x - y)}{4} \, e^{\beta(x+B)} \left[ f \left( x + \frac{2}{n}, y +\frac{2}{n} \right) - f(x, y)\right] \nonumber \\
&+ \frac{n(1 - \overline{\eta}_n - x + y)}{4} \, e^{\beta(x-B)} \left[ f \left( x + \frac{2}{n}, y - \frac{2}{n} \right) - f(x, y)\right].
\end{align}
For later convenience, let us introduce the functions 
\begin{align}\label{def:G's}
G_{1,\beta,B}^{\pm}(x,y) &= \cosh [\beta (x \pm B)] - (x \pm y) \sinh [\beta (x \pm B)], \nonumber\\[-0.2cm]
&\\[-0.2cm]
G_{2,\beta,B}^{\pm}(x,y) &= \sinh [\beta (x \pm B)] - (x \pm y) \cosh [\beta (x \pm B)]. \nonumber
\end{align}

We start with a large deviation principle for the trajectory of $\left\{ \left( m_n(t), q_n(t) \right) \right\}_{t \geq 0}$. Note that 
\begin{align*}
m_n + q_n & = \frac{1}{n} \sum_i \sigma_i(1 + \eta_i) = \frac{2}{n}\sum_{i: \eta_i = 1} \sigma_i, \\
m_n - q_n & = \frac{1}{n} \sum_{i} \sigma_i(1 -\eta_i) = \frac{2}{n} \sum_{i: \eta_i = -1} \sigma_i,
\end{align*}
which implies that given $\eta$, $(m_n + q_n,m_n - q_n)$ is a pair of 
variables taking their value in discrete subsets of the square $[-1 - \overline{\eta}_n,1+ \overline{\eta}_n]\times [-1 + \overline{\eta}_n,1-\overline{\eta}_n]$. Denote the limiting set by $E_0 := \left\{(x,y) \, \middle| \, (x+y,x-y) \in [-1,1]^2\right\}$.

\begin{proposition}[Large deviations, Theorem 1 in \cite{Kr16b}]
\label{prop:LDP}

Suppose that $(m_n(0),q_n(0))$ satisfies a large deviation principle with speed $n$ on $\bR^2$ with a good rate function $I$ such that $\{(x,y) \, | \, I(x,y) < \infty\} \subseteq E_0$. Then, $\mu$-almost surely, the trajectories $\left\{ \left( m_n(t), q_n(t)\right) \right\}_{t \geq 0}$ satisfy the large deviation principle 
\begin{equation*}
\mathbb{P} \left[ \left\{ \left( m_n(t), q_n(t)\right) \right\}_{t \geq 0} \approx \left\{ \gamma (t) \right\}_{t \geq 0} \right] \asymp e^{-n I(\gamma)} 
\end{equation*}
on $D_{\bR^2}(\bR^+)$, with good rate function $I$ that is finite only for trajectories in $E_0$ and
\begin{equation}\label{rate:fct:LDP}
I(\gamma) = 
\begin{cases}
I_0(\gamma(0)) + \int_0^{+\infty} \cL(\gamma(s),\dot{\gamma}(s))\, \dd s & \text{ if } \gamma \in \mathcal{AC},\\
\infty & \text{ otherwise},
\end{cases}
\end{equation}
where $\cL((x,y),(v_x,v_y)) = \sup_{p \in \bR^2} \left\{ \ip{p}{v} - H((x,y),(p_x,p_y))\right\}$ is the Legendre transform of 
\begin{multline*}
H((x,y),(p_x,p_y)) =  \frac{1}{2} \Big\{ \big[ \cosh (2p_x +2 p_y) -1 \big] G_{1,\beta,B}^{+}(x,y) + \sinh (2p_x +2 p_y) G_{2,\beta,B}^{+}(x,y) \\
 + \big[ \cosh (2p_x - 2 p_y) -1 \big] G_{1,\beta,B}^{-}(x,y) + \sinh (2p_x -2 p_y) G_{2,\beta,B}^{-}(x,y) \Big\}.
\end{multline*}
\end{proposition}

\begin{proof}
Arguing for the pair $(m_n + q_n, m_n - q_n)$, we can use Theorem 1 in \cite{Kr16b}. We obtain our result by undoing the coordinate transformation. 
\end{proof}

We recall that a large deviation principle in the trajectory space can also be derived via contraction of a large deviation principle for the non-interacting particle system; see \cite{DaPdHo96} for details. Moreover, a static quenched large deviation principle for the empirical magnetization has been proved in \cite{LoMeTo13}. In both the aforementioned papers, the large deviation principle is obtained under assumptions that cover more general disorder than dichotomous.\\

The path-space large deviation principle in Proposition~\ref{prop:LDP} allows to derive the infinite volume dynamics for our model: if $(m_n(0), q_n(0))$ converges weakly to the constant $(m_0,q_0)$, then the empirical process $\left( m_n(t), q_n(t) \right)_{t \geq 0}$ converges weakly, as $n \to \infty$, to the solution of
\begin{equation}\label{MKV:randomCW}
\begin{array}{ccl}
\dot{m}(t) & = & G_{2,\beta,B}^{+}(m(t),q(t)) +  G_{2,\beta,B}^{-}(m(t),q(t)) \\[0.3cm]
\dot{q}(t) & = & G_{2,\beta,B}^{+}(m(t),q(t)) -  G_{2,\beta,B}^{-}(m(t),q(t)) 
\end{array}
\end{equation}
with initial condition $(m_0,q_0)$.\\
The phase portrait of system \eqref{MKV:randomCW} is known; for instance, see \cite{AdMPaZa92,DaPdHo95}. We briefly recall the analysis of equilibria. First of all, observe that any stationary solution of \eqref{MKV:randomCW} is of the form
\begin{equation}\label{eqn:stat:magn}
\begin{array}{ccl}
m &=& \frac{1}{2} \left[ \tanh ( \beta ( m + B ) ) + \tanh ( \beta ( m - B ) ) \right]  \\[0.3cm]
q &=& \frac{1}{2} \left[ \tanh ( \beta ( m + B ) ) - \tanh ( \beta ( m - B ) ) \right]
\end{array} 
\end{equation}
and that $(0,\tanh (\beta B))$ satisfies \eqref{eqn:stat:magn} for all the values of the parameters. Solutions with $m=0$ are called \emph{paramagnetic}, those with $m \neq 0$ \emph{ferromagnetic}. On the phase space $(\beta, B)$ we get the following:
\begin{enumerate}[(I)]
\item 
If $\beta \leq 1$, then $(0,\tanh(\beta B))$ is the unique fixed point for \eqref{MKV:randomCW} and it is globally stable.
\item 
If $\beta > 1$, the situation is more subtle. There exist two functions
\[
g_1(\beta) = \frac{1}{\beta} \arccosh (\sqrt{\beta}) 
\]
and
\[
g_2: [1,+\infty) \to \left[ 0,1\right), \text{ strictly increasing, $g(1)=0$, $g(\beta_n) \uparrow 1$ as $\beta_n \uparrow +\infty$},
\]
satisfying 
\begin{itemize}
\item
$g_1(\beta) \leq g_2(\beta)$ on $[1,+\infty)$, \\
\item
$g_1(\beta)$ and $g_2(\beta)$ coincide for $\beta \in \left[1, \frac{3}{2} \right]$ and separate at the tri-critical point $(\beta_{\mathrm{tc}}, B_{\mathrm{tc}}) = ( \frac{3}{2}, \frac{2}{3} \arccosh ( \sqrt{\frac{3}{2}}))$,
\end{itemize}
such that
\begin{enumerate}[(i)]
\item
if $B \geq g_2(\beta)$ the same result as in $(I)$ holds;
\item
if $B < g_1(\beta)$, then $(0, \tanh(\beta B))$ becomes unstable and two (symmetric) stable ferromagnetic solutions arise;
\item
if $\beta > \frac{3}{2}$ and $B = g_1(\beta)$, then $(0, \tanh(\beta B))$ is neutrally stable and coexists with a pair of stable ferromagnetic solutions;
\item
if $\beta > \frac{3}{2}$ and $g_1(\beta) < B < g_2(\beta)$, then $(0, \tanh(\beta B))$ is stable and, in addition, we have two pairs (one is stable and the other is not) of ferromagnetic solutions. Inside this phase there is a coexistence line, above which the paramagnetic solution is stable and the two stable ferromagnetic solutions are metastable, and below which the reverse is true.
\end{enumerate}
\end{enumerate}
We refer to Figure~\ref{fig:PhD} for a visualization of the previous assertions.

\begin{figure}
\centering
\includegraphics[scale=.8]{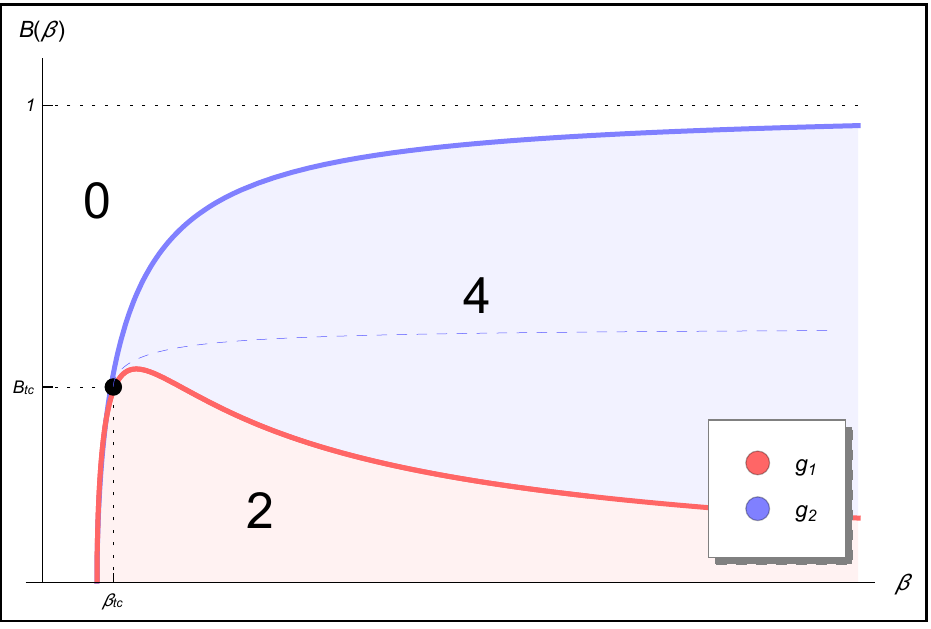}
\caption{Qualitative picture of the phase space $(\beta, B)$ for equation \eqref{MKV:randomCW}. Each colored region represents a phase with as many ferromagnetic solutions of \eqref{eqn:stat:magn} as indicated by the numerical label. The thick red and blue separation lines are $g_1$ and $g_2$ respectively. The thin dashed blue  line is the coexistence line relevant for metastability (cf. II(iv)).}
\label{fig:PhD}
\end{figure}

%

\subsection{Main results}

%
We consider the moderate deviations of the microscopic dynamics \eqref{CWfi:micro:gen:m} around their stationary macroscopic  limit in the various regimes. \\
The first of our statements is mainly of interest in the paramagnetic phase, but is indeed valid for all values of the parameters.

\begin{theorem}[Moderate deviations around $(0,\tanh(\beta B))$]
\label{thm:MD:paramagnetic:subcritical}
Let $\{b_n\}_{n\geq 1}$ be a sequence of positive real numbers such that 
\[
b_n \to \infty \quad \mbox{ and } \quad b_n^2 n^{-1} \log \log n \to 0.
\] 
Suppose that $\left( b_n m_n(0), b_n (q_n(0)-\tanh(\beta B)) \right)$ satisfies a large deviation principle with speed $nb_n^{-2}$ on $\mathbb{R}^2$ and rate function $I_0$. Then, $\mu$-almost surely, the trajectories 
\[
\left\{ \left( b_n m_n(t), b_n (q_n(t) - \tanh (\beta B)) \right) \right\}_{t \geq 0}
\] 
satisfy the large deviation principle 
\[
\mathbb{P} \left[ \left\{ \left( b_n m_n(t), b_n (q_n(t) - \tanh (\beta B)) \right) \right\}_{t \geq 0} \approx \left\{ \gamma (t) \right\}_{t \geq 0} \right] \asymp e^{-nb_n^{-2} I(\gamma)} 
\]
on $D_{\mathbb{R} \times \mathbb{R}}(\mathbb{R}^+)$, with good rate function
\begin{equation}\label{rate:fct:subcritical}
I(\gamma) = 
\left\{
\begin{array}{ll}
I_0(\gamma(0)) + \int_0^{+\infty} \, \mathcal{L} (\gamma(s), \dot{\gamma}(s)) \, \dd s & \text{ if } \gamma \in \mathcal{AC},\\ 
\infty & \text{ otherwise},
\end{array}
\right.
\end{equation}
where
\begin{equation}\label{eqn:Lagrangian:paramagnetic_MD:subcritical}
\cL(\mathbf{x}, \mathbf{v}) := \frac{\cosh(\beta B)}{8} \left\vert \mathbf{v} - 2 \begin{pmatrix} \frac{\beta - \cosh^2(\beta B)}{\cosh(\beta B)} & 0 \\ 0 & -\cosh(\beta B) \end{pmatrix} \mathbf{x}\right\vert^2.
\end{equation}
\end{theorem}

Observe that the growth condition $b_n^2 n^{-1} \log \log n \to 0$ is necessary to ensure that $b_n \overline{\eta}_n$ (re-scaled empirical average of the local fields) converges to zero almost surely as $n \to +\infty$. A similar effect is also known in moderate deviation principles for the overlap in the Hopfield model, see \cite{EiLo04}. The peculiar scaling is prescribed by the \emph{law of iterated logarithm}, that provides the scaling factor where the limits of the weak and strong law of large numbers become different, cf. \cite[Corollary 14.8]{Ka02}. Analogous requirements will appear also in the following statements.\\

Our next result considers moderate deviations around ferromagnetic solutions of \eqref{eqn:stat:magn}. To shorten notation and not to clutter the statement, let us introduce the following matrices
\begin{equation}\label{def:G_matrix}
\bG_{1,\beta,B}(x,y) = \begin{pmatrix}  G_{1,\beta,B}^+(x,y)  + G_{1,\beta,B}^-(x,y)  & G_{1,\beta,B}^+(x,y)  - G_{1,\beta,B}^-(x,y)  \\ G_{1,\beta,B}^+(x,y)  - G_{1,\beta,B}^-(x,y)  & G_{1,\beta,B}^+(x,y)  + G_{1,\beta,B}^-(x,y)  \end{pmatrix},
\end{equation}
\begin{equation}\label{def:Ghat_matrix}
\hat{\bG}_{1,\beta,B} (x,y)=
\begin{pmatrix}
G_{1,\beta,B}^+(x,y) + G_{1,\beta,B}^-(x,y) & 0 \\ 
G_{1,\beta,B}^+(x,y) - G_{1,\beta,B}^-(x,y) & 0  
\end{pmatrix}
\end{equation}
and
\begin{equation}\label{def:B_matrix}
\mathbb{B}(x) = 
\begin{pmatrix}\
\cosh(\beta x) \cosh(\beta B) & \sinh(\beta x) \sinh(\beta B) \\ 
\sinh(\beta x) \sinh(\beta B) & \cosh(\beta x) \cosh(\beta B) 
\end{pmatrix}.
\end{equation}
We get the following.

\begin{theorem}[Moderate deviations: super-critical regime $\beta > 1$, $B < g_2(\beta)$]
\label{thm:MD:ferromagnetic:supercritical}
Let $(m,q)$ be a solution of \eqref{eqn:stat:magn} with $m \neq 0$. Moreover, let $\{b_n\}_{n\geq 1}$ be a sequence of positive real numbers such that 
\[
b_n \to \infty \quad \mbox{ and } \quad b_n^2 n^{-1} \log \log n  \to 0.
\] 
Suppose that $\left( b_n (m_n(0)-m), b_n (q_n(0)-q) \right)$ satisfies a large deviation principle with speed $nb_n^{-2}$ on $\mathbb{R}^2$ and rate function $I_0$. Then, $\mu$-almost surely, the trajectories 
\[
\left\{ \left( b_n (m_n(t)-m), b_n (q_n(t) - q) \right) \right\}_{t \geq 0}
\] 
satisfy the large deviation principle 
\[
\mathbb{P} \left[ \left\{ \left( b_n (m_n(t)-m), b_n (q_n(t) - q) \right) \right\}_{t \geq 0} \approx \left\{ \gamma (t) \right\}_{t \geq 0} \right] \asymp e^{-nb_n^{-2} I(\gamma)} 
\]
on $D_{\mathbb{R} \times \mathbb{R}}(\mathbb{R}^+)$, with good rate function
\begin{equation}\label{rate:fct:supercritical}
I(\gamma) = 
\left\{
\begin{array}{ll}
I_0(\gamma(0)) + \int_0^{+\infty} \, \mathcal{L} (\gamma(s), \dot{\gamma}(s)) \, \dd s & \text{ if } \gamma \in \mathcal{AC},\\ 
\infty & \text{ otherwise},
\end{array}
\right.
\end{equation}
where 
\[
\cL(\mathbf{x}, \mathbf{v}) := \frac{1}{4}\ip{\bG_{1,\beta,B}^{-1}(m,q) [\mathbf{v} - (\beta\hat{\bG}_{1,\beta,B}(m,q) -2\mathbb{B}(m)) \mathbf{x}]}{\mathbf{v} - (\beta\hat{\bG}_{1,\beta,B}(m,q) -2\mathbb{B}(m)) \mathbf{x}}.
\]
\end{theorem}


We see that the Lagrangian \eqref{eqn:Lagrangian:paramagnetic_MD:subcritical} trivializes in the $x$ coordinate if $\beta = \cosh^2(\beta B)$. The latter equation corresponds to $(\beta,B)$ lying on the critical curve $B = g_1(\beta)$. This fact can be seen as the dynamical counterpart of the bifurcation occurring at the stationary point as $B$ varies for fixed $\beta$: $(0,\tanh(\beta B))$ is turning unstable from being a stable equilibrium. 

\smallskip

At the critical curve, the fluctuations of $m_n(t)$ behave homogeneously in the distance from the stationary point, whereas the fluctuations of $q_n(t)$ are confined around $0$. To further study the fluctuations of $m_n(t)$, we speed up time to capture higher order effects of the microscopic dynamics. Speeding up time implies that the probability of deviations from $q_n(t)$ decays faster than exponentially. 

\begin{theorem}[Moderate deviations: critical curve $1 < \beta \leq \frac{3}{2}$, $B=g_1(\beta)$]
\label{thm:MD:paramagnetic:critical}
Let $\{b_n\}_{n\geq 1}$ be a sequence of positive real numbers such that 
\[
b_n \to \infty \quad \mbox{ and } \quad b_n^6 n^{-1} \log \log n \to 0.
\] 
Suppose that $b_n m_n(0)$ satisfies a large deviation principle with speed $nb_n^{-4}$ on $\mathbb{R}$ and rate function $I_0$. Then, $\mu$-almost surely, the trajectories $\left\{ b_n m_n(b_n^2 t) \right\}_{t \geq 0}$ satisfy the large deviation principle 
\[
\mathbb{P} \left[ \left\{b_n m_n(b_n^2t)\right\}_{t \geq 0} \approx \left\{ \gamma (t) \right\}_{t \geq 0} \right] \asymp e^{-nb_n^{-4} I(\gamma)} 
\]
on $D_{\mathbb{R}}(\mathbb{R}^+)$, with good rate function
\begin{equation}\label{rate:fct:critical}
I(\gamma) = 
\left\{
\begin{array}{ll}
I_0(\gamma(0)) + \int_0^{+\infty} \, \mathcal{L} (\gamma(s), \dot{\gamma}(s)) \, \dd s & \text{ if } \gamma \in \mathcal{AC},\\ 
\infty & \text{ otherwise},
\end{array}
\right.
\end{equation}
where 
\[
\mathcal{L}(x,v) = \frac{\cosh(\beta B)}{8} \left\vert v - \frac{2}{3} \beta (2\beta-3)\cosh(\beta B) x^3 \right\vert^2.
\]
\end{theorem}

At the tri-critical point, again the Lagrangian trivializes, and a further speed-up of time is possible.

\begin{theorem}[Moderate deviations: tri-critical point $\beta = \frac{3}{2}$ and $B = g_1(\frac{3}{2})$]
\label{thm:MD:paramagnetic:tricritical}
Let $\{b_n\}_{n\geq 1}$ be a sequence of positive real numbers such that 
\[
b_n \to \infty \quad \mbox{ and } \quad b_n^{10} n^{-1} \log \log n \to 0.
\] 
Suppose that $b_n m_n(0)$ satisfies a large deviation principle with speed $nb_n^{-6}$ on $\mathbb{R}$ and rate function $I_0$. Then, $\mu$-almost surely, the trajectories $\left\{ b_n m_n(b_n^4 t) \right\}_{t \geq 0}$ satisfy the large deviation principle 
\[
\mathbb{P} \left[ \left\{b_n m_n(b_n^4t)\right\}_{t \geq 0} \approx \left\{ \gamma (t) \right\}_{t \geq 0} \right] \asymp e^{-nb_n^{-6} I(\gamma)} 
\]
on $D_{\mathbb{R}}(\mathbb{R}^+)$, with good rate function
\begin{equation}\label{rate:fct:tricritical}
I(\gamma) = 
\left\{
\begin{array}{ll}
I_0(\gamma(0)) + \int_0^{+\infty} \, \mathcal{L} (\gamma(s), \dot{\gamma}(s)) \, \dd s & \text{ if } \gamma \in \mathcal{AC},\\ 
\infty & \text{ otherwise},
\end{array}
\right.
\end{equation}
where 
\[
\mathcal{L}(x,v) = \frac{1}{8} \sqrt{\frac{3}{2}} \left\vert v  + \frac{9}{10} \sqrt{\frac{3}{2}} \, x^5 \right\vert^2.
\]
\end{theorem}

We want to conclude the analysis by considering moderate deviations for volume-dependent temperature and magnetic field approaching the critical curve first and the tri-critical point afterwards. In the sequel let $\{m_n^{\beta, B}(t)\}_{t \geq 0}$ denote the process evolving at temperature $\beta$ and subject to a random field of strength $B$.

\begin{theorem}[Moderate deviations: critical curve $1 < \beta \leq \frac{3}{2}$, $B = g_1(\beta)$, temperature and field rescaling] \label{theorem:moderate_deviations_CW_critical_curve_rescaling}

Let $\{b_n\}_{n\geq 1}$ be a sequence of positive real numbers such that 
\[
b_n \to \infty \quad \mbox{ and } \quad b_n^6 n^{-1} \log \log n \to 0.
\]
Let $\{\kappa_n\}_{n \geq 1}$, $\{\theta_n\}_{n \geq 1}$ be two sequences of real numbers such that 
\[
\kappa_n b_n^{2} \rightarrow \kappa \quad \mbox{ and } \quad \theta_n b_n^{2} \rightarrow \theta.
\]
Set $\beta_n := \beta + \kappa_n$ and $B_n := B + \theta_n$, where $B=g_1(\beta)$, with $1 < \beta \leq \frac{3}{2}$.  Suppose that $b_n m_n^{\beta_n, B_n}(0)$ satisfies the large deviation principle with speed $n b_n^{-4}$ on $\bR$ with rate function $I_0$. Then, $\mu$-almost surely, the trajectories $\left\{b_n m_n^{\beta_n, B_n}(b_n^{2} t) \right\}_{t \geq 0}$ satisfy the  large deviation principle on $D_\bR(\bR^+)$:
\begin{equation*}
\PR\left[\left\{b_n m_n^{\beta_n, B_n}(b_n^2t)\right\}_{t \geq 0} \approx \{\gamma(t)\}_{t \geq 0}  \right] \asymp e^{-n b_n^{-4}I(\gamma)},
\end{equation*}
where $I$ is the good rate function
\begin{equation}\label{CW:criticalMD:temp_resc:RF}
I(\gamma) = \begin{cases}
I_0(\gamma(0)) + \int_0^\infty \mathcal{L}(\gamma(s),\dot{\gamma}(s))\dd s & \text{if } \gamma \in \cA\cC, \\
\infty & \text{otherwise},
\end{cases}
\end{equation}
and 
\begin{multline*}
\mathcal{L}(x,v) = \frac{\cosh(\beta B)}{8} \left\vert v -2 \left[ \frac{1- 2 \beta B \tanh (\beta B)}{\cosh(\beta B)} \, \kappa - 2 \beta \sinh (\beta B) \, \theta\right] x \right.\\
\left.  - \frac{2}{3} \beta (2\beta-3)\cosh(\beta B) x^3 \right\vert^2.
\end{multline*}
\end{theorem}

For approximations of the tri-critical point, we consider two scenarios. The first considers an approximation along the critical curve, whereas the second scenario considers approximation from an arbitrary direction.

\begin{theorem}[Moderate deviations: tri-critical point $\beta=\frac{3}{2}$, $B = g_1(\frac{3}{2})$, temperature and field rescaling on the critical curve] \label{theorem:moderate_deviations_CW_tri-critical_point_rescaling}

Let $\{b_n\}_{n\geq 1}$ be a sequence of positive real numbers such that 
\[
b_n \to \infty \quad \mbox{ and } \quad b_n^{10} n^{-1} \log \log n \to 0.
\] 
Let $\{\kappa_n\}_{n \geq 1}$, $\{\theta_n\}_{n \geq 1}$ be two sequences of real numbers such that 
\[
\kappa_n b_n^{2} \rightarrow \kappa \quad \mbox{ and } \quad \theta_n b_n^{2} \rightarrow \theta.
\]
Set $\beta_n := \beta_{\mathrm{tc}} + \kappa_n$ and $B_n := B_{\mathrm{tc}} + \theta_n$, where $(\beta_{\mathrm{tc}}, B_{\mathrm{tc}}) = (\frac{3}{2}, \frac{2}{3} \arccosh (\sqrt{\frac{3}{2}}))$. Assume that $\beta_n = \cosh^2 (\beta_n B_n)$ for all $n \in \mathbb{N}$. Moreover, suppose that $b_n m_n^{\beta_n, B_n}(0)$ satisfies the large deviation principle with speed $n b_n^{-6}$ on $\bR$ with rate function $I_0$. Then, $\mu$-almost surely, the trajectories $\left\{b_n m_n^{\beta_n, B_n}(b_n^{4} t) \right\}_{t \geq 0}$ satisfy the  large deviation principle on $D_\bR(\bR^+)$:
\begin{equation*}
\PR\left[\left\{b_n m_n^{\beta_n, B_n}(b_n^4t)\right\}_{t \geq 0} \approx \{\gamma(t)\}_{t \geq 0}  \right] \asymp e^{-n b_n^{-6}I(\gamma)},
\end{equation*}
where $I$ is the good rate function
\begin{equation}\label{CW:tri_criticalMD:temp_resc:RF}
I(\gamma) = \begin{cases}
I_0(\gamma(0)) + \int_0^\infty \mathcal{L}(\gamma(s),\dot{\gamma}(s))\dd s & \text{if } \gamma \in \cA\cC, \\
\infty & \text{otherwise},
\end{cases}
\end{equation}
and
\begin{equation*}
\mathcal{L}(x,v) = \frac{1}{8} \sqrt{\frac{3}{2}} \left\vert v  - \left[2  \sqrt{2} \arccosh \left(\sqrt{\frac{3}{2}}\right) \kappa +\frac{9}{\sqrt{2}} \, \theta \right] x^3+ \frac{9}{10} \sqrt{\frac{3}{2}} \, x^5 \right\vert^2.
\end{equation*}

\end{theorem}

\begin{remark*}
To ensure that $(\beta_n,B_n)$ approximate $(\beta_{\mathrm{tc}},B_{\mathrm{tc}})$ over the critical curve, $\kappa$ and $\theta$ must satisfy
\begin{equation}\label{eqn:theta_kappa_relationship}
\frac{\theta}{\kappa} = - \frac{4}{9} \arccosh\left(\sqrt{\frac{3}{2}}\right) + \frac{2}{3} \sqrt{\frac{1}{3}}.
\end{equation}
\end{remark*}

\begin{theorem}[Moderate deviations: tri-critical point $\beta=\frac{3}{2}$, $B = g_1(\frac{3}{2})$, temperature and field rescaling] \label{theorem:moderate_deviations_CW_tri-critical_point_rescaling:2}

Let $\{b_n\}_{n\geq 1}$ be a sequence of positive real numbers such that 
\[
b_n \to \infty \quad \mbox{ and } \quad b_n^{10} n^{-1} \log \log n \to 0.
\] 
Let $\{\kappa_n\}_{n \geq 1}$, $\{\theta_n\}_{n \geq 1}$ be two sequences of real numbers such that 
\[
\kappa_n b_n^{4} \rightarrow \kappa \quad \mbox{ and } \quad \theta_n b_n^{4} \rightarrow \theta.
\]
Set $\beta_n := \beta_{\mathrm{tc}} + \kappa_n$ and $B_n := B_{\mathrm{tc}} + \theta_n$, where $(\beta_{\mathrm{tc}}, B_{\mathrm{tc}}) = (\frac{3}{2}, \frac{2}{3} \arccosh (\sqrt{\frac{3}{2}}))$. Suppose that $b_n m_n^{\beta_n, B_n}(0)$ satisfies the large deviation principle with speed $n b_n^{-6}$ on $\bR$ with rate function $I_0$. Then, $\mu$-almost surely, the trajectories $\left\{b_n m_n^{\beta_n, B_n}(b_n^{4} t) \right\}_{t \geq 0}$ satisfy the  large deviation principle on $D_\bR(\bR^+)$:
\begin{equation*}
\PR\left[\left\{b_n m_n^{\beta_n, B_n}(b_n^4t)\right\}_{t \geq 0} \approx \{\gamma(t)\}_{t \geq 0}  \right] \asymp e^{-n b_n^{-6}I(\gamma)},
\end{equation*}
where $I$ is the good rate function
\begin{equation}\label{CW:tri_criticalMD:temp_resc:RF:2}
I(\gamma) = \begin{cases}
I_0(\gamma(0)) + \int_0^\infty \mathcal{L}(\gamma(s),\dot{\gamma}(s))\dd s & \text{if } \gamma \in \cA\cC, \\
\infty & \text{otherwise},
\end{cases}
\end{equation}
and
\begin{equation*}
\mathcal{L}(x,v) = \frac{1}{8} \sqrt{\frac{3}{2}} \left\vert v  - \left[ \frac{2}{3} \left( \sqrt{6} - 2 \sqrt{2} \arccosh \left( \sqrt{\frac{3}{2}} \right)\right)   \kappa - 3 \sqrt{2} \, \theta \right] x + \frac{9}{10} \sqrt{\frac{3}{2}} \, x^5 \right\vert^2.
\end{equation*}
\end{theorem}

By choosing the sequence $b_n=n^{\alpha}$, with $\alpha > 0$, we can rephrase Theorems~\ref{thm:MD:paramagnetic:subcritical}--\ref{theorem:moderate_deviations_CW_tri-critical_point_rescaling:2} in terms of more familiar moderate scalings involving powers of the system-size. We therefore get estimates for the probability of a typical trajectory on a scale that is between a law of large numbers and a central limit theorem. These results, together with the central limit theorem and the study of fluctuations at $\beta = \cosh^2(\beta B)$ in \cite[Prop.~2.7 and Thm.~2.12]{CoDaP12}, give a clear picture of the behaviour of fluctuations for the random field Curie-Weiss model. We summarize these facts in Tables~\ref{tab:RFCW:deviations}~and~\ref{tab:RFCW:deviations:critical}. The displayed conclusions are drawn under the assumption that in each case either the initial condition satisfies a large deviation principle at the correct speed or the initial measure converges weakly.
Observe that \emph{not all} scales can be covered. Indeed, to control disorder fluctuations and avoid they are too large, the range of allowed spatial scalings becomes quite limited. \\ 
To conclude, it is worth to mention that our results are consistent with the moderate deviation principles obtained in \cite{LoMe12} for the random field Curie-Weiss model at equilibrium. Indeed, as prescribed by Thm.~5.4.3 in \cite{FW98}, in each of the cases above, the rate function of the stationary measure satisfies $H(x,S'(x))=0$, where $H: \mathbb{R} \times \mathbb{R} \to \mathbb{R}$ is the Legendre transform of $\mathcal{L}$. 

\begin{table}[h!]
\centering
\begin{tabular}{|c|c|c|c|}
\hline
\parbox[c][1cm][c]{1.5cm}{\scriptsize \centering \bf SCALING EXPONENT} & \parbox[c][1cm][c]{1.5cm}{\scriptsize \centering \bf REGIME} & \parbox[c][1cm][c]{4cm}{ \scriptsize \centering \bf RESCALED PROCESS} & \parbox[c][1cm][c]{3cm}{\scriptsize \centering \bf LIMITING THEOREM}\\
\hline \hline
$\alpha = 0$ & \parbox[c][1cm][c]{1.5cm}{\centering \footnotesize all $\beta$\\ all $B$} & {\footnotesize $\left(m_n(t),q_n(t)\right)$} & \parbox[c][1.3cm][c]{3cm}{\scriptsize \centering LDP at speed $n$ with rate function \eqref{rate:fct:LDP}}\\
\hline
\multirow{3}{*}{$\alpha \in \left( 0, \frac{1}{2} \right)$} & \parbox[c][1cm][c]{1.5cm}{\centering \footnotesize all $\beta$ \\ all $B$} & {\footnotesize $\left( n^\alpha m_n(t), n^{\alpha} (q_n(t)-\tanh(\beta B))\right)$} & \parbox[c][1.5cm][c]{3cm}{\scriptsize \centering LDP at speed $n^{1-2\alpha}$ with rate function \eqref{rate:fct:subcritical}}\\
 & \parbox[c][1cm][c]{1.5cm}{\centering \footnotesize $\beta > 1$\\ $B < g_2(\beta)$} & {\footnotesize $\left( n^\alpha (m_n(t) - m), n^{\alpha} (q_n(t)-q)\right)$} &  \parbox[c][1.5cm][c]{3cm}{\scriptsize \centering LDP at speed $n^{1-2\alpha}$ with rate function \eqref{rate:fct:supercritical}} \\
 \hline
 \multirow{4}{*}{$\alpha = \frac{1}{2}$} & \parbox[c][1cm][c]{1.5cm}{\centering \footnotesize all $\beta$\\ $B > g_1(\beta)$} & {\footnotesize $\left( n^{1/2} m_n(t), n^{1/2}(q_n(t)-\tanh(\beta B))\right)$} & \multirow{4}{*}{\parbox{3cm}{\scriptsize \centering CLT \\ weak convergence to the unique solution of a linear diffusion equation \\ (see \cite[Prop.~2.7]{CoDaP12})}} \\
 &&&\\
 & \parbox[c][1cm][c]{1.5cm}{\centering \footnotesize $\beta > 1$\\ $B < g_2(\beta)$} & {\footnotesize $\left( n^{1/2} (m_n(t)-m), n^{1/2} (q_n(t)-q)\right)$} &  \\ 
 \hline 
 \end{tabular}
\caption{Non-critical fluctuations for the order parameter of the random field \mbox{Curie-Weiss} spin-flip dynamics}
 \label{tab:RFCW:deviations}
 \end{table}

\begin{table}[h!]
\centering
\begin{tabular}{|c|c|c|c|}
\hline
\parbox[c][1cm][c]{1.5cm}{\bf \scriptsize \centering SCALING EXPONENT} & \parbox[c][1cm][c]{4.5cm}{\bf \scriptsize \centering REGIME} & \parbox[c][1cm][c]{2cm}{\bf \scriptsize \centering RESCALED PROCESS} & \parbox[c][1cm][c]{3cm}{\bf \scriptsize \centering LIMITING THEOREM}\\
\hline \hline
$\alpha = \frac{1}{4}$ & \parbox[c][1cm][c]{4.5cm}{\centering \footnotesize $\beta > 1$, $B=g_1(\beta)$} & {\footnotesize $n^{1/4} m_n\left( n^{1/4}t \right)$} & \parbox[c][2.6cm][c]{3cm}{\scriptsize \centering weak convergence to the process\\[0.1cm] $Y(t) = 2 X t$\\[0.1cm] with \\[0.1cm] $X \sim N(0,\sinh^2 (\beta B))$ \\[0.1cm] (see \cite[Thm.~2.12]{CoDaP12})}\\
\hline
\multirow{3}{*}{\parbox[c][2.5cm][c]{1.7cm}{$\alpha \in \left( 0, \frac{1}{6} \right)$}} & \parbox[c][0.5cm][c]{4.5cm}{\centering \footnotesize $1 < \beta \leq \frac{3}{2}$, $B = g_1(\beta)$} & {\footnotesize $n^\alpha m_n \left( n^{2\alpha}t \right)$} & \parbox[c][1.5cm][c]{3cm}{\scriptsize \centering LDP at speed $n^{1-4\alpha}$ with rate function \eqref{rate:fct:critical}}\\
&&&\\
& \parbox[c][2.5cm][c]{4.5cm}{\centering \footnotesize $\beta_n = \beta + \kappa_n$, $B_n = B + \theta_n$\\ where\\[0.1cm] $B = g_1(\beta)$, $1 < \beta \leq \frac{3}{2}$ \\[0.1cm] $\kappa_n n^{2\alpha} \to \kappa$, $\theta_n n^{2\alpha} \to \theta$} & {\footnotesize $n^\alpha m_n \left( n^{2\alpha}t \right)$} & \parbox[c][2.5cm][c]{3cm}{\scriptsize \centering LDP at speed $n^{1-4\alpha}$ with rate function \eqref{CW:criticalMD:temp_resc:RF}}\\
\hline
\multirow{5}{*}{\parbox[c][6cm][c]{1.7cm}{\centering $\alpha \in \left( 0, \frac{1}{10} \right)$}} & \parbox[c][0.5cm][c]{4.5cm}{\centering \footnotesize $\beta = \frac{3}{2}$, $B = g_1(\beta)$} & {\footnotesize $n^\alpha m_n \left( n^{4\alpha}t \right)$} &  \parbox[c][1.5cm][c]{3cm}{\scriptsize \centering LDP at speed $n^{1-6\alpha}$ with rate function \eqref{rate:fct:tricritical}} \\
&&&\\
& \parbox[c][2.5cm][c]{4.5cm}{\centering \footnotesize $\beta_n = \frac{3}{2} + \kappa_n$, $B_n = g_1(\frac{3}{2}) + \theta_n$\\ where \\[0.1cm]  $\beta_n = \cosh^2(\beta_n B_n)$, $\forall n \in \mathbb{N}$ \\[0.1cm] $\kappa_n n^{2\alpha} \to \kappa$, $\theta_n n^{2\alpha} \to \theta$} & {\footnotesize $n^\alpha m_n \left( n^{4\alpha}t \right)$} & \parbox[c][2.5cm][c]{3cm}{\scriptsize \centering LDP at speed $n^{1-6\alpha}$ with rate function \eqref{CW:tri_criticalMD:temp_resc:RF}}\\
&&&\\
& \parbox[c][2.5cm][c]{4.5cm}{\centering \footnotesize $\beta_n = \frac{3}{2} + \kappa_n$, $B_n = g_1(\frac{3}{2}) + \theta_n$\\ where \\[0.1cm]  $\kappa_n n^{4\alpha} \to \kappa$, $\theta_n n^{4\alpha} \to \theta$} & {\footnotesize $n^\alpha m_n \left( n^{4\alpha}t \right)$} & \parbox[c][2.5cm][c]{3cm}{\scriptsize \centering LDP at speed $n^{1-6\alpha}$ with rate function \eqref{CW:tri_criticalMD:temp_resc:RF:2}}\\
 \hline 
 \end{tabular}
 \caption{Critical fluctuations for the order parameter of the random field Curie-Weiss spin-flip dynamics}
 \label{tab:RFCW:deviations:critical}
 \end{table}
 %


\section{Expansion of the Hamiltonian and moderate deviations in the sub- and supercritical regimes}
\label{sct:proofs}


Following the methods of \cite{FK06}, the authors have studied large and moderate deviations for the Curie-Weiss model based on the convergence of Hamiltonians and well-posedness of a class of Hamilton-Jacobi equations corresponding to a limiting Hamiltonian in \cite{Kr16b,CoKr17}. For the results in Theorems~\ref{thm:MD:paramagnetic:subcritical} and \ref{thm:MD:ferromagnetic:supercritical}, considering moderate deviations for the pair $(m_n(t),q_n(t))$, we will follow a similar strategy and we will refer to the large deviation principle in \cite[Appendix~A]{CoKr17}. For the results in Theorems~\ref{thm:MD:paramagnetic:critical}--\ref{theorem:moderate_deviations_CW_tri-critical_point_rescaling:2} stated for the process $m_n(t)$ only, we need a more sophisticated large deviation result, which is based on the abstract framework introduced in \cite{FK06}. We will recall the notions needed for these results in Appendix \ref{appendix:large_deviations_for_projected_processes}.

In both settings, however, a main ingredient is the convergence of Hamiltonians. Therefore, we start by deriving an expansion for the Hamiltonian associated to a \emph{generic} time-space scaling of the fluctuation process. We will then use such an expansion to obtain the results stated in Theorems~\ref{thm:MD:paramagnetic:subcritical} and \ref{thm:MD:ferromagnetic:supercritical}. For Theorems \ref{thm:MD:paramagnetic:critical}--\ref{theorem:moderate_deviations_CW_tri-critical_point_rescaling:2}, we need additional methods to obtain a limiting Hamiltonian, that will be introduced in Sections \ref{section:projection_fixed_betaB} and \ref{section:projection_varying_betaB} below. 



\subsection{Expansion of the  Hamiltonian}\label{subsct:expansion:Hamiltonian} 

Let $(m,q)$ be a stationary solution of equation \eqref{MKV:randomCW}. We consider the fluctuation process $\big(b_n \left(m_n(b_{n}^{\nu} t) - m\right), b_n \left(q_n(b_{n}^{\nu} t) - q \right) \big)$. Its generator $A_n$ can be deduced from \eqref{CWfi:micro:gen:m} and is given by
\begin{allowdisplaybreaks}
\begin{align*}
A_n f(x,y) & =  \frac{b_{n}^{\nu} n}{4} \, \left[ 1 + \overline{\eta}_n + m + q  + (x+y)b_{n}^{-1} \right] e^{-\beta(xb_{n}^{-1} +m+B)} \times\\
& \hspace{3cm} \times \left[f\left(x - 2 b_n n^{-1}, y - 2 b_n n^{-1}\right) - f(x,y) \right] \\
& +  \frac{b_{n}^{\nu} n}{4} \, \left[ 1 + \overline{\eta}_n - m - q - (x+y)b_{n}^{-1} \right] e^{\beta(xb_{n}^{-1} +m+B)} \times\\
& \hspace{3cm} \times \left[f\left(x + 2 b_n n^{-1}, y + 2 b_n n^{-1}\right) - f(x,y) \right] \\
& +  \frac{b_{n}^{\nu} n}{4} \, \left[ 1- \overline{\eta}_n+ m - q  + (x-y)b_{n}^{-1} \right] e^{-\beta(xb_{n}^{-1} +m -B)} \times\\
& \hspace{3cm} \times \left[f\left(x - 2 b_n n^{-1}, y + 2 b_n n^{-1}\right) - f(x,y) \right] \\
& +  \frac{b_{n}^{\nu} n}{4} \, \left[ 1- \overline{\eta}_n- m + q - (x-y)b_{n}^{-1} \right]e^{\beta(xb_{n}^{-1} +m -B)} \times\\
& \hspace{3cm} \times \left[f\left(x + 2 b_n n^{-1}, y - 2 b_n n^{-1}\right) - f(x,y) \right].
\end{align*}
\end{allowdisplaybreaks}
Therefore, at speed $n b_n^{-\delta}$, the Hamiltonian
\begin{equation}\label{eqn:generic:Hamiltonian}
H_n f = b_n^{\delta} n^{-1} e^{-n b_n^{-\delta} f} A_n e^{n b_n^{-\delta} f}
\end{equation}
results in
\begin{allowdisplaybreaks}
\begin{align*}
H_n f(x,y) & =  \frac{b_{n}^{\nu+\delta}}{4} \, \left[ 1 + \overline{\eta}_n + m + q  + (x+y)b_{n}^{-1} \right] e^{-\beta(xb_{n}^{-1} +m+B)} \times\\
& \hspace{3cm} \times \left[ e^{n b_n^{-\delta} \left[f\left(x - 2 b_n n^{-1}, y - 2b_n n^{-1} \right) - f(x,y) \right]} - 1 \right] \\
& +  \frac{b_{n}^{\nu+\delta}}{4} \, \left[ 1 + \overline{\eta}_n - m - q - (x+y)b_{n}^{-1} \right] e^{\beta(xb_{n}^{-1} +m+B)} \times\\
& \hspace{3cm} \times \left[  e^{n b_n^{-\delta} \left[f\left(x + 2 b_n n^{-1}, y + 2b_n n^{-1} \right) - f(x,y) \right]} - 1 \right] \\
& +  \frac{b_{n}^{\nu+\delta}}{4} \, \left[ 1- \overline{\eta}_n+ m - q  + (x-y)b_{n}^{-1} \right] e^{-\beta(xb_{n}^{-1} +m-B)} \times\\
& \hspace{3cm} \times \left[  e^{n b_n^{-\delta} \left[f\left(x - 2 b_n n^{-1}, y + 2b_n n^{-1} \right) - f(x,y) \right]} - 1 \right] \\
& +  \frac{b_{n}^{\nu+\delta}}{4} \, \left[ 1- \overline{\eta}_n- m + q - (x-y)b_{n}^{-1} \right]e^{\beta(xb_{n}^{-1} +m-B)} \times\\
& \hspace{3cm} \times \left[ e^{n b_n^{-\delta} \left[f\left(x + 2 b_n n^{-1}, y - 2b_n n^{-1} \right) - f(x,y) \right]} - 1\right].
\end{align*}
\end{allowdisplaybreaks}

Let $\nabla f(x,y) = (f_x(x,y), f_y(x,y))^{\intercal}$ be the gradient of $f$. Moreover, denote
\[
\bONE_{\pm} = \begin{pmatrix} 1 & \pm 1 \\ \pm 1 & 1 \end{pmatrix}
\quad \mbox{ and } \quad
\mathbf{e}_{\pm} = \begin{pmatrix} 1 \\ \pm 1 \end{pmatrix}.
\]
We Taylor expand the exponential functions containing $f$ up to second order:
\begin{align*}
& \exp \left\{n b_n^{-\delta} \left[f\left(x \pm 2 b_n n^{-1}, y \pm 2b_n n^{-1} \right) - f(x,y) \right] \right\} - 1 \\
& \quad = \pm 2b_{n}^{-\delta+1} \ip{\mathbf{e}_+}{\nabla f(x,y)} + 2b_{n}^{-2\delta+2} \ip{\bONE_+ \nabla f(x,y)}{\nabla f(x,y)} + o \left( b_{n}^{-2\delta+2} \right)
\end{align*}
and
\begin{align*}
& \exp \left\{n b_n^{-\delta} \left[f\left(x \pm 2 b_n n^{-1}, y \mp 2b_n n^{-1} \right) - f(x,y) \right] \right\} - 1 \\
& \quad = \pm 2b_{n}^{-\delta+1} \ip{\mathbf{e}_-}{\nabla f(x,y)} + 2b_{n}^{-2\delta+2} \ip{\bONE_- \nabla f(x,y)}{\nabla f(x,y)} + o \left( b_{n}^{-2\delta+2} \right).
\end{align*}

To write down the intermediate result after Taylor-expansion, we introduce the functions 
\begin{equation}\label{def:Gn's}
\begin{aligned}
G_{n,1,\beta,B}^{\pm} (x,y) &= (1 \pm \overline{\eta}_n) \cosh [\beta (x \pm B)] - (x \pm y) \sinh [\beta(x \pm B)]\\[0.2cm]
G_{n,2,\beta,B}^{\pm} (x,y) &= (1 \pm \overline{\eta}_n) \sinh [\beta (x \pm B)] - (x \pm y) \cosh [\beta(x \pm B)]
\end{aligned}
\end{equation}
and the matrix 
\begin{equation*}
\bG_{n,1,\beta,B}(x,y) = \begin{pmatrix}  G_{n,1,\beta,B}^+(x,y)  + G_{n,1,\beta,B}^-(x,y)  & G_{n,1,\beta,B}^+(x,y)  - G_{n,1,\beta,B}^-(x,y)  \\ G_{n,1,\beta,B}^+(x,y)  - G_{n,1,\beta,B}^-(x,y)  & G_{n,1,\beta,B}^+(x,y)  + G_{n,1,\beta,B}^-(x,y)  \end{pmatrix},
\end{equation*}
which are the finite-volume analogues of \eqref{def:G's} and \eqref{def:G_matrix}. In what follows, not to clutter our formulas we will drop subscripts highlighting the dependence on the inverse temperature $\beta$ and the magnetic field $B$. A tedious but straightforward calculation yields
\begin{align*}
H_n f(x,y) & = b_n^{\nu+1} \ip{\begin{pmatrix} G_{n,2}^{+} \left(xb_{n}^{-1} + m,yb_{n}^{-1} + q \right) + G_{n,2}^{-} \left( xb_{n}^{-1} + m,yb_{n}^{-1} + q \right) \\[0.1cm] G_{n,2}^{+} \left( xb_{n}^{-1} + m,yb_{n}^{-1} + q \right) - G_{n,2}^{-} \left( xb_{n}^{-1} + m,yb_{n}^{-1} + q \right)\end{pmatrix}}{\nabla f(x,y)} \\
&+ b_{n}^{\nu -\delta +2} \, \ip{\bG_{n,1}(xb_n^{-1} +m,yb_n^{-1} + q) \nabla f(x,y)}{\nabla f(x,y)} \\
&+  o \left( b_{n}^{\nu -\delta +2} \right).
\end{align*}
To have an interesting reminder in the limit, we need $\delta = \nu + 2$.  This gives	
\begin{multline}
\!\!\!\!\!\! H_n f(x,y) = b_n^{\nu+1} \ip{\begin{pmatrix} G_{n,2}^{+} \left(xb_{n}^{-1} + m,yb_{n}^{-1} + q \right) + G_{n,2}^{-} \left( xb_{n}^{-1} + m,yb_{n}^{-1} + q \right) \\[0.1cm] G_{n,2}^{+} \left( xb_{n}^{-1} + m,yb_{n}^{-1} + q \right) - G_{n,2}^{-} \left( xb_{n}^{-1} + m,yb_{n}^{-1} + q \right)\end{pmatrix}}{\nabla f(x,y)} \\[0.2cm]
+ \ip{\bG_{n,1}(xb_n^{-1} +m,yb_n^{-1} + q) \nabla f(x,y)}{\nabla f(x,y)} +  o \left( 1 \right). \label{short:Ham}
\end{multline}

In the sequel we will Taylor expand $G_{n,1}^{\pm}$ and $G_{n,2}^{\pm}$ around $(m,q)$. Therefore we need the derivatives of the $G$'s. By direct computation, we get the following lemma.

\begin{lemma}\label{lmm:derivatives}
Let $G_{n,1}^{\pm}$ and $G_{n,2}^{\pm}$ be defined by \eqref{def:Gn's}. Then, we obtain
\begin{multline*}
\frac{\partial^k}{\partial x^k}\left( G_{n,2}^{+} +  G_{n,2}^{-} \right)(x,y) \\
= \begin{cases}
\beta^k\left(G_{2}^{+} +  G_{2}^{-}\right)(x,y) - 2k\beta^{k-1} \sinh(\beta x) \cosh (\beta B) &\\
\hspace{3.5cm} + 2\beta^{k} \, \overline{\eta}_n \cosh(\beta x) \sinh (\beta B) & \text{if $k$ is even}, \\[0.3cm]
\beta^k\left(G_{1}^{+} +  G_{1}^{-}\right)(x,y) - 2 k\beta^{k-1} \cosh(\beta x) \cosh (\beta B) &\\
\hspace{3.5cm} + 2 \beta^{k} \, \overline{\eta}_n \sinh(\beta x) \sinh (\beta B) & \text{if $k$ is odd},
\end{cases}
\end{multline*}

	\begin{multline*}
	\frac{\partial^k}{\partial x^k} \left( G_{n,2}^{+} -  G_{n,2}^{-} \right)(x,y) \\
	= \begin{cases}
	\beta^k\left(G_{2}^{+} - G_{2}^{-}\right)(x,y) - 2k\beta^{k-1} \cosh(\beta x) \sinh (\beta B) &\\
	\hspace{3.5cm} + 2\beta^{k} \, \overline{\eta}_n \sinh(\beta x) \cosh (\beta B) & \text{if $k$ is even}, \\[0.3cm]
	\beta^k\left(G_{1}^{+} - G_{1}^{-}\right)(x,y) - 2k\beta^{k-1} \sinh(\beta x) \sinh (\beta B) &\\
	\hspace{3.5cm}  + 2\beta^{k} \, \overline{\eta}_n \cosh(\beta x) \cosh (\beta B) & \text{if $k$ is odd},
	\end{cases}
	\end{multline*}

\[
\frac{\partial}{\partial y} G_{n,1}^{\pm}(x,y) = \mp \sinh(\beta(x \pm B)), \quad \frac{\partial}{\partial y} G_{n,2}^{\pm}(x,y) = \mp \cosh(\beta(x \pm B)), 
\]
\[
\frac{\partial^{k+1}}{\partial x^k \partial y}\left( G_{n,2}^{+} +  G_{n,2}^{-} \right)(x,y) 
= \begin{cases}
-2 \beta^k \sinh(\beta x) \sinh (\beta B) & \text{if $k$ is even}, \\[0.3cm]
-2 \beta^k \cosh(\beta x) \sinh (\beta B) & \text{if $k$ is odd},
\end{cases}
\]
and
\[
\frac{\partial^{k+1}}{\partial x^k \partial y} \left( G_{n,2}^{+} -  G_{n,2}^{-} \right)(x,y) 
= \begin{cases}
-2 \beta^k \cosh(\beta x) \cosh (\beta B) & \text{if $k$ is even}, \\[0.3cm]
-2 \beta^k \sinh(\beta x) \cosh (\beta B) & \text{if $k$ is odd}.
\end{cases}
\]
\end{lemma}

Note that the functions $G$, whenever differentiated twice in the $y$ direction, equal zero. For the sake of readability, we again put the terms of the Taylor expansions of $G_{n,1}^\pm$ and $G_{n,2}^\pm$ in matrix form. For $k \in \mathbb{N}$, $k \geq 1$, let us denote
\begin{align*}
D^k \cG_{n,2}(x,y) & := \begin{pmatrix} \frac{\partial^k}{\partial x^k} (G_{n,2}^+ + G_{n,2}^-)(x,y) & \frac{\partial^k}{\partial x^{k-1} \partial y} (G_{n,2}^+ + G_{n,2}^-)(x,y) \\
\frac{\partial^k}{\partial x^k} (G_{n,2}^+ - G_{n,2}^-)(x,y) & \frac{\partial^k}{\partial x^{k-1} \partial y} (G_{n,2}^+ - G_{n,2}^-)(x,y)\end{pmatrix} \\
\intertext{and}
D^k \cG_{2}(x,y) & := \begin{pmatrix} \frac{\partial^k}{\partial x^k} (G_{2}^+ + G_{2}^-)(x,y) & \frac{\partial^k}{\partial x^{k-1} \partial y} (G_{2}^+ + G_{2}^-)(x,y) \\
\frac{\partial^k}{\partial x^k} (G_{2}^+ - G_{2}^-)(x,y) & \frac{\partial^k}{\partial x^{k-1} \partial y} (G_{2}^+ - G_{2}^-)(x,y)\end{pmatrix},
\end{align*}
where $G_{n,2}^{\pm}$ are defined in \eqref{def:Gn's} and $G_{2}^{\pm}$ in \eqref{def:G's}. Moreover, set
\[
N_k  := \begin{pmatrix} \frac{\partial^k}{\partial x^k} \cosh(\beta x) \sinh(\beta B) & 0 \\
\frac{\partial^k}{\partial x^k} \sinh(\beta x) \cosh(\beta B) & 0\end{pmatrix}.
\]
By Lemma~\ref{lmm:derivatives}, it follows that $D^k \cG_{n,2}(x,y) = D^k \cG_2(x,y) + 2\overline{\eta}_n N^k$,  for all $k \geq 1$. We obtain the following expansion.

\begin{lemma} \label{lemma:expansion_of_H_half_way_form}
For $f \in C^3_c(\mathbb{R}^2)$, we have 
\begin{align}
H_n f(x,y) & = b_n^{\nu+1} \ip{\begin{pmatrix}G_{2}^+(m,q) + G_{2}^-(m,q) \\ G_{2}^+(m,q) - G_{2}^-(m,q) \end{pmatrix}}{\nabla f(x,y)} \label{eqn:prelim_compact_exp_G_1} \\
& + 2 b_n^{\nu+1} \overline{\eta}_n \ip{\begin{pmatrix}\cosh(\beta m) \sinh(\beta B) \\ \sinh(\beta m) \cosh(\beta B) \end{pmatrix}}{\nabla f(x,y)} \label{eqn:prelim_compact_exp_G_2}  \\
& + \sum_{k=1}^5 \frac{b_n^{\nu+ 1-k}}{k!}\ip{D^k \cG_{2}(m,q) \begin{pmatrix} x^k \\ k x^{k-1}y \end{pmatrix}}{\nabla f(x,y)} \label{eqn:prelim_compact_exp_G_3}  \\ 
& + 2 \overline{\eta}_n  b_n^{\nu+1} \sum_{k=1}^5 \frac{b_n^{-k}}{k!}\ip{N_k(m,q) \begin{pmatrix} x^k \\ k x^{k-1}y \end{pmatrix}}{\nabla f(x,y)} + o(b_n^{\nu -4}) \label{eqn:prelim_compact_exp_G_4}  \\
& + \ip{\bG_{1}(m,q) \nabla f(x,y)}{\nabla f(x,y)} \label{eqn:prelim_compact_exp_G_5}\\
& +  o \left( 1 \right) \label{eqn:prelim_compact_exp_G_6},
\end{align}
where the remainder terms converge to zero uniformly on compact sets.
\end{lemma}


\begin{proof}
Consider \eqref{short:Ham}. We Taylor expand up to fifth order the terms involving $G_{n,2}^{\pm}$. This yields
\begin{align*}
& \ip{\begin{pmatrix} G_{n,2}^{+} \left(xb_{n}^{-1} + m,yb_{n}^{-1} + q \right) + G_{n,2}^{-} \left( xb_{n}^{-1} + m,yb_{n}^{-1} + q \right) \\[0.1cm] G_{n,2}^{+} \left( xb_{n}^{-1} + m,yb_{n}^{-1} + q \right) - G_{n,2}^{-} \left( xb_{n}^{-1} + m,yb_{n}^{-1} + q \right)\end{pmatrix}}{\nabla f(x,y)} \notag \\
& = \ip{\begin{pmatrix}G_{n,2}^+(m,q) + G_{n,2}^-(m,q) \\ G_{n,2}^+(m,q) - G_{n,2}^-(m,q) \end{pmatrix}}{\nabla f(x,y)}  \\
& \qquad + \sum_{k=1}^5 \frac{b_n^{-k}}{k!}\ip{D^k \cG_{n,2}(m,q) \begin{pmatrix} x^k \\ k x^{k-1}y \end{pmatrix}}{\nabla f(x,y)} + o(b_n^{-5})  \\
& =  \ip{\begin{pmatrix}G_{2}^+(m,q) + G_{2}^-(m,q) \\ G_{2}^+(m,q) - G_{2}^-(m,q) \end{pmatrix}}{\nabla f(x,y)} + 2 \overline{\eta}_n \ip{\begin{pmatrix}\cosh(\beta m) \sinh(\beta B) \\ \sinh(\beta m) \cosh(\beta B) \end{pmatrix}}{\nabla f(x,y)}   \\
& \qquad + \sum_{k=1}^5 \frac{b_n^{-k}}{k!}\ip{D^k \cG_{2}(m,q) \begin{pmatrix} x^k \\ k x^{k-1}y \end{pmatrix}}{\nabla f(x,y)} \notag \\ 
& \qquad + 2 \overline{\eta}_n \sum_{k=1}^5 \frac{b_n^{-k}}{k!}\ip{N_k(m,q) \begin{pmatrix} x^k \\ k x^{k-1}y \end{pmatrix}}{\nabla f(x,y)} + o(b_n^{-5}).   
\end{align*}
Multiplying by $b_n^{\nu+1}$ we obtain \eqref{eqn:prelim_compact_exp_G_1}-\eqref{eqn:prelim_compact_exp_G_4}. Finally, an expansion of the $\bG_{n,1}$ matrix shows that only the zero-th order term remains, giving \eqref{eqn:prelim_compact_exp_G_5}.
\end{proof}

Observe that $o(1)+o(b_n^{\nu-4})$ (cf. lines \eqref{eqn:prelim_compact_exp_G_4} and \eqref{eqn:prelim_compact_exp_G_6}) includes all the remainder terms coming from first Taylor expanding the exponentials, and then the functions $G_1^{\pm},G_2^{\pm}$. For any $f \in C_c^3(\mathbb{R}^2)$, let us denote by $R_{n,f}^{\mathrm{exp}}$ and $R_{n,f}^G$ these two contributions.  In what follows we will need a more accurate control on these remainders.  For this reason we state the following lemma.  
\begin{lemma}\label{lemma:remainder_control}
Let $f \in C_c^3(\bR^2)$ and let $\nu \in \{0,2,4\}$. Set $K_{n,0} = [-\log^{1/2} b_n,\log^{1/2} b_n ]^2$. There exists a positive constant $C$ (dependent on the sup-norms of the first to third order partial derivatives of $f$, but not on $n$) such that we have
 \begin{equation}\label{eqn:remainder_control}
 \sup_{(x,y) \in K_{n,0}} \left\vert R_{n,f}^{\mathrm{exp}}(x,y) + R_{n,f}^G(x,y) \right\vert \leq C \left( n^{-1} b_n^{\nu+2} + b_n^{\nu-5} \log^{3} b_n \right).
 \end{equation}
\end{lemma}  
\begin{proof}
We study the Taylor expansion of the exponential functions first. We treat explicitly only the case of 
\[
\frac{b_n^{2\nu+2}}{4}  \left[ \exp \left\{nb_n^{-(\nu+2)} [f(x+2b_n n^{-1},y+2b_n n^{-1}) - f(x,y)] \right\}-1\right],
\]
the others being analogous. We denote by $R_{n,f}^{\mathrm{exp},+}$ the remainder terms coming from Taylor expanding such a function. To shorten our next formula, we set $\mathbf{x} = (x,y)^{\intercal}$ and $\boldsymbol{\xi}=(\xi_1,\xi_2)^{\intercal}$. By Lagrange's form of the Taylor expansion, there is some $\boldsymbol{\xi} \in \mathbb{R}^2$ with $\xi_1 \in \, (x,x+2b_n n^{-1})$ and $\xi_2 \in \, (y, y+2b_n n^{-1})$ and


%
\begin{multline*}
R_{n,f}^{\mathrm{exp},+}(\mathbf{x}) = \frac{n^{-1} b_n^{\nu+2}}{2} \, \ip{D^2 f(\mathbf{x}) \mathbf{e}_+}{\mathbf{e}_+} \\
+ \exp \left\{ nb_n^{-(\nu+2)} [f(\boldsymbol{\xi}) - f(\mathbf{x})] \right\} \bigg\{ 2b_n^{-(\nu+1)} \sum_{j_1+j_2=3} \frac{(\partial_1f (\boldsymbol{\xi}))^{j_1}(\partial_2 f (\boldsymbol{\xi}))^{j_2}}{j_1! j_2!} \\
+ b_n n^{-1} \ip{\nabla f(\boldsymbol{\xi})}{\mathbf{e}_+} \ip{D^2 f(\boldsymbol{\xi}) \mathbf{e}_+}{\mathbf{e}_+} + 2 b_n^{\nu+3} n^{-2} \sum_{j_1+j_2=3} \frac{\partial_1^{j_1} \partial_2^{j_2} f (\boldsymbol{\xi})}{j_1! j_2!} \bigg\}.
\end{multline*}
Observe that, by the mean-value theorem, we can control the exponential. Indeed, there exists a point $\mathbf{z}$, on the line-segment connecting $\boldsymbol{\xi}$ and $\mathbf{x}$, for which we have \mbox{$f(\boldsymbol{\xi}) - f(\mathbf{x}) = \ip{\nabla f(\mathbf{z})}{\boldsymbol{\xi} - \mathbf{x}}$}. Since $\boldsymbol{\xi} - \mathbf{x} \in \,  (0, 2 b_n n^{-1})^2$, we can estimate $|f(\boldsymbol{\xi}) - f(\mathbf{x})| \leq 4 (\|\partial_1 f \| \vee \|\partial_2 f \|) b_n n^{-1}$ and, in turn, we get 
\begin{align*}
\exp \left\{ nb_n^{-(\nu+2)} [f(\boldsymbol{\xi}) - f(\mathbf{x})] \right\} &\leq \exp \left\{ 4 b_n^{-(\nu+1)} (\|\partial_1 f \| \vee \|\partial_2 f \|) \right\} \\
&\leq \exp \left\{ 4 (\|\partial_1 f \| \vee \|\partial_2 f \|)\right\}.
\end{align*}
Therefore, we can find positive constants $c_1$ and $c_2$ (depending on sup-norms of the first, second and third order partial derivatives of $f$, but not on $n$), such that 
\[
\sup_{(x,y) \in \mathbb{R}^2} \left\vert R_{n,f}^{\mathrm{exp},+}(x,y) \right\vert \leq c_1 \,  n^{-1} b_n^{\nu+2} + c_2 \, b_n^{-(\nu+1)}.
\]
Analogously, we get the same control for the other three exponential terms. We conclude
\[
\sup_{(x,y) \in \mathbb{R}^2} \left\vert R_{n,f}^{\mathrm{exp}}(x,y) \right\vert \leq 4 \left[c_1 \,  n^{-1} b_n^{\nu+2} + c_2 \, b_n^{-(\nu+1)} \right].
\]
We focus now on the remainder terms relative to the expansion of the $G$'s function. We have
\[
R_{n,f}^G(x,y) = \frac{b_n^{\nu-5}}{6!} \ip{D^6 \cG_{2}(\zeta_1,\zeta_2) \begin{pmatrix} x^6 \\ 6 x^{5-1}y \end{pmatrix}}{\nabla f(x,y)},
\] 
with $\zeta_1 \in \, (m,m+xb_n^{-1})$ and $\zeta_2 \in \, (q,q+yb_n^{-1})$. We easily derive the following bound
\[
\sup_{(x,y) \in K_{0,n}} \left\vert R_{n,f}^G(x,y) \right\vert \leq c_3 \, b_n^{\nu-5} \log^{3} b_n,
\]
where $c_3=c_3(\|\partial_1 f\|,\|\partial_2 f\|)$ is a suitable positive constant, independent of $n$. The conclusion then follows. 
\end{proof}

Turning back to the expansion of $H_n$ in Lemma~\ref{lemma:expansion_of_H_half_way_form}, we analyze now the terms containing $\overline{\eta}_n$ that appear in \eqref{eqn:prelim_compact_exp_G_2} and \eqref{eqn:prelim_compact_exp_G_4}. As $\cosh$ is a positive function and $b_n \rightarrow \infty$, any contribution by $\overline{\eta}_n$ is dominated by the one in \eqref{eqn:prelim_compact_exp_G_2}. To make sure that this term vanishes, we apply the Law of Iterated Logarithm.
 
\begin{theorem}[Law of Iterated Logarithm, {\cite[Corollary 14.8]{Ka02}}] \label{theorem:LIL}
We have 
\begin{equation*}
\limsup_{n \rightarrow \infty} \frac{\overline{\eta}_n \sqrt{n}}{\sqrt{\log \log n}} = \sqrt{2}
\quad \text{ and } \quad \liminf_{n \rightarrow \infty} \frac{\overline{\eta}_n \sqrt{n}}{\sqrt{\log \log n}} = -\sqrt{2} \quad \text{ $\mu$-a.s.}.
\end{equation*}
\end{theorem}

As an immediate corollary, we obtain conditions ensuring that $b_n^{\nu+1} \overline{\eta}_n$ converges to zero almost surely.  

\begin{corollary} 
Let $\nu \in \bN$. If $\{b_n\}_{n \geq 1}$ is a sequence such that
\begin{equation}\label{eqn:growth:condition}
b_n^{2\nu + 2} \, n^{-1} \, \log \log n \rightarrow 0,
\end{equation}
then $b_n^{\nu+1} \overline{\eta}_n \rightarrow 0$ $\mu$-almost surely.
\end{corollary}

Note that condition \eqref{eqn:growth:condition} corresponds to the growth assumption in Theorems \ref{thm:MD:paramagnetic:subcritical} and \ref{thm:MD:ferromagnetic:supercritical} for $\nu = 0$, in Theorems \ref{thm:MD:paramagnetic:critical} and \ref{theorem:moderate_deviations_CW_critical_curve_rescaling} for $\nu = 2$ and in Theorems \ref{thm:MD:paramagnetic:tricritical}, \ref{theorem:moderate_deviations_CW_tri-critical_point_rescaling} and \ref{theorem:moderate_deviations_CW_tri-critical_point_rescaling:2} for $\nu = 4$. The result of Lemma \ref{lemma:expansion_of_H_half_way_form}, combined with the corollary, yields a preliminary expansion for
\[
H_n f = b_n^{\nu+2}n^{-1} e^{-nb_n^{-\nu-2}f}A_n e^{nb_n^{-\nu-2}f},
\]
which is obtained from the generic Hamiltonian \eqref{eqn:generic:Hamiltonian} after the choice $\delta = \nu +2$ we made to get a non-trivial expansion with controlled remainder.

\begin{proposition} \label{proposition:compact_expression_for_H_n}
Let $f \in C^3_c(\bR^2)$ and $\nu \in \mathbb{N}$. Moreover, let $\{b_n\}_{n \geq 1}$ be a sequence such that 
\begin{equation*}
b_n \rightarrow \infty, \qquad b_n^{2\nu + 2} \, n^{-1} \, \log \log n \rightarrow 0.
\end{equation*}
Then, $\mu$-almost surely, we have
\begin{align}
H_n f(x,y) & = b_n^{\nu + 1} \ip{\begin{pmatrix}G_{2}^+(m,q) + G_{2}^-(m,q) \\ G_{2}^+(m,q) - G_{2}^-(m,q) \end{pmatrix}}{\nabla f(x,y)} + o (1)+ o(b_n^{\nu - 4}) \label{eqn:compact_exp_G_1}  \\
& + \sum_{k=1}^5 \frac{b_n^{\nu + 1 -k}}{k!}\ip{D^k \cG_{2}(m,q) \begin{pmatrix} x^k \\ k x^{k-1}y \end{pmatrix}}{\nabla f(x,y)}  \label{eqn:compact_exp_G_2} \\
& + \ip{\bG_{1}(m,q) \nabla f(x,y)}{\nabla f(x,y)}.  \label{eqn:compact_exp_G_3}
\end{align}
\end{proposition}

In the setting of our main theorems $\nu \in \{0,2,4\}$ and $(m,q)$ is a stationary point. This implies that all contributions on the right hand side of \eqref{eqn:compact_exp_G_1} vanish almost surely and uniformly on compact sets as $n \rightarrow \infty$. Furthermore, the expression in \eqref{eqn:compact_exp_G_3} is constant and we do not need to consider this expression any further.\\
Thus, the analysis for our main results focuses on the expressions in \eqref{eqn:compact_exp_G_2}. The next lemma gives expressions for the matrices $D^k \cG_2(m,q)$.

\begin{lemma} \label{lemma:DkG2_expressions}
Let $k \in \mathbb{N}$, $k \geq 1$. 
\begin{enumerate}[(a)]
\item 
If $(m,q)$ is a generic point, then
\begin{equation*}
D^k \cG_2(m,q) = 
\begin{cases}
\beta^k \begin{pmatrix} G_2^+(m,q) + G_2^-(m,q) & 0 \\ G_2^+(m,q) - G_2^-(m,q) & 0 \end{pmatrix} \\
	-2 \beta^{k-1} \begin{pmatrix} k \sinh(\beta m) \cosh(\beta B) & \cosh(\beta m) \sinh(\beta B) \\ k \cosh(\beta m)\sinh(\beta B) & \sinh(\beta m) \cosh(\beta B) \end{pmatrix} & \text{if $k$ is even},\\[0.7cm]
	\beta^k \begin{pmatrix} G_1^+(m,q) + G_1^-(m,q) & 0 \\ G_1^+(m,q) - G_1^-(m,q) & 0 \end{pmatrix} \\
	 -2\beta^{k-1} \begin{pmatrix}k \cosh(\beta m)\cosh(\beta B) & \sinh(\beta m) \sinh(\beta B) \\ k \sinh(\beta m) \sinh(\beta B) & \cosh(\beta m) \cosh(\beta B) \end{pmatrix} & \text{if $k$ is odd}.
	\end{cases}
	\end{equation*}	
\item 
If $(m,q)$ is a stationary point, then
	\begin{equation*}
	D^k \cG_2(m,q) = \begin{cases}
	-2 \beta^{k-1} \begin{pmatrix} k \sinh(\beta m) \cosh(\beta B) & \cosh(\beta m) \sinh(\beta B) \\ k \cosh(\beta m)\sinh(\beta B) & \sinh(\beta m) \cosh(\beta B) \end{pmatrix} & \text{if $k$ is even},\\[0.7cm]
	\beta^k \begin{pmatrix} G_1^+(m,q) + G_1^-(m,q) & 0 \\ G_1^+(m,q) - G_1^-(m,q) & 0 \end{pmatrix} \\ 
	-2\beta^{k-1} \begin{pmatrix}k \cosh(\beta m)\cosh(\beta B) & \sinh(\beta m) \sinh(\beta B) \\ k \sinh(\beta m) \sinh(\beta B) & \cosh(\beta m) \cosh(\beta B) \end{pmatrix} & \text{if $k$ is odd}.
	\end{cases}
	\end{equation*}	
\item 
If $(m,q) = (0,\tanh(\beta B))$, we additionally obtain $G_1^+(0,\tanh(\beta B)) = G_1^-(0,\tanh(\beta B))$ and then
	\begin{equation*}
	D^k \cG_2(m,q) = \begin{cases}
	-2 \beta^{k-1}  \sinh(\beta B)  \begin{pmatrix} 0 &  1 \\ k & 0 \end{pmatrix} & \text{if $k$ is even},\\[0.5cm]
	 \dfrac{2 \beta^k}{\cosh(\beta B)} \begin{pmatrix} 1 & 0 \\ 0 & 0 \end{pmatrix} -2\beta^{k-1}\cosh(\beta B) \begin{pmatrix}k  & 0 \\ 0 & 1 \end{pmatrix} & \text{if $k$ is odd}.
	\end{cases}
	\end{equation*}	
\item 
If  $(m,q) = (0,\tanh(\beta B))$ and $\beta = \cosh^2(\beta B)$, then
	\begin{equation*}
	D^k \cG_2(m,q) = \begin{cases}
	-2 \beta^{k-1}\sinh(\beta B) \begin{pmatrix} 0 &  1 \\ k  & 0 \end{pmatrix} & \text{if $k$ is even},\\[0.5cm]
	-2\beta^{k-1} \cosh(\beta B)\begin{pmatrix} k-1 & 0 \\ 0 & 1 \end{pmatrix} & \text{if $k$ is odd}.
	\end{cases}	
	\end{equation*}		
\item 
If $(m,q) = (0,\tanh(\beta B))$, $(\beta,B) = (\frac{3}{2}, \frac{2}{3} \arccosh(\sqrt{\frac{3}{2}}))$ and $k = 3$, then
	\begin{equation*}
	D^3 \cG_2(m,q) =  -2 \beta^2 \cosh(\beta B)\begin{pmatrix} 2  & 0 \\ 0 & 1\end{pmatrix}.
	\end{equation*}
\end{enumerate}
\end{lemma}

We are now ready to prove Theorems \ref{thm:MD:paramagnetic:subcritical} and \ref{thm:MD:ferromagnetic:supercritical}. The large deviation principles follow from the abstract results in \cite[Appendix A]{CoKr17}.\\

\subsection{Proof of Theorems \ref{thm:MD:paramagnetic:subcritical} and \ref{thm:MD:ferromagnetic:supercritical}}

The setting of Theorems \ref{thm:MD:paramagnetic:subcritical} and \ref{thm:MD:ferromagnetic:supercritical} corresponds to that of Proposition \ref{proposition:compact_expression_for_H_n} with $\nu = 0$. Having chosen a stationary point $(m,q)$ and applying Lemma \ref{lemma:DkG2_expressions}(b), we find
\begin{align*}
H_nf(\mathbf{x}) & = \ip{D \cG_2(m,q) \mathbf{x}}{ \nabla f(\mathbf{x})} + \ip{\bG_1(m,q) \nabla f(\mathbf{x})}{\nabla f(\mathbf{x})} + o(1) \\
& =  \ip{(\beta \hat{\bG}_1(m,q) - 2 \mathbb{B}(m)) \mathbf{x}}{\nabla f(\mathbf{x})} + \ip{\bG_1(m,q)\nabla f(\mathbf{x})}{\nabla f(\mathbf{x})} + o(1),
\end{align*}
where the matrices $\hat{\bG}_1$ and $\mathbb{B}$ are defined in \eqref{def:Ghat_matrix} and \eqref{def:B_matrix} respectively. The remainder $o(1)$ is uniform on compact sets. Therefore, for $f \in C_c^2(\mathbb{R}^2)$, $H_nf$ converges uniformly to $Hf(\mathbf{x}) = H(\mathbf{x},\nabla f(\mathbf{x}))$, where
\[
H(\mathbf{x}, \mathbf{p}) = \ip{(\beta \hat{\bG}_1(m,q) - 2 \mathbb{B}(m)) \mathbf{x}}{\mathbf{p}} + \ip{\bG_1(m,q)\mathbf{p}}{\mathbf{p}}.
\]
The large deviation results follow by Theorem A.14, Lemma 3.4 and Proposition 3.5 in \cite{CoKr17}. The Lagrangian is found by taking the Legendre-Fenchel transform of $H$ and is given by
\begin{equation*}
\cL(\mathbf{x}, \mathbf{v}) := \frac{1}{4}\ip{\bG_1^{-1}(m,q) [\mathbf{v} - (\beta\hat{\bG}_1(m,q) -2\mathbb{B}(m)) \mathbf{x}]}{\mathbf{v} - (\beta\hat{\bG}_1(m,q) -2\mathbb{B}(m)) \mathbf{x}}.
\end{equation*}
Observe that, in the case when $(m,q) = (0,\tanh(\beta B))$, we get
\[
\cL(\mathbf{x}, \mathbf{v}) := \frac{\cosh(\beta B)}{8} \left\vert \mathbf{v} - 2 \begin{pmatrix} \frac{\beta - \cosh^2(\beta B)}{\cosh(\beta B)} & 0 \\ 0 & -\cosh(\beta B) \end{pmatrix} \mathbf{x}\right\vert^2.
\]
This concludes the proof.

\section{Projection on a one-dimension subspace and moderate deviations at criticality}
 \label{section:projection_fixed_betaB}

For the proofs of Theorems \ref{thm:MD:paramagnetic:critical} and \ref{thm:MD:paramagnetic:tricritical}, we consider the stationary point $(m,q) = (0,\tanh(\beta B))$. Recall that, given the correct assumptions on the sequence $\{b_n\}_{n \geq 1}$, the expression for the Hamiltonian in Proposition \ref{proposition:compact_expression_for_H_n} reduces $\mu$-a.s. to
\begin{multline}\label{eqn:Hn:compact_form:generic:paramagnetic}
H_n f(x,y) =  \sum_{k=1}^5 \frac{b_n^{\nu + 1 -k}}{k!}\ip{D^k \cG_{2}(0,\tanh(\beta B)) \begin{pmatrix} x^k \\ k x^{k-1}y \end{pmatrix}}{\nabla f(x,y)}   \\
+ \ip{\bG_{1}(0,\tanh(\beta B)) \nabla f(x,y)}{\nabla f(x,y)} + o (1)+ o(b_n^{\nu - 4}). 
\end{multline}

If $\nu \in \{2,4\}$, the term corresponding to $D^1 \mathcal{G}_2 (0,\tanh (\beta B))$ is diverging and, more precisely, is diverging through a term containing the $y$ variable (see Lemma \ref{lemma:DkG2_expressions}(d)). We have a natural time-scale separation for the evolutions of our variables: $y$ is fast and converges very quickly to zero, whereas $x$ is slow and its limiting behavior can be characterized after suitably ``averaging out'' the dynamics of $y$. Corresponding to this observation, 
%
our aim is to prove that the sequence $\{H_n\}_{n \geq 1}$ admits a limiting operator $H$ and, additionally, the graph of this limit depends only on the $x$ variable. In other words, we want to prove a path-space large deviation principle for a projected process. \\

The projection on a one-dimensional subspace relies on the formal and recursive calculus explained in the next section (an analogous approach will be implemented also in Section~\ref{sect:formal_calculus:extended} to achieve the large deviation principles of Theorems~\ref{theorem:moderate_deviations_CW_critical_curve_rescaling}--\ref{theorem:moderate_deviations_CW_tri-critical_point_rescaling:2}). We want to mention that the results presented in Sections~\ref{section:formal_calculus_of_operators} and \ref{sect:formal_calculus:extended} take inspiration from the perturbation theory for Markov processes introduced in \cite{PaStVa77}.

\subsection{Formal calculus with operators and a recursive structure} \label{section:formal_calculus_of_operators}

We start by introducing a formal structure allowing to write the drift component in \eqref{eqn:Hn:compact_form:generic:paramagnetic} in abstract form. Afterwards, we introduce a method based on this abstract structure to perturb a function $\psi$ depending on the only variable $x$ to a function $F_{n,\psi}$ 
depending on $(x,y)$, so that the perturbation exactly cancels out the contributions of the drift operators to the $y$ variable.
	
\smallskip

Consider  the vector spaces of functions
\[
V := \left\{ \psi: \mathbb{R}^2 \rightarrow \mathbb{R} \, \left\vert \, \psi \text{ is of the type } \sum_{i=0}^r y^i \, \psi_i(x), \text{ with } \psi_i \in C^\infty_c(\bR) \right.\right\}
\]
and $V_i := \left\{ \psi: \mathbb{R}^2 \rightarrow \mathbb{R} \, \left\vert \, \psi \text{ is of the type } y^i \, \psi_i(x), \text{ with } \psi_i \in C^\infty_c(\bR) \right.\right\}$, for $i \in \mathbb{N}$. Moreover, for notational convenience, we will denote 
\begin{equation*}
V_{\mathrm{odd}} := \bigcup_{i \text{ odd}} V_{i}, \qquad V_{\mathrm{even}} := \bigcup_{i \text{ even}} V_{i}, \qquad V_{\mathrm{even}\setminus\{0\}} := \bigcup_{i \text{ even}, i \neq 0} V_{i}
\end{equation*} 
and 
\begin{equation*}
V_{\leq j} := \bigcup_{i \leq j} V_{i}.
\end{equation*}
Next we define a collection of operators on $V$. Let $a \in \mathbb{R}$ and $g: \mathbb{R}^2 \rightarrow \mathbb{R}$ a differentiable function. We consider the operators
\begin{equation}\label{def:calE:even}
\left.
\begin{array}{l}
\mathcal{Q}_k^+[a] g(x,y) := a x^{k-1}y g_x(x,y) \\[0.2cm]
\mathcal{Q}_k^-[a] g(x,y) := a x^k g_y(x,y)
\end{array}
\quad
\right]
\text{ for even $k$}
\end{equation}
and
\begin{equation}\label{def:calE:odd}
\left.
\begin{array}{l}
\mathcal{Q}_k^0[a] g(x,y) := a x^k g_x(x,y) \\[0.2cm]
\mathcal{Q}_k^1[a] g(x,y) := a x^{k-1} y g_y(x,y)
\end{array}
\quad\,
\right]
\text{ for odd $k$.}
\end{equation}
Note that the drift 
component in \eqref{eqn:Hn:compact_form:generic:paramagnetic} can be rewritten in terms of operators of the form \eqref{def:calE:even} and \eqref{def:calE:odd}. The result of the following lemma is immediate. 

\begin{lemma} \label{lemma:mappings_wrt_Vindex}
For all $a \in \bR$ and $i \in \bN$, we have 
\[
\mathcal{Q}_k^+[a] : V_i \rightarrow V_{i+1} \text{ and } \mathcal{Q}_k^-[a] : V_i \rightarrow V_{i-1},
\text{ for even $k$}
\]
and
\[
\mathcal{Q}_k^0[a], \mathcal{Q}_k^1[a] : V_i \rightarrow V_i, \text{ for odd $k$}.
\]
\end{lemma}

In particular, note that all operators map $V$ into $V$. Furthermore, the operators $\mathcal{Q}_k^1$, with odd $k$, have $V_0$ as a kernel. We will see that $Q_1^1$ plays a special role. \\

\begin{assumption} \label{assumption:abstract_sequences}
Assume there exist real constants $(a_k^+)_{k \geq 1}$, $(a_k^-)_{k \geq 1}$ if $k$ is even and $a_1^0=0$, $(a_k^0)_{k >1}$, $(a_k^1)_{k \geq 1}$ if $k$ is odd, for which, given a continuously differentiable function $g: \mathbb{R}^2 \rightarrow \mathbb{R}$, we can write
\begin{itemize}
\item
for even $k$,
\[
\mathcal{Q}_k g = \mathcal{Q}_k^+ g + \mathcal{Q}_k^- g
\quad \text{ with } \quad
\left\{
\begin{array}{l}
\mathcal{Q}_k^+ g (x,y) := \mathcal{Q}_k^+[a_k^+] g (x,y),\\[0.2cm]
\mathcal{Q}_k^- g (x,y) := \mathcal{Q}_k^-[a_k^-] g (x,y) ,
\end{array}
\right.
\]
\item
for odd $k$,
\[
\mathcal{Q}_k g = \mathcal{Q}_k^0 g + \mathcal{Q}_k^1 g
\quad \text{ with } \quad
\left\{
\begin{array}{l}
\mathcal{Q}_k^0 g (x,y) := \mathcal{Q}_k^0[a_k^0] g (x,y), \\[0.2cm]
\mathcal{Q}_k^1 g (x,y) := \mathcal{Q}_k^1[a_k^1] g (x,y). 
\end{array}
\right.
\]
\end{itemize}
\end{assumption}

Observe that the drift term in \eqref{eqn:Hn:compact_form:generic:paramagnetic} is of the form
\begin{equation*}
\cQ^{(n)} \psi(x,y) := \sum_{k=1}^{\nu+1} b_n^{\nu+1 -k} \cQ_k \psi(x,y).
\end{equation*}\
We aim at abstractly showing that, for any function $\psi \in V_0$ and sequence $b_n \rightarrow \infty$, we can find a perturbation $F_{n,\psi} \in V$ of $\psi$ for which there exists $\tilde{\psi} \in V_0$ such that
\begin{equation*}
\tilde{\psi}(x) - \cQ^{(n)} F_{n,\psi} (x,y) = o(1). 
\end{equation*}
We will construct the perturbation in an inductive fashion. We start by motivating the first step of the construction. Let $\psi \in V_0$, i.e. a function only depending on $x$. Then:
\begin{enumerate}[(1)]
	\item $\mathcal{Q}_1 \psi = 0$, but $\mathcal{Q}_2 \psi \neq 0$ and moreover $\mathcal{Q}_2 \psi \in V_1$ because of the action of $\mathcal{Q}_2^+$.
	\item The leading order term in $\mathcal{Q}^{(n)} \psi$ is given by $b_n^{\nu - 1} \mathcal{Q}_2 \psi$.
\end{enumerate}
Next, we consider a perturbation $\psi + b_n^{-1} \psi[1]$ of $\psi$, with $\psi[1]$ and the order $b_n^{-1}$ chosen in the following way:
\begin{enumerate}[(1)]\setcounter{enumi}{2}
	\item The action of $\mathcal{Q}^{(n)}$ on $b_n^{-1} \psi[1]$ yields a leading order term $b_n^{\nu - 1} \mathcal{Q}_1 \psi[1]$, which matches the order of $b_n^{\nu - 1} \mathcal{Q}_2 \psi$ in step (2) above.
	\item We choose $\psi[1]$ so that $\mathcal{Q}_2 \psi + \mathcal{Q}_1 \psi[1] = 0$. 
\end{enumerate}
At this point, the leading order term of $\mathcal{Q}^{(n)}(\psi + b_n^{-1} \psi[1])$ equals $b_n^{\nu-2}\left(\mathcal{Q}_3 \psi + \mathcal{Q}_2 \psi[1]\right)$ and the construction proceeds by considering $\psi + b_n^{-1} \psi[1] + b_n^{-2} \psi[2]$, where $\psi[2]$ is chosen so that
\begin{enumerate}[(1)]\setcounter{enumi}{4}
	\item $\mathcal{Q}_3 \psi + \mathcal{Q}_2 \psi[1] + \mathcal{Q}_1 \psi[2] \in V_0$. 
\end{enumerate}
Note that we can only assure that the sum is in $V_0$. This is due to the specific structure of the operators that we will discuss in Lemma \ref{lemma:fk_in_even_or_odd}.

Now there are two possibilities
\begin{enumerate}[(a)]
	\item $\mathcal{Q}_3 \psi + \mathcal{Q}_2 \psi[1] + \mathcal{Q}_1 \psi[2] \neq 0$. In this case $\nu =2$ is the maximal $\nu$ that we can use for this particular problem. In addition, the outcome of this sum will be in the form $c x^3 \psi'(x)$ and hence determine the limiting drift in the Hamiltonian.
	\item $\mathcal{Q}_3 \psi + \mathcal{Q}_2 \psi[1] + \mathcal{Q}_1 \psi[2] = 0$. In this case, $\nu = 2$ is a possible option. However, we can use a larger $\nu$ and proceed with perturbing $\psi$ with even higher order terms.
\end{enumerate}
As a final outcome, our perturbation of $\psi \in V_0$ will be of the form
\begin{equation}\label{eqn:def:perturbation}
F_{n,\psi}(x,y) := \sum_{l=0}^\nu b_n^{-l} \psi[l](x,y),
\end{equation}
where we write $\psi[0] = \psi$ for notational convenience. Our next step is to introduce the procedure that tells us how to choose $\psi[r+1]$ if we know $\psi$ and $\psi[1],\dots,\psi[r]$.

\begin{lemma} \label{lemma:kernel_operator_decomp}
	Let $\psi \in V$. Define the maps
	\[
	P_0 \colon V \rightarrow V_0, \quad \text{ with }  \quad P_0(\psi)(x,y) := \psi_0(x), 
	\]
	and
	\[
	P \colon V \rightarrow \bigcup_{i \geq 1} V_i, \quad \text{ with } \quad P(\psi)(x,y) := - \sum_{i=1}^r y^i \dfrac{\psi_i(x)}{i a_1^1}.
	\]
	Then, we have $\psi(x,y) + \mathcal{Q}_1^1 P(\psi)(x,y) = \psi_0(x)$.
\end{lemma}

\begin{proof}
	By direct computation, we get
	\[
	\mathcal{Q}_1^1 P(\psi)(x,y) = a_1^1 y \partial_y [P(\psi)(x,y)] = - \sum_{i=1}^r y^i \psi_i(x),
	\]
	from which the conclusion follows.
\end{proof}

Starting from $\psi = \psi[0] \in V_0$, we define recursively 
\begin{equation}\label{def:recursion:psi:phi}
\psi[r] =  P\left(\sum_{l=0}^{r-1} \mathcal{Q}_{r+1-l} \psi[l] \right)
\text{ and } 
\phi[r] = \sum_{l=0}^{r-1} \mathcal{Q}_{r+1-l} \psi[l], 
\end{equation}
for all $1 \leq r \leq \nu$. 

\begin{remark}\label{rmk:properties:psi:phi}
	For all $l \geq 1$, $\psi[l] = P \phi[l]$ and, by Lemma~\ref{lemma:kernel_operator_decomp}, $\phi[l] + \mathcal{Q}_1^1 \psi[l] = P_0 \phi[l]$, which is exactly the result that we aimed to find in steps (4) and (5) above.
\end{remark}

Next, we evaluate the action of $\mathcal{Q}^{(n)}$ applied to our perturbation of $\psi$.

\begin{proposition} \label{proposition:perturbation_abstract_Q}
	Fix $\nu \geq 2$ an even natural number and suppose that Assumption \ref{assumption:abstract_sequences} holds true for this $\nu$. Consider the operator
	\begin{equation}\label{eqn:operator:calQ}
	\cQ^{(n)} \psi(x,y) := \sum_{k=1}^{\nu+1} b_n^{\nu+1 -k} \cQ_k \psi(x,y)
	\end{equation}
	and, for $\psi = \psi[0] \in V_0$, define $F_{n,\psi}(x,y) := \sum_{l=0}^\nu b_n^{-l} \psi[l](x,y)$. We have
	\begin{equation*}
	\cQ^{(n)} F_{n,\psi}(x,y) = \sum_{i = 1}^\nu b_n^{\nu - i} \, P_0 \phi[i](x) + o(1),
	\end{equation*}
	where $o(1)$ is meant according to Definition~\ref{def:o(1)}.
\end{proposition}

\begin{proof}
	We aim at determining the leading order term of
	\begin{align*}
	\cQ^{(n)} F_{n,\psi}(x,y) &= \sum_{k=1}^{\nu+1} b_n^{\nu+1-k} \mathcal{Q}_k F_{n,\psi}(x,y) \\
	&= \sum_{k=1}^{\nu+1} \sum_{l=0}^\nu b_n^{\nu+1-k-l} \, \mathcal{Q}_k \psi[l](x,y) + o(1).
	\end{align*}
	The remainder $o(1)$ contains lower order terms in the expansion and it is small 
	%
	as $b_n^{-l} \psi[l]$ is uniformly bounded on the state space $E_n$ for any $n \in \mathbb{N}$ (see Lemma~\ref{lemma:convergence_of_functions_in_domain}). We re-arrange the first sum by changing indices $r = k+l-1$. It yields
	\begin{equation*}
	\cQ^{(n)} F_{n,\psi}(x,y) = \sum_{r=0}^{\nu} b_n^{\nu-r} \sum_{l=0}^{r} \mathcal{Q}_{r-l+1} \psi[l](x,y) + o(1).
	\end{equation*} 
	Observe that the term corresponding to $r=0$ vanishes as $\mathcal{Q}_1 \psi[0] = 0$. By \eqref{def:recursion:psi:phi} and the properties stated in Remark~\ref{rmk:properties:psi:phi}, we get
	%
	\begin{align*}
	\cQ^{(n)} F_{n,\psi}(x,y) & = \sum_{r=1}^{\nu} b_n^{\nu-r} \left[\mathcal{Q}_{1} \psi[r](x,y) + \sum_{l=0}^{r-1} \mathcal{Q}_{r-l+1} \psi[l](x,y)\right] + o(1) \\
	& = \sum_{r=1}^{\nu} b_n^{\nu-r} \big[ \mathcal{Q}_1 \psi[r](x,y) + \phi[r](x,y) \big]+ o(1)  \\
	& = \sum_{r=1}^{\nu} b_n^{\nu-r} \, P_0 \phi[r](x) + o(1).
	\end{align*}

	%
\end{proof}

For the cases we are interested in, we will use $\nu \in \{2,4\}$. Thus, to conclude, we need to consider the action of $P_0$ on the functions $\phi[r]$, for $r = 1, \dots, 4$. This is the content of the next two statements.

The functions $\psi[r]$, $\phi[r]$ belong to the vector spaces $V_i$ according to the classification given in the next lemma. 

\begin{lemma} \label{lemma:fk_in_even_or_odd}
	If $\psi = \psi[0] = \phi[0] \in V_0$, then
	\begin{equation*}
	\psi[r] \in \begin{cases}
	V_{\leq r} \cap V_{\mathrm{even} \setminus\{0\}}  & \text{if } r \text{ is even}, \\
	V_{\leq r} \cap V_{\mathrm{odd}} & \text{if } r \text{ is odd}
	\end{cases}
	\end{equation*}
	and
	\begin{equation*}
	\phi[l] \in \begin{cases}
	V_{\leq r} \cap V_{\mathrm{even}}  & \text{if } r \text{ is even}, \\
	V_{\leq r} \cap V_{\mathrm{odd}} & \text{if } r \text{ is odd}.
	\end{cases}
	\end{equation*}
\end{lemma}

\begin{proof}
	As all operators $\mathcal{Q}_k$ map $V_{i \leq k}$ into $V_{i \leq k+1}$ and the projection $P$  maps $V_0$ to~$\{0\}$, it suffices to prove that, for any $r \in \mathbb{N}$, $\psi[2r] \in V_{\mathrm{even}}$ and $\psi[2r+1] \in V_{\mathrm{odd}}$. \\
	We proceed by induction. Clearly the result holds true for $r = 0$. We are left to show the inductive step. Suppose the claim is valid for all positive integers less than $r$, we must prove that, if $r$ is odd (resp. even) and $\psi[r] \in V_{\mathrm{odd}}$ (resp. $V_{\mathrm{even}}$), then $\psi[r+1] \in V_{\mathrm{even}}$ (resp. $V_{\mathrm{odd}}$). We stick on the odd $r$ case, the other being similar.  By definition, we have
	\begin{equation*}
	\psi[r+1] =  P\left(\sum_{l=0}^{r} \mathcal{Q}_{r+2-l} \psi[l] \right).
	\end{equation*}
	Let us analyze the sum on the right-hand side of the previous formula. It is composed of terms of two types: either $l$ is even or it is odd. 
	\begin{itemize}
		\item If $l$ even, then by inductive hypothesis $\psi[l] \in V_{\mathrm{even}}$. Additionally, $r+2 - l$ is odd, so that by Lemma~\ref{lemma:mappings_wrt_Vindex}, the operator $\mathcal{Q}_k$ maps $V_{\mathrm{even}}$  to  $V_{\mathrm{even}}$ and therefore, $\mathcal{Q}_{r+2-l} \psi[l] \in V_{\mathrm{even}}$.
		\item If $l$ odd, then by inductive hypothesis $\psi[l] \in V_{\mathrm{odd}}$. Additionally, $r+2 - l$ is even, so that by Lemma~\ref{lemma:mappings_wrt_Vindex}, the operator $\mathcal{Q}_k$ maps $V_{\mathrm{odd}}$  to  $V_{\mathrm{even}}$ and therefore, $\mathcal{Q}_{r+2-l} \psi[l] \in V_{\mathrm{even}}$.
	\end{itemize}
	This yields that $\psi[r+1] \in V_{\mathrm{even}}$, giving the induction step.
\end{proof}

To evaluate the limiting drift from the expression in Proposition \ref{proposition:perturbation_abstract_Q}, we need to evaluate $P_0$ in the functions $\phi[i]$, with $i \in \mathbb{N}$, $i \leq \nu$. From Lemma~\ref{lemma:fk_in_even_or_odd}, we have $\phi[2i+1] \in V_{\mathrm{odd}}$ which is in the kernel of $P_0$. An explicit calculation for the even $i$ is done in the next lemma.

\begin{lemma} \label{lemma:evaluation_P0_phi_l}
Consider the setting of Proposition \ref{proposition:perturbation_abstract_Q}. For $\psi = \psi[0] \in V_0$, we have $P_0 \phi[l] = 0$ if $l$ is odd and 
\begin{equation}\label{lemma:projection_on_kernel}
P_0 \phi[l] = \begin{cases}
\mathcal{Q}_3^0 \psi + \mathcal{Q}_2^- P \mathcal{Q}_2^+ \psi  & \text{if } l =2, \\[0.3cm]
\mathcal{Q}_5^0 \psi + \mathcal{Q}^-_2 P \mathcal{Q}_4^+ \psi + \mathcal{Q}_4^- P \mathcal{Q}_2^+ \psi + \mathcal{Q}_2^- P\mathcal{Q}_3^1P \mathcal{Q}^+_2 \psi  \\
\qquad  + \mathcal{Q}_2^- P(\mathcal{Q}_3^0 + \mathcal{Q}_2^- P \mathcal{Q}_2^+)P \mathcal{Q}^+_2 \psi & \text{if } l =4.
\end{cases}
\end{equation}
\end{lemma}

\begin{proof}
By Lemma \ref{lemma:fk_in_even_or_odd}, if $l$ is odd, then $\phi[l] \in V_{\text{odd}}$ and, as a consequence, $P_0 \phi[l] = 0$. We are left to understand how the even terms contribute to $V_0$. We exploit the recursive structure of the functions $\psi[l]$ and $\phi[l]$.\\
For $k=2$, we find $\mathcal{Q}_3 \psi + \mathcal{Q}_2^- \psi[1]$, as $\mathcal{Q}_2^+$ always maps into the kernel of $P_0$.
For $k = 4$ we find $\mathcal{Q}_5 \psi + \mathcal{Q}_4^- \psi[1] + \mathcal{Q}_2^- \psi[3]$, as $\psi[2]$ has no $V_0$ component and $\mathcal{Q}_3$ maps $V_i$ into $V_i$ for all $i$. Thus, we obtain
\begin{multline*}
\mathcal{Q}_5^0 \psi + \mathcal{Q}_4^- P(\mathcal{Q}^+_2 \psi) + \mathcal{Q}_2^- P\left(\mathcal{Q}_4 \psi + \mathcal{Q}_3 \psi[1] + \mathcal{Q}_2^- \psi[2] \right) \\
= \mathcal{Q}_5^0 \psi + \mathcal{Q}_4^- P(\mathcal{Q}^+_2 \psi) + \mathcal{Q}_2^- P\left[\mathcal{Q}_4^+ \psi + \mathcal{Q}_3 P(\mathcal{Q}_2^+\psi) + \mathcal{Q}_2^- P(\mathcal{Q}^+_2 P(\mathcal{Q}^+_2 \psi))  \right]. 
\end{multline*}
\end{proof}

A straightforward computation yields the following result that will be useful for the computation of the constants involved in the operators in the previous lemma.

\begin{lemma} \label{lemma:combo-P+}
Given $\psi \in V$, it holds
\begin{align*}
\mathcal{Q}_k^- P \mathcal{Q}_j^+ \psi(x,y) & = - \frac{a_k^- a_j^+}{a_1^1} x^{k+j-1} \psi_x(x,y), \\
\mathcal{Q}_2^- P\mathcal{Q}_3^1P \mathcal{Q}^+_2 \psi(x,y) & = \frac{a_2^- a_3^1 a_2^+}{(a_1^1)^2} x^5 \psi_x(x,y).
\end{align*}	
\end{lemma}


We see in \eqref{lemma:projection_on_kernel} that $P_0 \phi[4]$ contains a part resembling $P_0 \phi[2]$. On the one hand, if $P_0 \phi[2] = 0$, then $P_0\phi[4]$ has a much simpler structure. On the other, whenever $P_0 \phi[2] \neq 0$, $P_0\phi[4]$ is not needed as in \eqref{eqn:compact_exp_G_2} there are terms of higher order that dominate. As a consequence, we will only ever work with the simplified result for $P_0 \phi[4]$. Similar observations involving higher order recursions can be made for any arbitrary $P_0 \phi[2l]$ with $l \in \mathbb{N}$, $l \geq 3$. By combining these remarks with the type of calculations made for getting the expressions presented in Lemma~\ref{lemma:combo-P+}, we conjecture the following general structure.

\begin{conjecture} \label{conjecture:form_of_drift}
Let Assumption \ref{assumption:abstract_sequences} be satisfied with $\nu$ even. Suppose that $P_0\phi[2l] = 0$ for all $l \in \mathbb{N}$ with $2l < \nu$.
Then
{\small
\begin{multline}\label{eqn:conjecture:drift}
P_0 \psi[\nu](x,y)  \\
= \left[a_{\nu+1}^0 + \sum_{n \geq 2} \sum_{\substack{i_1 + \dots + i_n = \nu + n \\ i_1, i_n \text{ even} \\ i_j \text{ odd and } \neq 1 \text{ for } j \notin \{1,n\}}} (-1)^{n-1}\frac{a_{i_1}^- \left(\prod_{j = 2}^{n-1} a_{i_j}^1 \right)a_{i_n}^+}{(a_1^1)^{n-1}} \right] x^{\nu +1} \psi_x(x) + o(1).
\end{multline}
}
\end{conjecture}

We neither prove nor use \eqref{eqn:conjecture:drift}, but we believe it is of interest from a structural point of view and deserves to be stated.

\subsection{Proofs of Theorems \ref{thm:MD:paramagnetic:critical} and \ref{thm:MD:paramagnetic:tricritical}}
	
The formal calculus we developed in Section~\ref{section:formal_calculus_of_operators} is used to formally identify the limiting operator $H$ of the sequence $H_n$ given in~\eqref{eqn:Hn:compact_form:generic:paramagnetic}. However, it is not possible to show directly $H \subseteq \LIM_n H_n$ as in the proof of Theorems~\ref{thm:MD:paramagnetic:subcritical} and~\ref{thm:MD:ferromagnetic:supercritical}, since the most functions $\psi \in C_c^{\infty}(\mathbb{R})$ cause $\sup_n \|H_n F_{n,\psi}\| \nless \infty $ and thus we can not prove $\LIM_n H_n F_{n,\psi} = H\psi$.

To circumvent the problem, we relax our definition of limiting operator. In particular, we introduce two limiting Hamiltonians $H_{\dagger}$ and $H_\ddagger$, approximating $H$ from above and below respectively, and then we characterize $H$ by matching  upper and lower bound. \\
We summarize the notions needed for our result and the abstract machinery used for the proof of a large deviation principle via well-posedness of Hamilton-Jacobi equations in Appendix \ref{appendix:large_deviations_for_projected_processes}. We rely on Theorem \ref{theorem:Abstract_LDP} for which we must check the following conditions:

\begin{enumerate}[(a)]
\item \label{item:LDP_list:1}
The processes $\left( \left( b_n m_n(b_n^{\nu}t), b_n \left( q_n(b_n^{\nu} t) - \tanh(\beta B) \right) \right), t \geq 0 \right)$  satisfy an appropriate exponential compact containment condition. 
\item \label{item:LDP_list:2}
There exist two Hamiltonians $H_\dagger \subseteq C_l(\mathbb{R}^2) \times C_b(\mathbb{R}^2)$ and $H_\ddagger \subseteq C_u(\mathbb{R}^2) \times C_b(\mathbb{R}^2)$ such that $H_\dagger \subseteq ex-\subLIM_n H_n$ and $H_\ddagger \subseteq ex-\superLIM_n H_n$. 
\item  \label{item:LDP_list:3}
There is an operator $H \subseteq C_b(\bR) \times C_b(\bR)$ such that every viscosity subsolution to $f- \lambda H_\dagger f = h$ is a viscosity subsolution to $f - \lambda Hf = h$ and such that every supersolution to $f - \lambda H_\ddagger f = h$ is a viscosity supersolution to $f - \lambda H f$. The operators $H_\dagger$ and $H_\ddagger$ should be thought of as upper and lower bounds for the ``true'' limiting $H$ of the sequence $H_n$. 
\item \label{item:LDP_list:4}
The comparison principle holds for the Hamilton-Jacobi equation $f - \lambda H f = h$ for all $h \in C_b(\mathbb{R})$ and all $\lambda>0$. The proof of this statement is immediate, since the operator $H$ we will be dealing with is of the type considered in \cite{CoKr17}.
\end{enumerate}

We will start with the verification of \eqref{item:LDP_list:2}+\eqref{item:LDP_list:3}, which are based on the expansion in Proposition \ref{proposition:compact_expression_for_H_n} and the formal calculus in Section \ref{section:formal_calculus_of_operators}. Afterwards, we proceed with the verification of \eqref{item:LDP_list:1}, for which we will use the result of \eqref{item:LDP_list:2}. Finally, the form of the operator $H$ is of the type  considered in e.g. \cite{CoKr17} or \cite[Section 10.3.3]{FK06} and thus, the establishment of \eqref{item:LDP_list:4} is immediate. \\

Consider the statement of Proposition \ref{proposition:compact_expression_for_H_n}. We want to extract the limiting behavior of the operators $H_n$ presented there. If $(m,q) = (0,\tanh(\beta B))$ the term in \eqref{eqn:compact_exp_G_1} vanishes, whereas the term in \eqref{eqn:compact_exp_G_3} converges if $\nabla f_n(x,y) \xrightarrow[]{\: n \to \infty \:} \nabla f(x)$ uniformly on compact sets (see Lemma~\ref{lemma:convergence_of_functions_in_domain} below). For the term in \eqref{eqn:compact_exp_G_2}, we use the results from Section~\ref{section:formal_calculus_of_operators}.

For $k \in \{1,\dots,5\}$,  denote by $Q_k : C^2(\bR^2) \rightarrow C^1(\bR^2)$ the operator 
\begin{equation*}
(Q_k g)(x,y) := \frac{1}{k!}\ip{D^k \cG_{2}(0,\tanh(\beta B)) \begin{pmatrix} x^k \\ k x^{k-1}y \end{pmatrix}}{\nabla g(x,y)}.
\end{equation*}
Note that, by the diagonal structure of $D^k \cG_{2}(0,\tanh(\beta B))$ established in Lemma \ref{lemma:DkG2_expressions}, we find
\begin{allowdisplaybreaks}%
\begin{align}
Q_1 g(x,y) & =  Q_1^1 g(x,y) = -2 \cosh(\beta B) y g_y(x,y), \label{def:Q1}\\[0.5cm]
Q_2 g (x,y) &= Q_2^+ g(x,y) + Q_2^- g(x,y), \nonumber\\[0.2cm]
&\quad \text{ with } 
\left\{
\begin{array}{l}
Q_2^+ g (x,y) = -2 \beta \sinh(\beta B) xy g_x(x,y)\\[0.2cm]
Q_2^- g (x,y) = - 2 \beta \sinh(\beta B) x^2 g_y(x,y), 
\end{array}
\right. \label{def:Q2}\\[0.5cm]
Q_3 g (x,y) &= Q_3^0 g(x,y) + Q_3^1 g(x,y), \nonumber\\[0.2cm]
& \quad \text{ with } 
\left\{
\begin{array}{l}
Q_3^0 g (x,y) = - \frac{2}{3} \beta^2 \cosh(\beta B) x^3 g_x(x,y)\\[0.2cm]
Q_3^1 (x,y) =- \beta^2 \cosh(\beta B) x^2 y g_y(x,y), 
\end{array}
\right. \label{def:Q3}\\[0.5cm]
Q_4 g (x,y) &= Q_4^+ g(x,y) + Q_4^- g(x,y), \nonumber\\[0.2cm]
& \quad \text{ with } 
\left\{
\begin{array}{l}
Q_4^+ g (x,y) = -\frac{1}{3} \beta^3 \sinh(\beta B) x^3y g_x(x,y)\\[0.2cm]
Q_4^- g (x,y) = - \frac{1}{3} \beta^3 \sinh(\beta B) x^4 g_y(x,y), 
\end{array}
\right. \label{def:Q4}\\[0.5cm]
Q_5 g (x,y) &= Q_5^0 g(x,y) + Q_5^1 g(x,y), \nonumber\\[0.2cm]
& \quad \text{ with }
\left\{
\begin{array}{l}
Q_5^0 g (x,y) = - \frac{1}{15} \beta^4 \cosh(\beta B) x^5 g_x(x,y)\\[0.2cm]
Q_5^1 (x,y) =- \frac{1}{12} \beta^4 \cosh(\beta B) x^4 y g_y(x,y). 
\end{array}
\right. \label{def:Q5}
\end{align}%
\end{allowdisplaybreaks}%
The operators $Q_k^z$, with $z \in \{+,-,0,1\}$, are of the type \eqref{def:calE:even} and \eqref{def:calE:odd} for particular choices of the constant $a \in \mathbb{R}$.

Observe that the term \eqref{eqn:compact_exp_G_2} has the form \eqref{eqn:operator:calQ}  with $\mathcal{Q}_k = Q_k$ given by \eqref{def:Q1}--\eqref{def:Q5}. Moreover, Assumption~\ref{assumption:abstract_sequences} is satisfied by the $Q_k$'s ($k=1, \dots, 5$). For $\psi \in C_c^\infty(\bR)$, we follow Proposition \ref{proposition:perturbation_abstract_Q} and 
%
define approximating functions $F_{n,\psi}$ thanks to which we can determine the linear part of the limiting Hamiltonian $H$. Recall that the quadratic part of $H$ comes from \eqref{eqn:compact_exp_G_3} after showing uniform convergence for the gradient. \\

The next lemma proves uniform convergence for the sequence of perturbation functions $F_{n,\psi}$ and for the sequence of the gradients.

\begin{lemma} \label{lemma:convergence_of_functions_in_domain}
Suppose we are either in the setting of Theorem \ref{thm:MD:paramagnetic:critical} and $\nu =2$ or in the setting of Theorem \ref{thm:MD:paramagnetic:tricritical} and $\nu = 4$. For $\psi \in C_c^\infty(\bR)$, define the approximation
\begin{equation}\label{eqn:perturbation:LDP_statements}
F_{n,\psi}(x,y) := \sum_{l=0}^\nu b_n^{-l} \psi[l](x,y),
\end{equation}
where $\psi[\cdot]$ are defined recursively according to \eqref{def:recursion:psi:phi}. Moreover, let $R := [a,b] \times [c,d]$, with $a < b$ and $c < d$, be a rectangle in $\mathbb{R}^2$. Then $F_{n,\psi} \in C_c^\infty(\bR^2)$, $\LIM F_{n,\psi} = \psi$ and 
\begin{equation} \label{eqn:convergence_gradients_Psi}
\sup_{(x,y) \in R \cap E_n} |\nabla F_{n,\psi}(x,y) - \nabla \psi(x)| = 0
\end{equation}
for all rectangles $R \subseteq \bR^2$.
\end{lemma}

\begin{proof}
By Lemma \ref{lemma:fk_in_even_or_odd} we find that $\psi[l] \in V_{\leq l}$, i.e. it is of the form $\psi[l](x,y) = \sum_{i = 0}^l y^i \psi[l]_i(x)$ with $\psi[l]_i \in C_c^\infty(\bR)$. Thus, as $b_n \rightarrow \infty$, we find that 
	\begin{equation*}
	\lim_{n \to \infty}  \sup_{(x,y) \in R \cap E_n} \left|F_{n,\psi}(x,y) - \psi(x)\right| + \left|\begin{bmatrix}
	\partial_x F_{n,\psi}(x,y) \\ \partial_y F_{n,\psi}(x,y) 
	\end{bmatrix} - \begin{bmatrix}
	\psi'(x) \\ 0 
	\end{bmatrix}\right| = 0
	\end{equation*}
	for all rectangles $R \subseteq \bR^2$. The second part of this limiting statement establishes \eqref{eqn:convergence_gradients_Psi}. To show $\LIM F_{n,\psi} = \psi$, we need the first part of this limiting statement and uniform boundedness of the sequence $F_{n,\psi}$. This final property follows since  $E_n \subseteq \bR \times [-2b_n,2b_n]$, implying that $b_n^{-l}\psi[l]$ is bounded for each $l$. 
\end{proof}

We start by calculating the terms in \eqref{eqn:compact_exp_G_2} that contribute to the limit 
via Proposition~\ref{proposition:perturbation_abstract_Q} and Lemma~\ref{lemma:evaluation_P0_phi_l}. 

\begin{lemma} \label{lemma:limiting_drifts}
Let $(\beta, B)$ satisfy $\beta = \cosh^2 (\beta B)$. For $f \in V$, we have
\begin{equation}\label{eqn:drift:critical_curve}
\left( Q_3^0 + Q_2^- P Q_2^+ \right) f(x,y) = \frac{2}{3} \beta (2 \beta - 3) \cosh (\beta B) x^3 f_x(x,y).
\end{equation}
Moreover, at the tri-critical point $(\beta_{\mathrm{tc}},B_{\mathrm{tc}})$ we obtain $(Q_3^0 + Q_2^- P Q_2^+)f = 0$ and
\begin{equation}\label{eqn:drift:tri-critical_point}
\left( Q_5^0 + Q_2^- P Q_4^+ + Q_4^- P Q_2^+ + Q_2^- P Q_3^1 P Q_2^+ \right) f(x,y) = - \frac{9}{10} \sqrt{\frac{3}{2}} \, x^5 f_x (x,y).
\end{equation}
\end{lemma}

\begin{proof}
	It suffices to prove the statement for $f$ of the form $y^i g(x)$, for some function $g \in C^2(\mathbb{R})$. Preparing for the use of Lemma \ref{lemma:combo-P+}, we list the relevant constants:
	\[
	a_1^1 = - 2 \cosh(\beta B), \quad a_2^{\pm}  = - 2 \beta \sinh(\beta B), \quad
	a_3^1 = - \beta^2 \cosh(\beta B), \quad  a_4^\pm  = - \frac{1}{3} \beta^3 \sinh(\beta B).
	\]
	
	We prove our first claim. The term $Q_3^0 f$ is given in \eqref{def:Q3}, whereas Lemma \ref{lemma:combo-P+} yields
	\begin{equation*}
	Q_2^- P Q_2^+ f(x,y) = 2 \beta^2 \sinh(\beta B) \tanh(\beta B) x^3 y^i g_x(x).
	\end{equation*}
	Combining these two results, we get
	\begin{equation*}
	\left( Q_3^0 + Q_2^- P Q_2^+ \right) f(x,y) = \frac{2}{3} \beta^2 \cosh (\beta B) \big[ 3 \tanh^2 (\beta B) - 1\big] x^3 f_x (x,y).
	\end{equation*}
	By using the fundamental identity $\cosh^2 - \sinh^2 = 1$ for hyperbolic functions and the fact we are on the critical curve, simple algebraic manipulations lead to the first conclusion.
	
	\smallskip
	
	We proceed with proving \eqref{eqn:drift:tri-critical_point}. The term $Q_5^0$ is defined in \eqref{def:Q5}. By Lemma~\ref{lemma:combo-P+} we find
	\[
	Q_2^- P Q_4^+ f(x,y) = Q_4^- P Q_2^+ f(x,y) = \frac{1}{3} \beta^4 \tanh(\beta B) \sinh(\beta B) x^5 f_x (x,y)
	\]
	and
	\[
	Q_2^- P Q_3^1 P Q_2^+ f(x,y) = - \beta^4 \tanh(\beta B) \sinh(\beta B) x^5 f_x(x,y).
	\]
	Adding the contributions above gives
	\begin{multline*}
	\left( Q_5^0 + Q_2^- P Q_4^+ + Q_4^- P Q_2^+ + Q_2^- P Q_3^1 P Q_2^+ \right) f(x,y) \\
	= - \frac{1}{15} \beta^4 \cosh(\beta B) \big[ 5 \tanh^2(\beta B) + 1 \big] x^5 f_x(x,y).
	\end{multline*}
	Plugging the value $\beta = \beta_{\mathrm{tc}} = \frac{3}{2}$ yields the result.
\end{proof}

{\em Approximating Hamiltonians and domain extensions.} The natural perturbations of our functions $\psi$ do not allow for uniform bounds of $\|H_n F_{n,\psi}\|$. We repair this lack by cutting off the functions. To this purpose, we introduce a collection of smooth increasing functions $\chi_{n} : \bR \rightarrow \bR$ such that  $0 \leq \chi_n' \leq 1$ and
\begin{equation}\label{eqn:cut-off}
\chi_n(z) = 
\begin{cases}
\log \log b_n + 1 & \text{ if } z \leq -  \log \log b_n \\[0.2cm]
z & \text{ if } - \log \log b_n + 2 \leq z \leq \log \log b_n - 2 \\[0.2cm]
\log \log b_n - 1  & \text{ if } z \geq \log \log b_n.
\end{cases}
\end{equation}

\begin{lemma}\label{lmm:cutoff_functions} 
Suppose we are either in the setting of Theorem~\ref{thm:MD:paramagnetic:critical} with $\nu=2$ or in the setting of Theorem~\ref{thm:MD:paramagnetic:tricritical} with $\nu=4$. Let $\varepsilon \in (0,1)$ and $\psi \in C_c^\infty(\mathbb{R})$. Consider the cut-off \eqref{eqn:cut-off} and  define the functions
\begin{equation}\label{def:cutoff_perturbation+Lyapunov_functions}
\chi_n \Big( F_{n,\psi}(x,y) \pm \varepsilon \log(1+x^2+y^2) \Big) 
\end{equation}
with $F_{n,\psi}$ as in \eqref{eqn:def:perturbation}+\eqref{def:recursion:psi:phi}. Then, 
\begin{enumerate}[(a)]
\item \label{item_1:lmm:cutoff_functions}
For any $C > 0$ there is an $N = N(C)$ such that, for any $n \geq N$, we have $\chi_n \equiv \mathrm{id}$ on the compact set $K_1 = K_1(C):=\left\{(x,y) \in \bR^2 \, \middle| \, \varepsilon \log(1+x^2+y^2) \leq C \right\}$.
\item \label{item_2:lmm:cutoff_functions}
Let $\overline{C}$ be a positive constant providing a uniform bound for the sequence $(F_{n,\psi})_{n \geq 1}$ (cf. Lemma~\ref{lemma:convergence_of_functions_in_domain}) and consider the compact set
\[
K_{2,n} := \left\{ (x,y) \in \mathbb{R}^2 \, \left\vert \, \frac{\varepsilon}{2} \log (1+x^2+y^2) \leq \max \{\overline{C}, \, 2 \log \log b_n \} \right.\right\}.
\]
The function \eqref{def:cutoff_perturbation+Lyapunov_functions} is constant outside $K_{2,n}$.
\end{enumerate}
\end{lemma}

\begin{proof}
We start by proving \eqref{item_1:lmm:cutoff_functions}. The function $F_{n,\psi}$ is uniformly bounded by Lemma~\ref{lemma:convergence_of_functions_in_domain}. Consider an arbitrary $C > 0$. The mapping $(x,y) \mapsto F_{n,\psi}(x,y) \pm \varepsilon \log(1+x^2+y^2)$ is thus bounded, uniformly in $n$, on the set $K_1$. To conclude, simply observe that, since the cut-off is moving to infinity, for sufficiently large $n$, we obtain $\chi_n \equiv \mathrm{id}$ on $K_1$.

We proceed with the proof of (b). For any $(x,y) \notin K_{2,n}$, we obtain 
\begin{align*}
F_{n,\psi}(x,y) + \varepsilon \log(1+x^2+y^2) &\geq  -\overline{C} + \frac{\varepsilon}{2} \log(1+x^2+y^2) + \frac{\varepsilon}{2} \log(1+x^2+y^2)\\
&\geq  \frac{\varepsilon}{2} \log(1+x^2+y^2) \\
&> \log \log b_n.
\end{align*}
The definition \eqref{eqn:cut-off} of the cut-off leads then to the conclusion. The proof for the function $F_{n,\psi}(x,y) - \varepsilon \log(1+x^2+y^2)$ follows similarly.
\end{proof}

\begin{lemma}\label{lemma:properties_fepsilon}
Suppose we are either in the setting of Theorem~\ref{thm:MD:paramagnetic:critical} with $\nu=2$ or in the setting of Theorem~\ref{thm:MD:paramagnetic:tricritical} with $\nu=4$. Let $\varepsilon \in (0,1)$ and $\psi \in C_c^\infty(\mathbb{R})$. Consider the cut-off \eqref{eqn:cut-off} and  define the functions
\begin{equation}\label{eqn:def:psi_n_varepsilon_plus_minus}
\psi_n^{\varepsilon,\pm} (x,y) := \chi_n \Big( F_{n,\psi}(x,y) \pm \varepsilon \log(1+x^2+y^2)\Big)  
\end{equation}
and
\[
\psi^{\varepsilon,\pm} (x,y) := \psi(x) \pm \varepsilon \log(1+x^2+y^2) ,
\]
with $F_{n,\psi}$ as in \eqref{eqn:def:perturbation}+\eqref{def:recursion:psi:phi}. Then, for every $\varepsilon \in (0,1)$, the following properties are satisfied:
\begin{enumerate}[(a)]
\item \label{lemma:item:fepsilon_to_a}
$\psi_n^{\varepsilon,\pm} \in \cD(H_n)$.
\item \label{lemma:item:fepsilon_to_b}
$\psi^{\varepsilon,+} \in C_l(\bR^2)$ and $\psi^{\varepsilon,-} \in C_u(\bR^2)$.
\item \label{lemma:item:fepsilon_to_c}
We have 
\[
\inf_n  \inf_{(x,y) \in E_n} \psi_n^{\varepsilon,+}(x,y) > - \infty \quad \mbox{ and } \quad \sup_n \sup_{(x,y) \in E_n} \psi_n ^{\varepsilon,-}(x,y) < \infty.
\]
\item \label{lemma:item:fepsilon_to_d}
For every compact set $K \subseteq \bR^2$, there exists a positive integer $N = N(K)$ such that, for $n \geq N$ and $(x,y) \in K$, we have 
\[
\psi_n^{\varepsilon,\pm}(x,y) =  F_{n,\psi}(x,y) \pm \varepsilon \log(1+x^2+y^2).
\]
\item \label{lemma:item:fepsilon_to_e} 
For every $c \in \bR$, we have
\[
\LIM_{n \uparrow \infty}  \psi_n^{\varepsilon,+} \wedge c  = \psi^{\varepsilon,+} \wedge c \quad \mbox{ and } \quad \LIM_{n \uparrow \infty}  \psi_n^{\varepsilon,-} \vee c = \psi^{\varepsilon,-} \vee c.
\]
\end{enumerate}
Moreover, it holds 
{\begin{enumerate}[(a)]\setcounter{enumi}{5}
\item \label{lemma:item:fepsilon_to_f} 
For every $c \in \bR$, we have 
\begin{equation*}
\lim_{\varepsilon \downarrow 0} \vn{\psi^{\varepsilon,+} \wedge c - \psi \wedge c} + \vn{\psi^{\varepsilon,-} \vee c - \psi \vee c} = 0.
\end{equation*} 
\end{enumerate}}
\end{lemma}
	
\begin{proof}
We prove all the properties for the `$+$' superscript case, the other being similar. 
\begin{enumerate}[(a)]
\item
%
As the cut-off \eqref{eqn:cut-off} is smooth, it yields $\psi_n^{\varepsilon,\pm} \in C^\infty(\bR^2)$. In addition, the location of the cut-off and Lemma~\ref{lmm:cutoff_functions}\eqref{item_2:lmm:cutoff_functions} make sure that $\psi_n^{\varepsilon,\pm}$ is constant outside a compact set $K \subset E_n$, implying $\psi_n^{\varepsilon,\pm} \in \cD(H_n)$. 
\item
This is immediate from the definitions of $\psi^{\varepsilon,\pm}$.
\item
Let $c >0$. From the uniform boundedness of $F_{n,\psi}$, we deduce
\[
\inf_{(x,y) \in \mathbb{R}^2} F_{n,\psi}(x,y) + \varepsilon \log(1+x^2+y^2) \geq - c + \varepsilon \log(1+x^2+y^2),
\]
which is bounded from below uniformly in $n$.
\item
This follows immediately by Lemma~\ref{lmm:cutoff_functions}\eqref{item_1:lmm:cutoff_functions}.
\item
Fix $\varepsilon > 0$ and $c \in \bR$. By \eqref{lemma:item:fepsilon_to_c}, the sequence $\left( \psi_n^{\varepsilon,+} \wedge c \right)_{n \in \mathbb{N}^*}$ is uniformly bounded from below and then, we obviously get $\sup_{n \in \mathbb{N}^*} \vn{\psi_{n}^{\varepsilon,+} \wedge c} < \infty$. Thus, it suffices to prove uniform convergence on compact sets. Let us consider an arbitrary sequence $(x_n,y_n)$ converging to $(x,y)$ and prove $\lim_n \psi_n^{\varepsilon,+}(x_n,y_n) = \psi^{\varepsilon,+}(x,y)$. As a converging sequence is bounded, it follows from \eqref{lemma:item:fepsilon_to_d} that, for sufficiently large $n$, we have 
\begin{equation*}
\psi_n^{\varepsilon,+}(x_n,y_n) = F_{n,\psi} (x_n,y_n) + \varepsilon \log(1+x_n^2+y_n^2),
\end{equation*}
which indeed converges to $\psi^{\varepsilon,+}(x,y)$ as $n \uparrow \infty$. See Lemma~\ref{lemma:convergence_of_functions_in_domain}. 
\item
This follows similarly as in the proof of \eqref{lemma:item:fepsilon_to_e}.
\end{enumerate}		
\end{proof}

\begin{definition}\label{definition:limiting_upper_lower_hamiltonians}
Suppose we are either in the setting of Theorem~\ref{thm:MD:paramagnetic:critical} with $\nu=2$ or in the setting of Theorem~\ref{thm:MD:paramagnetic:tricritical} with $\nu=4$. Let $H \subseteq C_b(\bR) \times C_b(\bR)$, with domain $\cD(H) = C_c^\infty(\bR)$, be defined as 
\begin{itemize}
\item
in the setting of Theorem~\ref{thm:MD:paramagnetic:critical} with $\nu=2$:
\begin{equation}\label{eqn:fixed_parameter_case_Hamiltonian_nu=2}
H(x,p) = \frac{2}{\cosh(\beta B)} p^2 + \frac{2}{3} \beta (2 \beta - 3) \cosh (\beta B) x^3 p;
\end{equation}

\item
in the setting of Theorem~\ref{thm:MD:paramagnetic:tricritical} with $\nu=4$:
\begin{equation}\label{eqn:fixed_parameter_case_Hamiltonian_nu=4}
H(x,p) = 2  \sqrt{\frac{2}{3}} p^2 - \frac{9}{10} \sqrt{\frac{3}{2}} \, x^5 p.
\end{equation}
\end{itemize}
We define the approximating Hamiltonians $H_\dagger \subseteq C_l(\bR^2) \times C_b(\bR^2)$ and $H_\ddagger \subseteq C_u(\bR^2) \times C_b(\bR^2)$ as 
\begin{align*}
H_\dagger & := \left\{ \left(  \psi(x) + \varepsilon \log(1+x^2+y^2), \, H\psi(x) + c(\varepsilon) \right) \, \middle| \, \psi \in C_c^\infty(\mathbb{R}), \varepsilon \in (0,1) \right\}, \\[0.1cm]
H_\ddagger & := \left\{ \left(  \psi(x) - \varepsilon \log(1+x^2+y^2), \, H\psi(x) - c(\varepsilon) \right) \, \middle| \, \psi \in C_c^\infty(\mathbb{R}), \varepsilon \in (0,1) \right\},
\end{align*}
with $c(\varepsilon) := 8 \left( \frac{\varepsilon}{2} \|\psi'\| +  \varepsilon^2 \right)$.
\end{definition}	
	
\begin{proposition} \label{proposition:limiting_hamiltonians}
Suppose we are either in the setting of Theorem~\ref{thm:MD:paramagnetic:critical} with $\nu=2$ or in the setting of Theorem~\ref{thm:MD:paramagnetic:tricritical} with $\nu=4$. Consider notation as in Definition~\ref{definition:limiting_upper_lower_hamiltonians}. We have $H_\dagger \subseteq ex-\subLIM_n  H_n$ and $H_\ddagger \subseteq ex-\superLIM_n  H_n$.
\end{proposition}

\begin{proof}
We prove only the first statement, i.e. $H_\dagger \subseteq ex-\subLIM_n  H_n$. Fix $\varepsilon > 0$ and $\psi \in C_c^\infty(\bR)$. We apply Lemma~\ref{lemma:properties_fepsilon} with $f_n := \psi_n^{\varepsilon,+}$. We show that $(\psi(x) + \varepsilon \log(1+x^2+y^2), H\psi(x)+c(\varepsilon))$ is approximated by $(f_n,H_nf_n)$ as in Definition \ref{definition:subLIM_superLIM}(a). Since \eqref{eqn:convergence_condition_sublim_constants} was proved in Lemma~\ref{lemma:properties_fepsilon}\eqref{lemma:item:fepsilon_to_e}, we are left to check conditions \eqref{eqn:convergence_condition_sublim_uniform_gn} and \eqref{eqn:sublim_generators_upperbound}.
\begin{itemize}
\item[{\footnotesize \eqref{eqn:convergence_condition_sublim_uniform_gn}}]
We start by showing that we can get a uniform (in $n$) upper bound for the function $H_n \psi_n^{\varepsilon,+}$. 
\begin{itemize}
\item
If $\vert F_{n,\psi}(x,y) +\varepsilon \log(1+x^2+y^2) \vert \geq \log \log b_n$, then the function $\psi_n^{\varepsilon,+}$ is constant and therefore $H_n \psi_n^{\varepsilon,+} \equiv 0$. \\
\item
If $\vert F_{n,\psi}(x,y) +\varepsilon  \log(1+x^2+y^2) \vert < \log \log b_n$, the variables $x$ and $y$ are at most of order $\log^{1/2} b_n$ and we can characterize $H_n \psi_n^{\varepsilon,+}$ by means of \eqref{eqn:Hn:compact_form:generic:paramagnetic}, since we can control the remainder term. Indeed, 
\begin{itemize}
\item 
the first, second and third order partial derivatives of $\psi \in C_c^{\infty}(\mathbb{R}^2)$ and $\log(1+x^2+y^2)$ are bounded, therefore by means of \eqref{eqn:remainder_control} we get control of the remainder up to order $\log^{1/2} b_n$ variables $x$ and $y$;\\
\item 
the function $\psi$ is constant outside a compact set and thus has zero derivatives outside such a compact set;\\
\item by smoothness of the cut-off \eqref{eqn:cut-off}, the derivative $\chi_n'$ is bounded between~$0$~and~$1$. 
\end{itemize}
We find 
\begin{multline}\label{eqn:Hamiltonian_uniform_estimate:fixed_parameter_case:1}
H_n \psi_n^{\varepsilon,+}(x,y) = \bigg\{-\frac{1}{15} \beta^3 \cosh(\beta B) (6 \beta -5) b_n^{\nu-4} x^5 \psi'(x) \\[0.1cm]
-\frac{2}{3} \beta \cosh(\beta B) (3-2\beta) b_n^{\nu-2} x^3 \psi'(x) + \frac{\varepsilon \, \Xi_n(x,y)}{30(1+x^2+y^2)} \bigg\} \chi_n'(-) \\[0.1cm]
+ \frac{2}{\cosh(\beta B)} \left[ \left( \psi'(x) \right)^2 + \frac{4 \varepsilon x \psi'(x)}{1+x^2+y^2}+ \frac{4\varepsilon^2(x^2+y^2)}{(1+x^2+y^2)^2}\right] \left( \chi_n'(-) \right)^2\\[0.1cm]
+ o(1) + o(b_n^{\nu-4}),
\end{multline}
with 
\begin{align*}
\Xi_n(x,y) =& -5 \beta^4 \cosh(\beta B) b_n^{\nu-4} x^4 y^2 -60 \beta^2 \cosh(\beta B) b_n^{\nu-2} x^2 y^2 \nonumber\\
& -120 \cosh(\beta B) b_n^{\nu} y^2 -40 \beta^3 \sinh(\beta B) b_n^{\nu-3} x^4 y \nonumber\\
& -240 \beta \sinh(\beta B) b_n^{\nu-1} x^2 y -4 \beta^4 \cosh(\beta B) b_n^{\nu-4} x^6 \nonumber\\
& -40 \beta^2 \cosh(\beta B) b_n^{\nu-2}x^4. 
\end{align*}
We want to show that \eqref{eqn:Hamiltonian_uniform_estimate:fixed_parameter_case:1} is uniformly bounded from above. We start by analyzing the terms in $\Xi_n(x,y)$. By completing the square, we can write
\begin{multline}\label{eqn:square_completion:1}
-120 \cosh(\beta B) b_n^{\nu} y^2 -240 \beta \sinh(\beta B) b_n^{\nu-1} x^2 y -40 \beta^2 \cosh(\beta B) b_n^{\nu-2}x^4 \\
= -120 \cosh(\beta B) b_n^{\nu-2} \left( b_n y + \beta \tanh(\beta B)x^2 \right)^2 \\
-120 \beta^2 \cosh(\beta B) \left( \tfrac{1}{3} - \tanh^2(\beta B)\right) b_n^{\nu-2} x^4.
\end{multline}
%
We take out $-40 \beta^2 \cosh(\beta B) b_n^{\nu-2} x^2 y^2$ from $-60 \beta^2 \cosh(\beta B) b_n^{\nu-2} x^2 y^2$ and, by the same trick as above, we also get
\begin{multline}\label{eqn:square_completion:2}
-40 \beta^2 \cosh(\beta B) b_n^{\nu-2} x^2 y^2 -40 \beta^3 \sinh(\beta B) b_n^{\nu-3} x^4 y -4 \beta^4 \cosh(\beta B) b_n^{\nu-4} x^6\\
= -40 \beta^2 \cosh (\beta B) b_n^{\nu-4} x^2 \left( b_n y + \tfrac{1}{2} \beta \tanh(\beta B) x^2 \right)^2  \\
 -2 \beta^4 \cosh (\beta B) \left(2 - 5\tanh^2(\beta B) \right) b_n^{\nu-4} x^6.
\end{multline}
Observe that both quantities \eqref{eqn:square_completion:1} and \eqref{eqn:square_completion:2} are non-positive, since $\beta \leq \frac{3}{2}$ implies $\frac{1}{3} - \tanh^2(\beta B) \geq 0$ and $2 - 5 \tanh^2(\beta B) \geq 0$. Putting all together yields
\begin{align}
\Xi_n(x,y) =& -5 \beta^4 \cosh(\beta B) b_n^{\nu-4} x^4 y^2 -20 \beta^2 \cosh(\beta B) b_n^{\nu-2} x^2 y^2 \nonumber\\
&-120 \cosh(\beta B) b_n^{\nu-2} \left( b_n y + \beta \tanh(\beta B)x^2 \right)^2 \nonumber\\
&-120 \beta^2 \cosh(\beta B) \left( \tfrac{1}{3} - \tanh^2(\beta B)\right) b_n^{\nu-2} x^4 \nonumber\\
&-40 \beta^2 \cosh (\beta B) b_n^{\nu-4} x^2 \left( b_n y + \tfrac{1}{2} \beta \tanh(\beta B) x^2 \right)^2  \nonumber\\
&-2 \beta^4 \cosh (\beta B) \left(2 - 5\tanh^2(\beta B) \right) b_n^{\nu-4} x^6, \label{eqn:Xi_fixed_parameter_case:after_manipulations}
\end{align}
which is then overall non-positive. Using also that $2x(1+x^2+y^2)^{-1} \leq 1$ and $(x^2+y^2)(1+x^2+y^2)^{-2} \leq 1$, we have
\begin{equation}\label{eqn:Hamiltonian_uniform_estimate:fixed_parameter_case:2}
H_n \psi_n^{\varepsilon,+}(x,y) \leq H\psi(x) + 8 \left( \frac{\varepsilon}{2} \| \psi' \| + \varepsilon^2\right) + o(1) + o(b_n^{\nu-4}),
\end{equation}
with $H$ as in \eqref{eqn:fixed_parameter_case_Hamiltonian_nu=2} if $\nu=2$ and as in \eqref{eqn:fixed_parameter_case_Hamiltonian_nu=4} if $\nu=4$ and $(\beta,B)=(\beta_{\mathrm{tc}},B_{\mathrm{tc}})$. In particular, as $\psi \in C_c^{\infty}(\mathbb{R})$ and we can control the remainder, there exists a positive constant $c_0$, independent of $n$ and $\varepsilon$, such that $H_n \psi_n^{\varepsilon,+}(x,y) \leq c_0$.
\end{itemize}
To conclude, observe that, since there exist positive constants $c_1$ and $c_2$ (independent of $n$) such that $\|H_n \psi_n^{\varepsilon,+}\| \leq c_1 b_n^{\nu} \log b_n + c_2$ (cf. equation \eqref{eqn:Hamiltonian_uniform_estimate:fixed_parameter_case:1}), choosing the sequence $v_n := b_n$ leads to $\sup_n v_n^{-1} \log \|H_n \psi_n^{\varepsilon,+}\| < +\infty.$

\item[{\footnotesize \eqref{eqn:sublim_generators_upperbound}}]
Let $K$ be a compact set. Consider an arbitrary converging sequence $(x_n, y_n) \in K$ and let $(x,y) \in K$ be its limit. We want to show $\limsup_n H_n \psi_n^{\varepsilon,+}(x_n,y_n) \leq H\psi(x)$.

As a converging sequence is bounded, by Lemma~\ref{lemma:properties_fepsilon}\eqref{lemma:item:fepsilon_to_d} we can find a sufficiently large $N=N(K) \in \mathbb{N}$ such that, for all $n \geq N$, we have 
\[
\psi_n^{\varepsilon,+} (x_n,y_n) = F_{n,\psi}(x_n,y_n) + \varepsilon \log(1+x_n^2+y_n^2).
\]
Thus, for any $n \geq N$, equation~\eqref{eqn:Hamiltonian_uniform_estimate:fixed_parameter_case:2} yields 
\[
H_n \psi_n^{\varepsilon,+}(x_n,y_n) \leq  H\psi(x) + 8 \left( \frac{\varepsilon}{2} \|\psi'\| +  \varepsilon^2 \right) + o(1) + o(b_n^{\nu-4}), 
\]
where the remainder terms  converge to zero uniformly on compact sets. Since $b_n \uparrow \infty$, the conclusion follows. 
\end{itemize}
\end{proof}

To conclude this section we obtain the Hamiltonian extensions.

\begin{proposition} \label{proposition:subsuper_extensions}
Consider notation as in Definition~\ref{definition:limiting_upper_lower_hamiltonians}. Moreover, set $\hat{H}_\dagger := H_\dagger \cup H$ and $\hat{H}_\ddagger := H_\ddagger \cup H$. Then $\hat{H}_\dagger$ is a sub-extension of $H_\dagger$ and $\hat{H}_\ddagger$ is a super-extension of $H_\ddagger$.
\end{proposition}

\begin{proof}
We prove only that $\hat{H}_\dagger$ is a sub-extension of $H_\dagger$. We use the first statement of Lemma \ref{lemma:extension_results}. Let $\psi \in \cD(H)$. We must show that $(\psi,H\psi)$ is appropriately approximated  by elements in the graph of $H_\dagger$.\\ 
For any $n \in \mathbb{N}^*$, set $\varepsilon(n) = n^{-1}$ and consider the function $\psi_n (x,y) = \psi(x) + \varepsilon(n) \log(1+x^2+y^2)$, with $H_\dagger \psi_n = H\psi + c(\varepsilon(n))$. From Lemma \ref{lemma:properties_fepsilon}\eqref{lemma:item:fepsilon_to_f}  we obtain that $\vn{\psi_n \wedge c - \psi \wedge c} \rightarrow 0$, for every $c \in \mathbb{R}$. In addition, as $H\psi \in C_b(\bR)$, we have $\vn{H_\dagger \psi_n - H\psi} \rightarrow 0$. This concludes the proof.
\end{proof}

\medskip

{\em Exponential compact containment.} The last open question we must address consists in verifying exponential compact containment for the fluctuation process. The validity of the compactness condition will be shown in Proposition~\ref{proposition:exponential_compact_containment}. We start by proving an auxiliary lemma.

\begin{lemma}\label{lemma:compact_containment_in_n}
Suppose we are either in the setting of Theorem \ref{thm:MD:paramagnetic:critical} and $\nu =2$ or in the setting of Theorem \ref{thm:MD:paramagnetic:tricritical} and $\nu = 4$. Let $G \subseteq \mathbb{R}^2$ be a relatively compact open set. Consider the cut-off introduced in \eqref{eqn:cut-off} and define
\[
\Upsilon_n (x,y) = \chi_n \left( \frac{1}{2}  \log(1+x^2+y^2) \right).
\]
Then, there exists a positive constant $c$ (independent of $n$) such that
\[
\limsup_{n \uparrow \infty} \sup_{(x,y) \in G \cap E_n} H_n \Upsilon_n (x,y) \leq \frac{2}{\cosh(\beta B)}.
\]
\end{lemma}

\begin{proof}
The proof is analogous to the verification of condition \eqref{eqn:sublim_generators_upperbound} in the proof of Proposition~\ref{proposition:limiting_hamiltonians}. Set $\psi \equiv 0$ and $\varepsilon = \frac{1}{2}$.
\end{proof}

\begin{proposition} \label{proposition:exponential_compact_containment}
Suppose we are either in the setting of Theorem \ref{thm:MD:paramagnetic:critical} and $\nu =2$ or in the setting of Theorem \ref{thm:MD:paramagnetic:tricritical} and $\nu = 4$. 

Moreover, assume that $(b_n m_n(0), b_n (q_n(0) - \tanh(\beta B)))$ is exponentially tight at speed $n b_n^{- \nu-2}$, then the process 
\[
Z_n(t) := (b_n m_n(b_n^\nu t), b_n (q_n(b_n^\nu t) - \tanh(\beta B)))
\]
satisfies the exponential compact containment condition at speed $n b_n^{- \nu-2}$. In other words, for every compact set $K \subseteq \bR^2$, every constant $a \geq 0$ and time $T \geq 0$, there exists a compact set $K' = K'(K,a,T) \subseteq \bR^2$ such that
\begin{equation*}
\limsup_{n \to \infty} \sup_{z \in K \cap E_n} n^{-1}b_n^{\nu+2} \log \PR\left[Z_n(t) \notin K' \text{ for some } t \leq T \, \middle| \, Z_n(0) = z \right] \leq - a.
\end{equation*}
\end{proposition}

\begin{proof}
The statement follows from Lemmas~\ref{lemma:compact_containment_in_n} and~\ref{lemma:compact_containment_FK} by choosing $f_n \equiv \Upsilon_n$ on a fixed, sufficiently large, compact set of $\mathbb{R}^2$. For similar proofs see e.g. \cite[Lem.~3.2]{DFL11} or \cite[Prop.~A.15]{CoKr17}.
\end{proof}

\begin{proof}[Proof of Theorems~\ref{thm:MD:paramagnetic:critical}~and~\ref{thm:MD:paramagnetic:tricritical}]
We check the assumptions of Theorem~\ref{theorem:Abstract_LDP}. We use operators $H_\dagger$, $H_\ddagger$ as in Definition~\ref{definition:limiting_upper_lower_hamiltonians} and limiting Hamiltonian $H \subseteq C_b(\bR) \times C_b(\bR)$, with domain $C_c^\infty(\bR)$, of the form $Hf(x) = H(x,f'(x))$ where
\begin{itemize}
\item
in the setting of Theorem~\ref{thm:MD:paramagnetic:critical} with $\nu=2$:
\[
H(x,p) = \frac{2}{\cosh(\beta B)} p^2 + \frac{2}{3} \beta (2 \beta - 3) \cosh (\beta B) x^3 p;
\]

\item
in the setting of Theorem~\ref{thm:MD:paramagnetic:tricritical} with $\nu=4$:
\[
H(x,p) = 2  \sqrt{\frac{2}{3}} p^2 - \frac{9}{10} \sqrt{\frac{3}{2}} \, x^5 p.
\]
\end{itemize}
We first verify Condition~\ref{condition:H_limit}: \eqref{condition:H_limit:item_1} follows from Proposition~\ref{proposition:limiting_hamiltonians}, \eqref{condition:H_limit:item_2} is satisfied by definition and \eqref{condition:H_limit:item_3} follows from Proposition~\ref{proposition:subsuper_extensions}.

The comparison principle for $f- \lambda Hf = h$ for $h \in C_b(\bR)$ and $\lambda > 0$ has been verified in e.g. \cite[Prop.~3.5]{CoKr17}. Note that the statement of the latter proposition is valid for $ f \in C_c^2(\bR)$, but the result generalizes straightforwardly to class $C_c^\infty(\bR)$ as the penalization and containment functions used in the proof are $C^\infty$.

Finally, the exponential compact containment condition follows from Proposition \ref{proposition:exponential_compact_containment}.
\end{proof}

\section{Variations in the external parameters} \label{section:projection_varying_betaB} 

Suppose we are either in the setting of Theorem~\ref{theorem:moderate_deviations_CW_critical_curve_rescaling} with $\nu=2$ or in the setting of Theorems~\ref{theorem:moderate_deviations_CW_tri-critical_point_rescaling} and \ref{theorem:moderate_deviations_CW_tri-critical_point_rescaling:2} with $\nu=4$. The major difference of these theorems compared to Theorems \ref{thm:MD:paramagnetic:critical} and \ref{thm:MD:paramagnetic:tricritical} is the variation in the parameters $\beta$ and $B$ as the system size increases. The inverse temperature and the magnetic field are respectively $\beta_n := \beta + \kappa_n$ and $B_n := B + \theta_n$, where $\{\kappa_n\}_{n \geq 1}$ and $\{\theta_n\}_{n \geq 1}$ are real sequences converging to zero.\\ 
In this more general framework, due to an extra Taylor expansion in $\beta$ and $B$, the Hamiltonian \eqref{eqn:Hn:compact_form:generic:paramagnetic} changes into
\begin{multline}\label{def:varying_parameters_Hamiltonian}
\sum_{l \geq 0} \sum_{j \geq 0} \sum_{k=1}^5 \frac{\kappa_n^l}{l!} \frac{\theta_n^j}{j!} \frac{b_n^{\nu + 1 -k}}{k!}\ip{\partial_\kappa^l \partial_\theta^j D^k \cG_{2}(0,\tanh(\beta B)) \begin{pmatrix} x^k \\ k x^{k-1}y \end{pmatrix}}{\nabla f(x,y)}\\
+ \ip{\bG_{1}(0,\tanh(\beta B)) \nabla f(x,y)}{\nabla f(x,y)} + o (1)+ o(b_n^{\nu - 4})
\end{multline}
and 
determining the terms that contribute to the limiting operator
is trickier than before. In the linear part of \eqref{eqn:Hn:compact_form:generic:paramagnetic} the operators appearing with a factor $b_n^{k}$ introduce terms with $x^k$; whereas, now this is no longer the case. Operators with pre-factor $b_n^{k}$ may introduce terms with $x^m$ for $m \leq k$ (the power $m$ depends on the order of $\kappa_n$ and $\theta_n$ with respect to $b_n$). Therefore, we need to extend the method presented in Section \ref{section:formal_calculus_of_operators} appropriately.

\subsection{Extending the formal calculus of operators}\label{sect:formal_calculus:extended}

Let $V$ and $V_i$, with $i \in \mathbb{N}$, be the vector spaces of functions introduced at the beginning of Section~\ref{section:formal_calculus_of_operators}. We are going to define an alternative set of operators on $V$. Let $a \in \mathbb{R}$ and $g: \mathbb{R}^2 \rightarrow \mathbb{R}$ be a differentiable function. Fix $k \in \mathbb{N}$ and consider the array of operators
\begin{equation}\label{def:calQ:even_extended}
\left.
\begin{array}{l}
\mathcal{Q}_{k,m}^+[a] g(x,y) := a x^{m-1}y g_x(x,y) \\[0.2cm]
\mathcal{Q}_{k,m}^-[a] g(x,y) := a x^m g_y(x,y)
\end{array}
\quad
\right]
\text{ for even $m$ and $m \leq k$}
\end{equation}
and
\begin{equation}\label{def:calQ:odd_extended}
\left.
\begin{array}{l}
\mathcal{Q}_{k,m}^0[a] g(x,y) := a x^m g_x(x,y) \\[0.2cm]
\mathcal{Q}_{k,m}^1[a] g(x,y) := a x^{m-1} y g_y(x,y)
\end{array}
\quad\,
\right]
\text{ for odd $m$ and $m \leq k$.}
\end{equation}

We have the direct analogue of Lemma \ref{lemma:mappings_wrt_Vindex}.

\begin{lemma} \label{lemma:mappings_wrt_Vindex_extended}
For all $a \in \bR$ and $k, i \in \bN$, we have 
\[
\mathcal{Q}_{k,m}^+[a] : V_i \rightarrow V_{i+1} \text{ and } \mathcal{Q}_{k,m}^-[a] : V_i \rightarrow V_{i-1},
\text{ for even $m$,}
\]
and
\[
\mathcal{Q}_{k,m}^0[a], \mathcal{Q}_{k,m}^1[a] : V_i \rightarrow V_i, \text{ for odd $m$}.
\]
\end{lemma}

Notice that also in this extended setting the operators with superscript $1$ (i.e., $Q_{k,m}^1$ with odd $m$ and $m \leq k$) have the peculiarity of admitting $V_0$ as a kernel.

\begin{assumption} \label{assumption:abstract_sequences_extended}
Assume there exist real constants $a_{k,m}^+$, $a_{k,m}^-$ if $m$ is even and \mbox{$1 \leq m \leq k$} and  $a_{1,1}^0=0$, $a_{1,1}^1$, $a_{k,m}^0$, $a_{k,m}^1$ if $m$ is odd and $1 < m \leq k$, for which, given a continuously differentiable function $g: \mathbb{R}^2 \rightarrow \mathbb{R}$, we can write
\begin{equation}\label{def:Qk's:extended}
\mathcal{Q}_{k} g = \sum_{\substack{m \leq k \\ m \text{ even}}} \left( \mathcal{Q}_{k,m}^+ g + \mathcal{Q}_{k,m}^- g \right) + \sum_{\substack{m \leq k \\ m \text{ odd}}} \left(\mathcal{Q}_{k,m}^0 g + \mathcal{Q}_{k,m}^1 g \right)
\end{equation}
with
\begin{equation*}
\left.
\begin{array}{l}
\mathcal{Q}_{k,m}^+ g (x,y) := \mathcal{Q}_{k,m}^+[a_{k,m}^+] g (x,y) \\[0.2cm]
\mathcal{Q}_{k,m}^- g (x,y) := \mathcal{Q}_{k,m}^-[a_{k,m}^-] g (x,y)
\end{array}
\quad
\right]
\text{ for even $m$ and $m \leq k$}
\end{equation*}
and
\begin{equation*}
\left.
\begin{array}{l}
\mathcal{Q}_{k,m}^0 g (x,y) := \mathcal{Q}_{k,m}^0[a_{k,m}^0] g (x,y) \\[0.2cm]
\mathcal{Q}_{k,m}^1 g (x,y) := \mathcal{Q}_{k,m}^1[a_{k,m}^1] g (x,y)
\end{array}
\quad\,
\right]
\text{ for odd $m$ and $m \leq k$.}
\end{equation*}
\end{assumption}

Observe that we recover Assumption \ref{assumption:abstract_sequences} if $a_{k,m}^{z} = 0$ whenever $k \neq m$ and set $\cQ_{k,k}^{z} := \cQ_k^{z}$ for appropriate $z \in \{+,-,1,0\}$. 

\smallskip

Using our new definitions, Lemma~\ref{lemma:kernel_operator_decomp} and the recursion relationships  \eqref{def:recursion:psi:phi} are unchanged and furthermore, the result of Proposition \ref{proposition:perturbation_abstract_Q} is still valid. The main modification is that we need to re-evalute the functions $P_0 \phi[i]$ as the $\cQ_k$ are defined by using a larger set of operators. We get the following two statements.

\begin{proposition} \label{proposition:perturbation_abstract_Q_extended}
Fix $\nu \geq 2$ an even natural number and suppose that Assumption \ref{assumption:abstract_sequences_extended} holds true for this $\nu$. Consider the operator
\begin{equation}\label{eqn:operator:calQ_extended}
\cQ^{(n)} \psi(x,y) := \sum_{k=1}^{\nu+1} b_n^{\nu+1 -k} \cQ_k \psi(x,y)
\end{equation}
and, for $\psi = \psi[0] \in V_0$, define $F_{n,\psi}(x,y) := \sum_{l=0}^\nu b_n^{-l} \psi[l](x,y)$. We have
\begin{equation*}
\cQ^{(n)} F_{n,\psi}(x,y) = \sum_{i = 1}^\nu b_n^{\nu - i} P_0 \phi[i](x) + o(1),
\end{equation*}
where $o(1)$ is meant according to Definition~\ref{def:o(1)}.
\end{proposition}

We can evaluate the functions $P_0 \phi[i]$ as we did in Lemma~\ref{lemma:evaluation_P0_phi_l}. 
Under the more general Assumption~\ref{assumption:abstract_sequences_extended}, more terms survive the infinite volume limit.
We calculate the outcomes only for the cases we will need below.

\begin{lemma} \label{lemma:evaluation_P0_phi_l_extended_specific}
Consider the setting of Proposition \ref{proposition:perturbation_abstract_Q_extended}. For $\psi = \psi[0] \in V_0$, we have $P_0 \phi[l] = 0$ if $l$ is odd and 
\begin{equation}\label{lemma:projection_on_kernel_extended}
P_0 \phi[l] = \begin{cases}
\mathcal{Q}_{3,3}^0 \psi + \mathcal{Q}_{2,2}^- P \mathcal{Q}_{2,2}^+ \psi + \mathcal{Q}_{3,1}^0 \psi  & \text{if } l =2, \\[0.3cm]
\mathcal{Q}_{5,5}^0 \psi + \mathcal{Q}^-_{2,2} P \mathcal{Q}_{4,4}^+ \psi + \mathcal{Q}_{4,4}^- P \mathcal{Q}_{2,2}^+ \psi + \mathcal{Q}_{2,2}^- P\mathcal{Q}_{3,3}^1P \mathcal{Q}^+_{2,2} \psi  \\
+ \mathcal{Q}_{5,3}^0 \psi + \mathcal{Q}^-_{2,2} P \mathcal{Q}_{4,2}^+ \psi + \mathcal{Q}_{4,2}^- P \mathcal{Q}_{2,2}^+ \psi + \mathcal{Q}_{2,2}^- P\mathcal{Q}_{3,1}^1P \mathcal{Q}^+_{2,2} \psi  \\
\:\: + \mathcal{Q}_{5,1}^0 \psi + + \mathcal{Q}_{2,2}^- P( \mathcal{Q}_{3,3}^0 + \mathcal{Q}_{2,2}^- P \mathcal{Q}_{2,2}^+ + \cQ_{3,1}^0)P \mathcal{Q}^+_{2,2} \psi & \text{if } l =4.
\end{cases}
\end{equation}
\end{lemma}

Following Conjecture \ref{conjecture:form_of_drift}, we can make a similar conjecture in this extended setting as well.

\begin{conjecture} \label{conjecture:form_of_drift_extended}
Let Assumption \ref{assumption:abstract_sequences_extended} be satisfied with $\nu$ even. Assume that $a_{k,l}^z=0$ for $k \neq l \text{ mod } 2$ and all $z \in \{+,-,0,1\}$. Moreover, suppose that $P_0\phi[2l] = 0$ for all $l \in \mathbb{N}$ with $2l < \nu$. Then
{\small
\begin{multline*}
P_0 \psi[\nu](x,y) = o(1)  \\ 
+ \sum_{\substack{0 \leq \mu \leq \nu \\ \mu \text{ even}}}\left[a_{\nu+1,\mu + 1}^0 +  \sum_{n \geq 2} \sum_{\substack{i_1 + \dots + i_n = \nu + n \\  i_j \neq 1 \text{ for } j \notin \{1,n\} \\  r_1 + \dots + r_n = \mu + n \\ r_j \leq i_j \, \text{all} \, j \\ r_1, r_n \text{ even} \\ r_j \text{ odd for } j \notin \{1,n\} }} (-1)^{n-1}\frac{a_{i_1,r_1}^- \left(\prod_{j = 2}^{n-1} a_{i_j,r_j}^1 \right)a_{i_n,r_n}^+}{(a_{1,1}^1)^{n-1}} \right] x^{\mu +1} \psi_x(x).
\end{multline*}
}
\end{conjecture}

\subsection{Preliminaries for the proofs of Theorems~\ref{theorem:moderate_deviations_CW_critical_curve_rescaling}--\ref{theorem:moderate_deviations_CW_tri-critical_point_rescaling:2}} \label{section:preliminaries_explicit_expansions_variation_parameters}

As we did before, we now connect the discussion of Section~\ref{sect:formal_calculus:extended} with the proofs of Theorems \ref{theorem:moderate_deviations_CW_critical_curve_rescaling}--\ref{theorem:moderate_deviations_CW_tri-critical_point_rescaling:2} via Theorem \ref{theorem:Abstract_LDP}. Recall that our purpose is to find an operator $H \subseteq C^\infty_c(\bR) \times C^\infty_c(\bR)$ such that $H \subseteq ex-\LIM H_n$. In other words, for $f \in \cD(H)$, we need to determine $f_n \in H_n$ such that $\LIM f_n = f$ and $\LIM H_n f_n = H f$.\\
Consider the statement of Proposition \ref{proposition:compact_expression_for_H_n}. We want to find the limit of the operator $H_n$ presented there. We analyze term by term.  If $(m,q) = (0,\tanh(\beta B))$ the term in \eqref{eqn:compact_exp_G_1} vanishes. The very same proof as the one of Lemma~\ref{lemma:convergence_of_functions_in_domain} gives the next lemma implying that the term in \eqref{eqn:compact_exp_G_3} converges as a consequence of the uniform convergence of the gradients.

\begin{lemma}\label{lemma:convergence_of_functions_in_domain:extended}
	Suppose we are either in the setting of Theorem~\ref{theorem:moderate_deviations_CW_critical_curve_rescaling} with $\nu=2$ or in the setting of Theorems~\ref{theorem:moderate_deviations_CW_tri-critical_point_rescaling} and \ref{theorem:moderate_deviations_CW_tri-critical_point_rescaling:2} with  $\nu=4$. For $\psi \in C_c^\infty(\mathbb{R})$, define the approximation
	\begin{equation}\label{eqn:perturbation:LDP_statements:extended}
	F_{n,\psi}(x,y) := \sum_{l=0}^\nu b_n^{-l} \psi[l](x,y),
	\end{equation}
	where $\psi[\cdot]$ are defined recursively according to \eqref{def:recursion:psi:phi} with the $\mathcal{Q}_k$'s given by \eqref{def:Qk's:extended}. Moreover, let $R := [a,b] \times [c,d]$, with $a < b$ and $c < d$, be a rectangle in $\mathbb{R}^2$. Then, we have $F_{n,\psi} \in C_c^\infty(\bR^2)$, $\LIM F_{n,\psi} = \psi$ and 
	\begin{equation} \label{eqn:convergence_gradients_Psi:extended}
	\sup_{(x,y) \in R \cap E_n} |\nabla F_{n,\psi}(x,y) - \nabla \psi(x)| = 0
	\end{equation}
	for all rectangles $R \subseteq \bR^2$.
\end{lemma}

For the term in \eqref{eqn:compact_exp_G_2}, we use the results from Section~\ref{sect:formal_calculus:extended}. At this point the proofs of Theorem \ref{theorem:moderate_deviations_CW_critical_curve_rescaling} and \ref{theorem:moderate_deviations_CW_tri-critical_point_rescaling} differ from the proof of Theorem \ref{theorem:moderate_deviations_CW_tri-critical_point_rescaling:2} in the sense that in the first case $\kappa_n,\theta_n$ are of order $b_n^{-2}$, whereas in the latter $\kappa_n,\theta_n$ are of order $b_n^{-4}$. Therefore, the connection between the linear part in \eqref{def:varying_parameters_Hamiltonian} and the operators in Assumption \ref{assumption:abstract_sequences_extended} changes. To give an explicit example: $\cQ_{5,1}^0$ is a different operator in the two settings.

Using \eqref{lemma:projection_on_kernel_extended} of Lemma~\ref{lemma:evaluation_P0_phi_l_extended_specific}, 
we calculate the drift of the limiting Hamiltonians by considering the relevant operators in the two sections below.


\subsection{Proof of Theorems \ref{theorem:moderate_deviations_CW_critical_curve_rescaling} and \ref{theorem:moderate_deviations_CW_tri-critical_point_rescaling}}\label{subsect:proofs_thms_rescaling}

In the setting of Theorems \ref{theorem:moderate_deviations_CW_critical_curve_rescaling} and \ref{theorem:moderate_deviations_CW_tri-critical_point_rescaling}, $\kappa_n$ and $\theta_n$ are of order $b_n^{-2}$. We identify the relevant operators for Assumption \ref{assumption:abstract_sequences_extended} from the linear part in the expansion in \eqref{def:varying_parameters_Hamiltonian}. First of all, there are the operators that do not involve derivations in the $\kappa$ and $\theta$ directions. These are the operators we also considered for Theorems \ref{thm:MD:paramagnetic:critical} and \ref{thm:MD:paramagnetic:tricritical}. Turning to our extended notation, we find for $k \in \{1,\dots,5\}$ and appropriate $z \in \{+,-,0,1\}$, the operators $Q_{k,k}^{z} = Q_k^{z}$ as defined in \eqref{def:Q1}--\eqref{def:Q5}.

\smallskip

Additional operators are being introduced by the differentiations in the $\theta,\kappa$ directions. In particular, the relevant operators are
\begin{enumerate}[(a)]
	\item $Q_{3,1}^0$ and $Q_{3,1}^1$ arising from the first and second coordinate of 
	\begin{equation*}
	(\partial_{\kappa} + \partial_{\theta}) D^1 \cG_2(0,\tanh(\beta B)) \begin{pmatrix}
	x \\ y
	\end{pmatrix};
	\end{equation*}
	\item $Q_{4,2}^+$ and $Q_{4,2}^-$ arising from the first and second coordinate of 
	\begin{equation*}
	\frac{1}{2}(\partial_{\kappa} + \partial_{\theta}) D^2 \cG_2(0,\tanh(\beta B)) \begin{pmatrix}
	x^2 \\ 2 xy
	\end{pmatrix};
	\end{equation*}
	\item $Q_{5,3}^0$ arising from the first coordinate of 
	\begin{equation*}
	\frac{1}{6}(\partial_{\kappa} + \partial_{\theta}) D^3 \cG_2(0,\tanh(\beta B)) \begin{pmatrix}
	x^3 \\ 3 x^2y
	\end{pmatrix};
	\end{equation*}
	\item $Q_{5,1}^0$ arising from the first coordinate of 
	\begin{equation*}
	\left(\frac{1}{2} \partial_{\kappa}^2 + \partial_{\kappa}\partial_{\theta} + \frac{1}{2} \partial_\theta^2 \right) D^1 \cG_2(0,\tanh(\beta B)) \begin{pmatrix}
	x^2 \\ 2 xy
	\end{pmatrix}.
	\end{equation*}
\end{enumerate}

We explicitly calculate the relevant operators from the results in Lemma~\ref{lemma:DkG2_expressions}. We start with $Q_{3,1}^0$ which is used for Theorem \ref{theorem:moderate_deviations_CW_critical_curve_rescaling}. From Lemma~\ref{lemma:DkG2_expressions}(c) we get
\begin{equation}
Q_{3,1}^0 g(x,y)  = \left[ \frac{2}{\cosh(\beta B)}  \left(1 - 2 \beta B \tanh(\beta B)\right) \kappa - 4 \beta \sinh (\beta B)  \theta \right] x g_x(x,y), \label{eqn:Q31_0}
\end{equation}
after using the identity $\beta = \cosh^2(\beta B)$ for $(\beta,B)$ on the critical curve. We proceed now with the operators needed for Theorem \ref{theorem:moderate_deviations_CW_tri-critical_point_rescaling}. In this setting, $(\beta_n,B_n)$ lies always on the critical curve, i.e. $\beta_n = \cosh^2(\beta_n B_n)$ for any $n \in \mathbb{N}$, and therefore we use Lemma~\ref{lemma:DkG2_expressions}(d) to compute the $D^k \mathcal{G}_2$'s.  It yields $Q_{3,1}^0 = Q_{5,1}^0 = 0$ and for the remaining operators we find:
\begin{align}
Q_{3,1}^1 g(x,y) & = - 2 \sinh (\beta B) \left[ B \kappa + \beta \theta\right] y g_y(x,y),\label{eqn:Q31_1} \\[0.3cm]
Q_{4,2}^+ g(x,y) & = -2 \big[ \left( \sinh(\beta B) + \beta B \cosh (\beta B)\right) \kappa + \beta^2 \cosh (\beta B) \theta \big] xy g_x(x,y), \label{eqn:Q42_+} \\[0.3cm]
Q_{4,2}^- g(x,y) & = -2 \big[ \left( \sinh(\beta B) + \beta B \cosh (\beta B)\right) \kappa + \beta^2 \cosh (\beta B) \theta \big] x^2 g_y(x,y), \label{eqn:Q42_-} \\[0.3cm]
Q_{5,3}^0 g(x,y) & = - \frac{2}{3} \big[ \beta \left(2 \cosh(\beta B) + \beta B \sinh(\beta B) \right) \kappa + \beta^3 \sinh(\beta B) \theta\big] x^3 g_x(x,y).\label{eqn:Q53_0} 
\end{align}

In the next lemma we calculate the expressions resulting from the concatenations of $P$'s and $Q$'s given in \eqref{lemma:projection_on_kernel_extended}. The action of $P$ is described in Lemma~\ref{lemma:kernel_operator_decomp}. 

\begin{lemma}\label{lemma:limiting_drifts:extended}
	Let $(\beta, B)$ satisfies $\beta = \cosh^2 (\beta B)$. For $f \in V$, we have
	\begin{multline}\label{eqn:drift:critical_curve:rescaling}
	(Q_{2,2}^- P Q_{2,2}^+  + Q_{3,3}^0  + Q_{3,1}^0) f(x,y) = \frac{2}{3} \beta (2\beta-3)\cosh(\beta B) x^3 f_x(x,y) \\ 
	+ 2 \left[ \frac{1- 2 \beta B \tanh (\beta B)}{\cosh(\beta B)} \, \kappa - 2 \beta \sinh (\beta B) \, \theta\right] x f_x(x,y).
	\end{multline}
	Moreover, if we approach the tri-critical point $(\beta_{\mathrm{tc}}, B_{\mathrm{tc}})$ along the critical curve, we obtain $(Q_{2,2}^- P Q_{2,2}^+  + Q_{3,3}^0 + Q_{3,1}^0) f = 0$ and 
	%
		\begin{align}\label{eqn:drift:tri-critical_point:rescaling}
	(Q_{5,5}^0  &+ Q_{5,3}^0 + Q_{5,1}^0 + Q^-_{2,2} P Q_{4,4}^+ + Q_{4,4}^- P Q_{2,2}^+ + Q^-_{2,2} P Q_{4,2}^+ + Q_{4,2}^- P Q_{2,2}^+   \nonumber\\
	& + Q_{2,2}^- P Q_{3,3}^1P Q^+_{2,2}   + Q_{2,2}^- P Q_{3,1}^1P Q^+_{2,2} )f(x,y) \nonumber\\
	&=  - \frac{9}{10} \sqrt{\frac{3}{2}} \, x^5 f_x(x,y) + \left[ 2 \sqrt{2} \arccosh \left(\sqrt{\frac{3}{2}}\right) \kappa +\frac{9}{\sqrt{2}} \, \theta \right] x^3 f_x(x,y).
	\end{align}	
		
	%
\end{lemma}

\begin{proof}
	It suffices to prove the statement for $f$ of the form $y^i g(x)$, for some function $g \in C^2(\mathbb{R})$.  The term $Q_{3,1}^0 f$ is given in \eqref{eqn:Q31_0} and the expression for $(Q_{2,2}^- P Q_{2,2}^+  + Q_{3,3}^0)f$ in \eqref{eqn:drift:critical_curve}. Combining these two results yields \eqref{eqn:drift:critical_curve:rescaling}.\\
	Lemma~\ref{lemma:limiting_drifts} and the observation that $Q_{3,1}^0 = 0$ on the critical curve imply that $(Q_{2,2}^- P Q_{2,2}^+  + Q_{3,3}^0 + Q_{3,1}^0) f = 0$ whenever $(\beta,B) = (\beta_{\mathrm{tc}},B_{\mathrm{tc}})$ and $\beta_n = \cosh^2(\beta_n B_n)$ for any $n \in \mathbb{N}$.\\
	We are left to show \eqref{eqn:drift:tri-critical_point:rescaling}. We start by stating the relevant constants
	\begin{gather*}
	a_{1,1}^1  = - 2 \cosh(\beta B), \qquad a_{2,2}^{\pm} = - 2 \beta \sinh(\beta B), \qquad	a^1_{3,1}  = - 2 \sinh(\beta B)[B \kappa + \beta \theta], \\
	 a_{4,2}^\pm = - 2\left[\left( \sinh(\beta B) + \beta B \cosh (\beta B)\right) \kappa  + \beta^2 \cosh (\beta B) \theta\right]. 
	\end{gather*}
	 Observe that
	\begin{itemize}
		\item 
		the expression for $(Q_{5,5}^0  + Q^-_{2,2} P Q_{4,4}^+ + Q_{4,4}^- P Q_{2,2}^+  + Q_{2,2}^- P Q_{3,3}^1P Q^+_{2,2})f$ is given in \eqref{eqn:drift:tri-critical_point};
		\item
		the operator $Q_{5,1}^0 = 0$ as we are on the critical curve;
		\item 
		$Q_{5,3}^0 f$ is defined in \eqref{eqn:Q53_0}; 
		\item by direct computation, or by using a variant of Lemma \ref{lemma:combo-P+}, we get
		\begin{align*}
		(Q^-_{2,2} P Q_{4,2}^+)f(x,y) &= (Q_{4,2}^- P Q_{2,2}^+)f(x,y) \\
		&= 2 \beta \sinh (\beta B) \left[ (\tanh(\beta B) + \beta B) \kappa + \beta^2 \theta\right] x^3 f_x(x,y)
		\end{align*}
		and
		\[
		(Q_{2,2}^- P Q_{3,1}^1P Q^+_{2,2})f(x,y) = -2 \beta^2 \sinh(\beta B) \tanh^2 (\beta B) \left[ B \kappa + \beta \theta \right] x^3 f_x(x,y).
		\]
	\end{itemize}
	Adding the contributions above gives  
	\begin{multline*}
	\left(Q_{5,5}^0  + Q^-_{2,2} P Q_{4,4}^+ + Q_{4,4}^- P Q_{2,2}^+  + Q_{2,2}^- P Q_{3,3}^1P Q^+_{2,2}  \right. \\
	\left. + Q_{5,3}^0 + Q_{5,1}^0 + Q^-_{2,2} P Q_{4,2}^+	+ Q_{4,2}^- P Q_{2,2}^+  + Q_{2,2}^- P Q_{3,1}^1P Q^+_{2,2} \right)f(x,y)\\
	= - \frac{1}{15} \beta^4 \cosh(\beta B) \big[ 5 \tanh^2(\beta B) + 1 \big] x^5 f_x(x,y) \\
	+ 2 \beta \sinh(\beta B) \left( \frac{2}{3} \beta + 1\right) \left[ B \kappa + \beta \theta \right] x^3 f_x(x,y).
	\end{multline*}
	Plugging the values $\beta = \beta_{\mathrm{tc}} = \frac{3}{2}$ and $B = B_{\mathrm{tc}} = \frac{2}{3} \arccosh (\sqrt{\frac{3}{2}})$ leads to the conclusion.
\end{proof}

\begin{proof}[Proof of Theorems~\ref{theorem:moderate_deviations_CW_critical_curve_rescaling} and \ref{theorem:moderate_deviations_CW_tri-critical_point_rescaling}]
	
The proof follows the proof of Theorems~\ref{thm:MD:paramagnetic:critical} and~\ref{thm:MD:paramagnetic:tricritical}. We highlight the differences. 
	
Due to the variations in $\beta$ and $B$, additional drift terms are introduced. These are given in Lemma~\ref{lemma:limiting_drifts:extended}. Therefore, we work with the following Hamiltonians
\begin{itemize}
\item
in the setting of Theorem~\ref{theorem:moderate_deviations_CW_critical_curve_rescaling} with $\nu=2$:
\begin{multline}\label{eqn:varying_parameter_case_Hamiltonian_nu=2}
H(x,p) = \frac{2}{\cosh(\beta B)} p^2 + 2 \bigg\{ \left[ \frac{1- 2 \beta B \tanh (\beta B)}{\cosh(\beta B)} \, \kappa - 2 \beta \sinh (\beta B) \, \theta\right] x \\
+ \frac{\beta}{3}  (2\beta-3)\cosh(\beta B) x^3 \bigg\} p;
\end{multline}
\item
in the setting of Theorem~\ref{theorem:moderate_deviations_CW_tri-critical_point_rescaling} with $\nu=4$:
\begin{equation}\label{eqn:varying_parameter_case_Hamiltonian_nu=4}
H(x,p) = 2 \sqrt{\frac{2}{3}} p^2 + \left\{\left[  2\sqrt{2} \arccosh \left(\sqrt{\frac{3}{2}}\right) \kappa +\frac{9}{\sqrt{2}} \, \theta \right] x^3 - \frac{9}{10} \sqrt{\frac{3}{2}} \, x^5 \right\}p.
\end{equation}
\end{itemize}
The presence of extra drift terms, involving $\theta$ and $\kappa$, makes the verification of condition \eqref{eqn:convergence_condition_sublim_uniform_gn} in Definition~\ref{definition:subLIM_superLIM} slightly more involved. \\
Consider the sequence of functions $\psi_n^{\varepsilon,+}$ defined as \eqref{eqn:def:psi_n_varepsilon_plus_minus}, where the operators $\mathcal{Q}_k$'s used to construct $F_{n,\psi}$ are given by \eqref{def:Qk's:extended} now. Recall equation \eqref{eqn:Hn:compact_form:generic:paramagnetic}. We want to show that, on the set $E_{4,n} = \{(x,y) \in \mathbb{R}^2 \, | \, | F_{n,\psi}(x,y) + \varepsilon \log(1+x^2+y^2)| < \log \log b_n \}$, we can obtain $\sup_n H_n \psi_n^{\varepsilon,+} < \infty$, uniformly in $n$. Observe that the cut-off guarantees $H_n \psi_n^{\varepsilon,+} \equiv 0$ on $E_{4,n}^c$.\\
In particular, on $E_{4,n}$ the variables $x$ and $y$ are at most of order $\log^{1/2} b_n$ and we can get control of the remainder terms in \eqref{eqn:Hn:compact_form:generic:paramagnetic} via Lemma~\ref{lemma:remainder_control}. Therefore,
following the exact same strategy as in the proof of Proposition~\ref{proposition:limiting_hamiltonians}, we find 
\begin{multline}\label{eqn:Hamiltonian_uniform_estimate:varying_parameter_case}
H_n \psi_n^{\varepsilon,+}(x,y) \leq H\psi(x) + \frac{\varepsilon }{30(1+x^2+y^2)}\, \big(\Xi_n(x,y)+\Theta_n(x,y) \big) \\
+ 8\left( \frac{\varepsilon}{2} \|\psi'\| + \varepsilon^2 \right) + o(1) +o(b_n^{\nu-4}),
\end{multline}
with $H$ as in \eqref{eqn:varying_parameter_case_Hamiltonian_nu=2} if $\nu=2$ and as in \eqref{eqn:varying_parameter_case_Hamiltonian_nu=4} if $\nu=4$ and $(\beta, B) = (\beta_{\mathrm{tc}},B_{\mathrm{tc}})$, with $\Xi_n(x,y)$ given in \eqref{eqn:Xi_fixed_parameter_case:after_manipulations} and where
\begin{align}
\Theta_n(x,y) =&-120 (\beta\theta + B\kappa) \sinh(\beta B) b_n^{\nu-2} y^2 -60 (\beta\theta + B\kappa)^2 \cosh(\beta B) b_n^{\nu-4} y^2 \nonumber \\
&-60 \beta \left[2\kappa \cosh(\beta B)  +\beta (\beta\theta + B\kappa) \sinh(\beta B) \right] b_n^{\nu-4} x^2 y^2 \nonumber \\
&-240 \left[ \kappa \sinh(\beta B)  +\beta (\beta\theta + B\kappa) \cosh(\beta B)\right] b_n^{\nu-3} x^2 y \nonumber \\
&-40 \beta \left[ 2\kappa \cosh(\beta B)  +\beta (\beta\theta + B\kappa) \sinh(\beta B)  \right] b_n^{\nu-4} x^4 \nonumber \\
&+120 \left[ \frac{1-2\beta B \tanh (\beta B)}{\cosh(\beta B)} \, \kappa -2 \beta \sinh(\beta B) \theta \right] b_n^{\nu-2} x^2. \label{eqn:Theta_varying_parameter_case}
\end{align}
We want to see that \eqref{eqn:Hamiltonian_uniform_estimate:varying_parameter_case} admits a uniform upper bound. Since $\psi \in C_c^{\infty}(\mathbb{R})$ and $\varepsilon \in (0,1)$, it suffices to show that the function  $(1+x^2+y^2)^{-1} (\Xi_n(x,y)+\Theta_n(x,y))$ is uniformly bounded from above. 

If $\nu=2$ the result is straightforward; indeed, $\Xi_n(x,y) \leq 0$ and $(1+x^2+y^2)^{-1}\Theta_n(x,y)$ is bounded. 
We take now $\nu=4$. First of all, the term $120 [\cdots] b_n^{\nu-2} x^2$ in \eqref{eqn:Theta_varying_parameter_case} vanishes when we are on the critical curve (use \eqref{eqn:theta_kappa_relationship}) and, moreover, the term $-60 \beta [\cdots] b_n^{\nu-4} x^2 y^2$ can be controlled by $-20 \beta^2 \cosh(\beta B) b_n^{\nu-2} x^2 y^2$ in $\Xi_n(x,y)$, cf. \eqref{eqn:Xi_fixed_parameter_case:after_manipulations}. 

We then combine the three terms in \eqref{eqn:Theta_varying_parameter_case} involving $b_n^{\nu-2} y^2$, $b_n^{\nu-3} x^2 y$ and $b_n^{\nu-4} x^4$ into a quadratic term of the type
\begin{align*}
d_1(\kappa,\theta) b_n^{\nu-4} \left(b_ny + d_2(\kappa,\theta) x^2\right)^2 + d_3(\kappa,\theta) b_n^{\nu-4} x^4, 
\end{align*}
where all the signs of the coefficients $d_i$ are undetermined. Observe that we can bound the size of the first square as follows
\[
\left(b_ny + d_2(\kappa,\theta) x^2\right)^2 \leq 2 \left(b_ny + \beta \tanh(\beta B) x^2 \right)^2+ 2 \left(d_2(\kappa,\theta) - \beta \tanh(\beta B)\right)^2 x^4,
\]
obtaining
\begin{multline*}
d_1(\kappa,\theta) b_n^{\nu-4} \left(b_ny + d_2(\kappa,\theta) x^2\right)^2 + d_3(\kappa,\theta) b_n^{\nu-4} x^4 \\
\leq 2 \left\vert  d_1(\kappa,\theta) \right\vert b_n^{\nu-4} \left(b_ny + \beta \tanh(\beta B) x^2 \right)^2 \\
+ \left[ 2 d_1(\kappa,\theta) \left(d_2(\kappa,\theta) - \beta \tanh(\beta B)\right)^2 + d_3(\kappa,\theta) \right] b_n^{\nu-4} x^4,
\end{multline*}
that can be in turn controlled by 
\[
-120 \cosh(\beta B) b_n^{\nu-2}(b_ny + \beta \tanh(\beta B)x^2)^2 - 2 \beta^4 \cosh(\beta B) (2-5\tanh^2(\beta B)) b_n^{\nu-4} x^6
\]
in $\Xi_n(x,y)$. Therefore, to conclude, there exist suitable positive constants $c_1$ and $c_2$ (independent on $n$ and $\varepsilon$)  for which we have
\[
H_n \psi_n^{\varepsilon,+}(x,y) \leq c_1 + \varepsilon c_2 + 8\left( \frac{\varepsilon}{2} \|\psi'\| + \varepsilon^2 \right),
\]
giving uniformly upper-boundedness in $n$. 
\end{proof}

\subsection{Proof of Theorem~\ref{theorem:moderate_deviations_CW_tri-critical_point_rescaling:2}}

As in the previous section, we first identify the relevant operators for Assumption \ref{assumption:abstract_sequences_extended} from the linear part in the expansion in \eqref{def:varying_parameters_Hamiltonian}. In the setting of Theorem \ref{theorem:moderate_deviations_CW_tri-critical_point_rescaling:2}, $\kappa_n$ and $\theta_n$ are of order $b_n^{-4}$ and therefore, the operators arising from the $\theta$ and $\kappa$ derivatives change.

\smallskip

For $k \in \{1,\dots,5\}$ and appropriate $z \in \{+,-,0,1\}$, the operators $Q_{k,k}^{z} = Q_k^{z}$ are still as defined in \eqref{def:Q1}--\eqref{def:Q5}. The only additional operator of relevance is $Q_{5,1}^0$. It comes from the first coordinate of 
\begin{equation*}
(\partial_{\kappa} + \partial_{\theta}) D^1 \cG_2(0,\tanh(\beta B)) \begin{pmatrix}
x \\ y
\end{pmatrix},
\end{equation*}
and is explicitly given by
\begin{equation}
Q_{5,1}^0 g(x,y)  = \left[ \frac{2}{\cosh(\beta B)}  \left(1 - 2 \beta B \tanh(\beta B)\right) \kappa - 4 \beta \sinh (\beta B)  \theta \right] x g_x(x,y). \label{eqn:Q51_0_arbitrary_direction}
\end{equation}

\begin{lemma}\label{lemma:limiting_drifts:extended_arbitrary_direction_to_tc}
	Let $f \in V$.	If we approach the tri-critical point $(\beta_{\mathrm{tc}}, B_{\mathrm{tc}})$ from an arbitrary direction, we obtain $(Q_{2,2}^- P Q_{2,2}^+  + Q_{3,3}^0) f = 0$ and 
\begin{multline*}
(Q_{5,5}^0 + Q_{5,1}^0 + Q^-_{2,2} P Q_{4,4}^+ + Q_{4,4}^- P Q_{2,2}^+ + Q^-_{2,2} P Q^1_{3,3} P Q_{2,2}^+ ) f(x,y)\\
=  \left[ \frac{2}{3} \left( \sqrt{6} - 2 \sqrt{2} \arccosh \left( \sqrt{\frac{3}{2}} \right)\right)   \kappa - 3 \sqrt{2} \, \theta \right] x f_x(x,y) - \frac{9}{10} \sqrt{\frac{3}{2}} \, x^5 f_x(x,y).
\end{multline*}
\end{lemma}

\begin{proof}[Proof of Theorem \ref{theorem:moderate_deviations_CW_tri-critical_point_rescaling:2}]
The proof follows the proof of Theorems~\ref{theorem:moderate_deviations_CW_critical_curve_rescaling} by using instead the Hamiltonian given by
	\begin{equation*}
	H(x,p) = 2 \sqrt{\frac{2}{3}} p^2 + \left\{ \left[ \frac{2}{3} \left( \sqrt{6} - 2 \sqrt{2} \arccosh \left( \sqrt{\frac{3}{2}} \right)\right)   \kappa - 3 \sqrt{2} \theta \right] x - \frac{9}{10} \sqrt{\frac{3}{2}} x^5\right\}p.
	\end{equation*}
\end{proof}


\appendix

\section{Appendix: path-space large deviations for a projected process} \label{appendix:large_deviations_for_projected_processes}

We turn to the derivation of the large deviation principle. We first introduce our setting.

\begin{assumption} \label{assumption:LDP_assumption}
Assume that, for each $n \in \mathbb{N}^*$, we have a Polish subset $E_n \subseteq \bR^2$ such that for each $x \in \bR^2$ there are $x_n \in E_n$ with $x_n \rightarrow x$.  Let $A_n \subseteq C_b(E_n) \times C_b(E_n)$ and existence and uniqueness holds for the $D_{E_n}(\bR^+)$ martingale problem for $(A_n,\mu)$ for each initial distribution $\mu \in \cP(E_n)$. Letting $\PR_{z}^n \in \cP(D_{E_n}(\bR^+))$ be the solution to $(A_n,\delta_z)$, the mapping $z \mapsto \PR_z^n$ is measurable for the weak topology on $\cP(D_{E_n}(\bR^+))$. Let $Z_n$ be the solution to the martingale problem for $A_n$ and set
\begin{equation*}
H_n f = \frac{1}{r_n} e^{-r_n f}A_n e^{r_n f} \qquad e^{r_n f} \in \cD(A_n),
\end{equation*}
for some sequence of speeds $(r_n)_{n \in \mathbb{N}^*}$, with $\lim_{n \uparrow \infty} r_n = \infty$.
\end{assumption}

Following the strategy of \cite{FK06}, the convergence of Hamiltonians $(H_n)_{n \in \mathbb{N}^*}$ is a major component in the proof of the large deviation principle. We postpone the discussion on how determining a limiting Hamiltonian $H$ due to the difficulties that taking the $n \uparrow \infty$ limit introduces in our particular context. We first focus on exponential tightness, an equally important aspect. 

\subsection{Compact containment condition}

Given the convergence of the Hamiltonians, to have exponential tightness it suffices to establish an exponential compact containment condition.

\begin{definition}
We say that a sequence of processes $(Z_n(t),t \geq 0)$ on $E_n \subseteq \bR^2$ satisfies the exponential compact containment condition at speed $(r_n)_{n \in \mathbb{N}^*}$, with $\lim_{n \uparrow \infty} r_n = \infty$, if for all compact sets $K\subseteq \bR^2$,  constants $a \geq 0$ and times $T > 0$, there is a compact set $K' \subseteq \bR^2$ with the property that
	\begin{equation*}
	\limsup_{n \uparrow \infty} \sup_{z \in K} \frac{1}{r_n} \log \PR\left[Z_n(t) \notin K' \text{ for some } t \leq T \, \middle| \, Z_n(0) = z\right] \leq - a.
	\end{equation*}
\end{definition}

The exponential compact containment condition can be verified by using approximate Lyapunov functions and martingale methods. This is summarized in the following lemma. Note that exponential compact containment can be obtained by taking deterministic initial conditions.

\begin{lemma}[Lemma 4.22 in \cite{FK06}] \label{lemma:compact_containment_FK}
Suppose Assumption \ref{assumption:LDP_assumption} is satisfied. Let $Z_n(t)$ be a solution of the martingale problem for $A_n$ and assume that $(Z_n(0))_{n \in \mathbb{N}^*}$ is exponentially tight with speed $(r_n)_{n \in \mathbb{N}^*}$. Consider the compact set $K = [a,b] \times [c,d]$ and let $G \subseteq \mathbb{R}^2$ be open and such that $[a,b]\times [c,d] \subseteq G$. For each $n$, suppose we have $(f_n,g_n) \in H_n$. Define
	\begin{align*}
	\beta(q,G) &:= \liminf_{n \uparrow \infty} \left( \inf_{(x,y) \in G^c} f_n(x,y) - \sup_{(x,y) \in K} f_n(x,y)\right), \\
	\gamma(G) & := \limsup_{n \uparrow \infty} \sup_{(x,y) \in G} g_n(x,y).
	\end{align*}
	Then
	\begin{multline*} 
	\limsup_{n \uparrow \infty} \frac{1}{r_n} \log \PR\left[Z_n(t) \notin G \text{ for some } t \leq T  \right] \\
	\leq \max \left\{-\beta(q,G) + T \gamma(G), \limsup_{n \uparrow \infty} \PR\left[Z_n(0) \notin [a,b]\times[c,d] \right] \right\}.
	\end{multline*}
\end{lemma}

\subsection{Operator convergence for a projected process}
%
%
%
%
%

In the papers \cite{Kr16b,CoKr17,DFL11}, one of the main steps in proving the large deviation principle was proving directly the existence of an operator $H$ such that $H \subseteq \LIM_n H_n$; in other words by verifying that, for all $(f,g) \in H$, there are $f_n \in H_n$ such that $\LIM_n f_n = f$ and $\LIM_n H_n f_n = g$ (the notion of $\LIM$ is introduced in Definition~\ref{def:LIM}). Here it is hard to follow a similar strategy.

We are dealing with functions 
\[
f_n(x,y) = f(x) + b_n^{-1} f_1(x,y) + b_n^{-2} f_2(x,y) \qquad \mbox{ (for suitably chosen $f_1$ and $f_2$)}
\]
given in a perturbative fashion and satisfying intuitively $f_n \rightarrow f$ and $H_n f_n \rightarrow Hf$ with Hamiltonian $H \subseteq C_b(\bR) \times C_b(\bR)$. In contrast to the setting of \cite{CoKr17}, even if $F_{n,f} \in C_c^{\infty}(\mathbb{R}^2)$, we can not guarantee $\sup_n \| H_n F_{n,\psi}\| < \infty$, implying we do not have $\LIM H_nf_n = Hf$. To circumvent this issue, we relax our definition of limiting operator.

In particular, we will work with two Hamiltonians $H_{\dagger}$ and $H_\ddagger$, that are limiting upper and lower bounds for the sequence of Hamiltonians $H_n$, respectively, and thus serve as natural upper and lower bounds for $H$. This extension allows us to consider unbounded functions in the domain and to argue with inequalities rather than equalities.

\begin{definition}[Definition 2.5 in \cite{FK06}]\label{def:LIM}
For $f_n \in C_b(E_n)$ and $f \in C_b(\bR^2)$, we will write $\LIM f_n = f$ if $\sup_n \vn{f_n} < \infty$ and, for all compact sets $K \subseteq \bR^2$,
\begin{equation*}
\lim_{n \uparrow \infty} \sup_{(x,y) \in K \cap E_n} \left|f_n(x,y) - f(x,y) \right| = 0.
\end{equation*}
\end{definition}

\begin{definition}[Condition 7.11 in \cite{FK06}] \label{definition:subLIM_superLIM}
Suppose that for each $n$ we have an operator $H_{n} \subseteq C_b(E_n) \times C_b(E_n)$. 
Let $(v_n)_{n \in \mathbb{N}^*}$ be a sequence of real numbers such that $v_n \uparrow \infty$.

\begin{enumerate}[(a)]
	\item The \textit{extended sub-limit}, denoted by $ex-\subLIM_n H_{n}$, is defined by the collection $(f,g) \in C_l(\bR^2)\times C_b(\bR)$ for which there exist $(f_n,g_n) \in H_{n}$ such that
	\begin{gather} 
	\LIM f_n \wedge c = f \wedge c, \qquad \forall \, c \in \bR, \label{eqn:convergence_condition_sublim_constants} \\
	\sup_n \frac{1}{v_n} \log \vn{g_n} < \infty, \qquad \sup_{n} \sup_{x \in \bR^2} g_n(x) < \infty, \label{eqn:convergence_condition_sublim_uniform_gn}
	\end{gather}
	and that, for every compact set $K \subseteq \bR^2$ and every sequence $z_n \in K$ satisfying $\lim_n z_n = z$ and $\lim_n f_n(z_n) = f(z) < \infty$, 
	\begin{equation} \label{eqn:sublim_generators_upperbound}
	\limsup_{n \uparrow \infty}g_{n}(z_n) \leq g(z).
	\end{equation}
	\item 
	The \textit{extended super-limit}, denoted by $ex-\superLIM_n H_{n}$, is defined by the collection $(f,g) \in C_u(\bR^2)\times C_b(\bR)$ for which there exist $(f_n,g_n) \in H_{n}$ such that
	\begin{gather} 
	\LIM f_n \vee c = f \vee c, \qquad \forall \, c \in \bR, \label{eqn:convergence_condition_superlim_constants} \\
	\sup_n \frac{1}{v_n} \log \vn{g_n} < \infty, \qquad \inf_{n} \inf_{x \in \bR^2} g_n(x) > - \infty, \label{eqn:convergence_condition_superlim_uniform_gn}
	\end{gather}
	and that, for every compact set $K \subseteq \bR^2$ and every sequence $z_n \in K$ satisfying $ \lim_n z_n = z$ and $\lim_n f_n(z_n) = f(z) > - \infty$,
	\begin{equation}\label{eqn:superlim_generators_lowerbound}
	\liminf_{n \uparrow \infty}g_{n}(z_n) \geq g(z).
	\end{equation}
\end{enumerate}
\end{definition}

For completeness, we also give the definition of the extended limit.

\begin{definition}\label{def:extended_limit}
Suppose that for each $n$ we have an operator $H_{n} \subseteq C_b(E_n) \times C_b(E_n)$. We write $ex-\LIM H_n$ for the set of $(f,g) \in C_b(\bR^2) \times C_b(\bR^2)$ for which there exist $(f_n,g_n) \in H_n$ such that $f = \LIM f_n$ and $g = \LIM g_n$.
\end{definition}

\begin{definition}[Viscosity solutions]
Let $H_\dagger \subseteq C_l(\bR^2) \times C_b(\bR^2)$ and $H_\ddagger \subseteq C_u(\bR^2) \times C_b(\bR^2)$ and let $\lambda > 0$ and $h \in C_b(\bR^2)$. Consider the Hamilton-Jacobi equations
\begin{align}
f - \lambda H_\dagger f & = h, \label{eqn:differential_equation_dagger} \\
f - \lambda H_\ddagger f & = h. \label{eqn:differential_equation_ddagger}
\end{align}
We say that $u$ is a \textit{(viscosity) subsolution} of equation \eqref{eqn:differential_equation_dagger} if $u$ is bounded, upper semi-continuous and if, for every $f \in \cD(H_\dagger)$ such that $\sup_x u(x) - f(x) < \infty$ and every sequence $x_n \in \mathbb{R}^2$ such that
\begin{equation*}
\lim_{n \uparrow \infty} u(x_n) - f(x_n)  = \sup_x u(x) - f(x),
\end{equation*}
we have
\begin{equation*}
\lim_{n \uparrow \infty} u(x_n) - \lambda H_\dagger f(x_n) - h(x_n) \leq 0.
\end{equation*}
We say that $v$ is a \textit{(viscosity) supersolution} of equation \eqref{eqn:differential_equation_ddagger} if $v$ is bounded, lower semi-continuous and if, for every $f \in \cD(H_\ddagger)$ such that $\inf_x v(x) - f(x) > - \infty$ and every sequence $x_n \in \mathbb{R}^2$ such that
\begin{equation*}
\lim_{n \uparrow \infty} v(x_n) - f(x_n)  = \inf_x v(x) - f(x),
\end{equation*}
we have
\begin{equation*}
\lim_{n \uparrow \infty} v(x_n) - \lambda H_\ddagger f(x_n) - h(x_n) \geq 0.
\end{equation*}
We say that $u$ is a \textit{(viscosity) solution} of equations \eqref{eqn:differential_equation_dagger} and \eqref{eqn:differential_equation_ddagger} if it is both a subsolution to \eqref{eqn:differential_equation_dagger} and a supersolution to \eqref{eqn:differential_equation_ddagger}.

We say that \eqref{eqn:differential_equation_dagger} and \eqref{eqn:differential_equation_ddagger} satisfy the \textit{comparison principle} if for every subsolution $u$ to \eqref{eqn:differential_equation_dagger} and supersolution $v$ to \eqref{eqn:differential_equation_ddagger}, we have $u \leq v$.
\end{definition}

Note that the comparison principle implies uniqueness of viscosity solutions. This in turn implies that a new Hamiltonian can be constructed based on the set of viscosity solutions.

\begin{condition} \label{condition:H_limit}
Suppose we are in the setting of Assumption \ref{assumption:LDP_assumption}. Suppose there are operators $H_\dagger \subseteq C_l(\bR^2) \times C_b(\bR^2)$, $H_\ddagger \subseteq C_u(\bR^2) \times C_b(\bR^2)$ and $H \subseteq C_b(\bR) \times C_b(\bR)$ with the following properties:
\begin{enumerate}[(a)]
\item \label{condition:H_limit:item_1} 
$H_\dagger \subseteq ex-\subLIM_n H_n$ and $H_\ddagger \subseteq ex-\superLIM_n H_n$.
\item \label{condition:H_limit:item_2} 
The domain $\cD(H)$ contains $C^\infty_c(\mathbb{R})$ and, for $f \in C_c^\infty(\bR)$, we have $Hf(x) = H(x,\nabla f(x))$.
\item \label{condition:H_limit:item_3}
For all $\lambda > 0$ and $h \in C_b(\bR)$, every subsolution to $f - \lambda H_\dagger f = h$ is a subsolution to $f - \lambda H f = h$ and every supersolution to $f - \lambda H_\ddagger f = h$ is a supersolution to $f - \lambda H f = h$.
\end{enumerate}
\end{condition}

Now we are ready to state the main result of this appendix: the large deviation principle for the projected process. We denote by $\eta_n: E_n \to \mathbb{R}$ the projection map $\eta_n(x,y)=x$.

\begin{theorem}[Large deviation principle] \label{theorem:Abstract_LDP}
Suppose we are in the setting of Assumption \ref{assumption:LDP_assumption} and Condition \ref{condition:H_limit} is satisfied. Suppose that for all $\lambda > 0$ and $h \in C_b(\mathbb{R})$ the comparison principle holds for \mbox{$f - \lambda H f = h$}.   

Let $Z_n(t)$ be the solution to the martingale problem for $A_n$. Suppose that the large deviation principle at speed $(r_n)_{n \in \mathbb{N}^*}$ holds for $\eta_n(Z_n(0))$ on $\bR$ with good rate-function $I_0$. Additionally suppose that the exponential compact containment condition holds at speed $(r_n)_{n \in \mathbb{N}^*}$ for the processes $Z_n(t)$.

Then the large deviation principle holds with speed $(r_n)_{n \in \mathbb{N}^*}$  for $(\eta_n(Z_n(t)))_{n \in \mathbb{N}^*}$ on $D_{\mathbb{R}}(\bR^+)$ with good rate function $I$. Additionally, suppose that the map $p \mapsto H(x,p)$ is convex and differentiable for every $x$ and that the map $(x,p) \mapsto \frac{\dd}{\dd p} H(x,p)$ is continuous. Then the rate function $I$ is given by
\begin{equation*}
I(\gamma) = \begin{cases}
I_0(\gamma(0)) + \int_0^\infty \cL(\gamma(s),\dot{\gamma}(s)) \dd s & \text{if } \gamma \in \cA\cC, \\
\infty & \text{otherwise},
\end{cases}
\end{equation*}
where $\cL : \mathbb{R}^2 \rightarrow \bR$ is defined by $\cL(x,v) = \sup_p \left\{pv - H(x,p)\right\}$.
\end{theorem}

\begin{proof}
The large deviation result follows by \cite[Cor.~8.28]{FK06} with $H_\dagger$ and $H_\ddagger$ as in the present paper and $\bfH_\dagger = \bfH_\ddagger = H$. The verification of the conditions for \cite[Thm.~8.27]{FK06} corresponding to a Hamiltonian of this type have been carried out in e.g. \cite[Sect.~10.3]{FK06} or \cite{CoKr17}.
\end{proof}

\subsection{Relating two sets of Hamiltonians}

For Condition \ref{condition:H_limit}, we need to relate the Hamiltonians $H_\dagger \subseteq C_l(\bR^2) \times C_b(\bR^2)$ and $H_\ddagger \subseteq C_u(\bR^2) \times C_b(\bR^2)$ to $H \subseteq C_b(\bR) \times C_b(\bR)$. 

\begin{definition}
Let $H_\dagger \subseteq C_l(\bR^2) \times C_b(\bR^2)$ and $H_\ddagger \subseteq C_u(\bR^2) \times C_b(\bR^2)$. We say that $\hat{H}_\dagger \subseteq  C_l(\bR^2) \times C_b(\bR^2)$ is a \textit{viscosity sub-extension} of $H_\dagger$ if $H_\dagger \subseteq \hat{H}_\dagger$ and if for every $\lambda >0$ and $h \in C_b(\bR^2)$ a viscosity subsolution to $f-\lambda H_\dagger f = h$ is also a viscosity subsolution to $f - \lambda \hat{H}_\dagger f = h$. Similarly, we define a \textit{viscosity super-extension} $\hat{H}_\ddagger$ of $H_\ddagger$.
\end{definition}

The following lemma allows us to obtain viscosity extensions.

\begin{lemma}[Lemma 7.6 in \cite{FK06}] \label{lemma:extension_results}
Let $H_\dagger \subseteq \hat{H}_\dagger \subseteq C_l(\bR^2) \times C_b(\bR^2)$ and $H_\ddagger \subseteq \hat{H}_\ddagger \subseteq C_u(\bR^2) \times C_b(\bR^2)$.

\smallskip

Suppose that for each $(f,g) \in \hat{H}_\dagger$ there exist $(f_n,g_n) \in H_\dagger$ such that, for every $c,d \in \bR$, we have
\begin{equation*}
\lim_{n \uparrow \infty} \vn{f_n \wedge c - f \wedge c} = 0
\end{equation*}
and 
\begin{equation*}
\limsup_{n \uparrow \infty} \sup_{z :  f(\gamma(z)) \vee f_n(\gamma(z)) \leq c} g_n(z) \vee d - g(z) \vee d \leq 0.
\end{equation*}
Then $\hat{H}_\dagger$ is a sub-extension of $H_\dagger$.

\smallskip

Suppose that for each $(f,g) \in \hat{H}_\ddagger$ there exist $(f_n,g_n) \in H_\ddagger$ such that, for every $c,d \in \bR$, we have
\begin{equation*}
\lim_{n \uparrow \infty} \vn{f_n \vee c - f \vee c} = 0
\end{equation*}
and 
\begin{equation*}
\liminf_{n \uparrow \infty} \inf_{z :  f(\gamma(z)) \wedge f_n(\gamma(z)) \geq c} g_n(z) \wedge d - g(z) \wedge d \geq 0.
\end{equation*}
Then $\hat{H}_\ddagger$ is a super-extension of $H_\ddagger$.
\end{lemma}

\begingroup
\renewcommand{\addcontentsline}[3]{}%

\endgroup

\end{document}